\documentclass[11pt,reqno]{amsart}
\usepackage{amsmath}
\usepackage{amsthm}
\usepackage{graphicx}
\usepackage{mathtools}
\usepackage{amsfonts}
\usepackage{amssymb}
\usepackage{hyperref}
\usepackage{bm}
\usepackage[margin=1 in]{geometry}
\usepackage{enumerate}
\usepackage[normalem]{ulem}
\usepackage{tikz}
\usepackage{tikz-cd}
\usepackage{xcolor}
\usetikzlibrary{matrix,arrows,positioning,automata}
\usepackage{cancel}
\usepackage[numbers]{natbib}
\usepackage{todonotes}

\numberwithin{equation}{section}
\newtheorem{theorem}{Theorem}[section]
\newtheorem{lemma}[theorem]{Lemma}
\newtheorem{proposition}[theorem]{Proposition}
\newtheorem{corollary}[theorem]{Corollary}

\newtheorem{condition}[theorem]{Condition}

\theoremstyle{definition}
\newtheorem{definition}[theorem]{Definition}
\newtheorem{remark}[theorem]{Remark}

\def\E{{\mathbb E}}

\def\R{{\mathbb R}}
\def\N{{\mathbb N}}

\def\P{{\mathcal P}}

\def\X{{\mathcal X}}
\def\Y{{\mathcal Y}}
\def\Q{{\mathcal Q}}

\def\F{{\mathcal F}}

\newcommand\cB{\mathcal B}
\newcommand\cC{\mathcal C}

\newcommand\cF{\mathcal F}
\newcommand\cG{\mathcal G}

\newcommand\cI{\mathcal I}

\newcommand\cL{\mathcal L}
\newcommand\cM{\mathcal M}
\newcommand\cN{\mathcal N}

\newcommand\cP{\mathcal P}

\newcommand\cR{\mathcal R}

\newcommand\cT{\mathcal T}

\newcommand\cW{\mathcal W}
\newcommand\cX{\mathcal X}
\newcommand\cY{\mathcal Y}

\def \E{\mathbb{E}}
\def \F{\mathbb{F}}

\def \N{\mathbb{N}}

\def \P{\mathbb{P}}
\def \Q{\mathbb{Q}}
\def \R{\mathbb{R}}
\def \S{\mathbb{S}}

\def \X{\mathbb{X}}
\def \Y{\mathbb{Y}}

\def \eps{\varepsilon}

\title[$L^2$-graphon SDEs]{Interacting particle systems on sparse $W$-random graphs}
\author{Carla Crucianelli \and Ludovic Tangpi}

\date{\today}
\thanks{Address: cjcrucianelli@princeton.edu and ludovic.tangpi@princeton.edu; Operations Research and Financial Engineering; and Program in Applied and Computational Mathematics, Princeton University. 
Funding: Research Partially supported by NSF CAREER award DMS-2143861 and the AMS Claytor-Gilmer fellowship}

\begin{document}

\begin{abstract}

We consider a general interacting particle system with interactions on a random graph, and study the large population limit of this system. When the sequence of underlying graphs converges to a graphon, we show convergence of the interacting particle system to a so called graphon stochastic differential equation. This is a system of uncountable many SDEs of McKean-Vlasov type driven by a continuum of Brownian motions. We make sense of this equation in a way that retains joint measurability and essentially pairwise independence of the driving Brownian motions of the system by using the framework of Fubini extension. The convergence results are general enough to cover nonlinear interactions, as well as various examples of sparse graphs. 
A crucial idea is to work with unbounded graphons and use the $L^p$ theory of sparse graph convergence.

\end{abstract}

\maketitle


\tableofcontents


\section{Introduction and main results}

Given a family of independent $m$-dimensional Brownian motions $W^1,\dots, W^N$ and a family of Borel probability measures $\gamma^i$ on $\R^n$, consider the interacting particle system
\begin{equation}
\label{eq:intro.part.system}
    Y_t^{i,N} =  Y^{i,N}_0 + \int_0^t b(s,Y_s^{i,N}, m_{\mathbb{Y}_s}^{i,N})ds + \int_0^t \sigma(s,Y_s^{i,N}, m_{\mathbb{Y}_s}^{i,N}) dW_s^{i},\quad 
    Y^{i,N}_0 \sim  \gamma^i, 
\end{equation}
where $\mathbb{Y}:= (Y^{1,N},\dots, Y^{N,N})$, and
\begin{equation}
\label{eq:definition.empirical.measure}
    m_{\mathbb{Y}_s}^{i,N} := \frac{1}{N} \sum_{j=1}^N \frac{\zeta_{ij}^N}{\beta_N} \delta_{Y_s^{j,N}}
\end{equation}
is a positive measure modeling the interaction among the particles $(Y^{i,N})_{i=1,\dots,N}$. $\beta_N$ is a sequence of non-negative real numbers that represent the sparsity parameters of the underlying graph of the system and the matrix $(\zeta^N_{ij})_{1\le i,j\le N}$ represents the adjacency matrix of a simple random graph $\mathcal{G}_N$.
We are chiefly interested in describing the asymptotic behavior of the particle system \eqref{eq:intro.part.system} when the graph $\mathcal{G}_N$ is sparse. 

Interacting particle systems have been extensively studied in the literature and have many different applications for example to study synchronization as in the Kuramoto model \cite{Chiba-Kuramoto-graphs,Oliveira-Reis-interacting-diffusions-on-random-graphs}, the Ising model \cite{spin-systems-on-random-graphs}, neuroscience \cite{Agathe-long-term-stability-Hawke} or games \cite{Carmona-stochastic-graphon-games} (See \cite{porter2015dynamical} for more examples). 
When studying complex systems arising from applications it is typical to have non-symmetric networks, which forces the models to have an underlying graph structure (see \cite{Oliveira-Reis-interacting-diffusions-on-random-graphs,oliveira-interacting-on-sparse-graphs,Chiba-Kuramoto-graphs, spin-systems-on-random-graphs,Medvedev-Nonlinear-heat-on-dense-graph}). 
In many applications, one central question is to study the large population limit.
This requires considering convergence of a sequence of large graphs. 
The classical theory of convergence of graph sequences defines this convergence by duality in terms of combinatorial quantities called \textit{graph parameters} (see \cite{lovazs-limits-dense-graph-seq,Large-networks-and-graph-limits-lovasz} for details); and it is well-known that it is necessary to consider a larger class of objects than finite graphs, these are called \textit{graphons}. 
Formally, a graphon\footnote{To be more precise, instead of graphon we should use the term kernel, as graphon usually reserved for bounded functions. We nevertheless call $g$ a graphon in accordance with the term graphon SDE used in the literature.} is a function
\begin{equation*}
    g: I\times I\to [0,\infty), \quad \text{with}\quad I=[0,1]
\end{equation*}
that is Borel-measurable and symmetric. Intuitively, these generalize the notion of graphs when we have an uncountable number of vertices and $g(x,y)$ would represent the weight of the edge between $x$ and $y$. 
Graphon stochastic differential equations (SDEs) therefore arise as natural (candidate) limit for interacting particle systems on graphs converging to a graphon. 
Importantly, we focus on square-integrable graphons, extending beyond the existing literature which has predominantly addressed the bounded case. We will elaborate on the relevance of this extension below.

By ``graphon SDE'' we mean an equation of the form
\begin{equation}
\label{eq:graphon-SDE}
    dX_t^x = b\bigg(t,X_t^x, \int_I g(x,y)\mathcal{L}(X_t^y) \lambda(dy)\bigg)dt + \sigma\bigg(t,X_t^x, \int_I g(x,y)\mathcal{L}(X_t^y) \lambda(dy)\bigg)dB_t^x, \;\;
    X_0^x \sim  \gamma^x 
\end{equation}
for a family of Brownian motions $(B^x)_{x\in I}$ on some probability space $(\Omega,\cF, \P)$ and where $\cL(X_t^y)$ denotes the law of $X^y_t$.
At this point, one could wonder why we need to work with a continuum of random variables, the reason lies on the fact that it is a typical assumption in economic literature to work with a continuum of agents, see \cite{Sun-2006}. For example \citet{Aumann-market-continuum} shows that when studying a market with perfect competition, as each individual will not be able to influence the prices by itself, continuum models are more suitable. 
This poses important measurability issues that we carefully address in this paper. 
It is a known result that a process $f$ defined on the usual product space between two probability spaces $(I,\mathcal{I},\lambda)$ and $(\Omega, \mathcal{F},\mathbb{P})$, that is jointly measurable and essentially pairwise independent (see Definition \ref{def:e.p.i}) implies that $f_i$ is $\lambda$-almost surely constant~\cite{Sun-2006}. 
Hence, if we attempt to construct the solution of \eqref{eq:graphon-SDE} on $(I\times\Omega, \cI\otimes\cF, \lambda\otimes\P)$ we would need to give up joint measurability!
However, notice that to make rigorous sense of equation \eqref{eq:graphon-SDE}, the mapping $x\mapsto\cL(X^x_t)$ needs to be measurable, the process $X^x$ adapted to the filtration of $B^x$ and $X^x,X^y$ independent for $\lambda$-almost every $y\neq x$. To achieve this, the family of Brownian motions $(B^x)_{x\in I}$ needs to be pairwise independent and measurable.

Different authors have proposed different approaches to this problem. 
One option would be to give up the joint measurability.
For instance \citet{bayraktar-graphon-mean-field-systems} propose a particle system driven by a family of Brownian motions indexed by $x\in I$. 
However, it is not assumed that $x\mapsto B^x$ is measurable. 
It is argued directly that $x\mapsto \cL(X^x)$ is measurable (but not necessarily $x\mapsto X^x$). 
The reason why the measurability of the laws is enough for their purposes is that the particle system considered in \cite{bayraktar-graphon-mean-field-systems} involves integrals with respect to $\cL(X^x)$ and not with respect to $X^x$.  
The main drawback of the approach taken in \cite{bayraktar-graphon-mean-field-systems} is that more general dependencies on the coefficients cannot be considered, for example dependencies such as $\int_I g(x,y)X^{y}_t \lambda(dy)$ which naturally appear in economics applications, see e.g. \citet{Carmona-stochastic-graphon-games,Sun-2006}. 
Our approach will allow to consider dependencies of this form using the Exact Law of Large Numbers. 
Moreover, anticipating our discussion of the $N\to\infty$ convergence, let us already mention that an easy corollary of our main convergence result will be a weak law of large numbers (Corollary \ref{coro:WLLN}) or an Exact law of large numbers~(\cite{Sun-2006,Law-large-numbers-Uhlig-1996,Judd-lln-1985}) both for which measurability of $x\mapsto X^x$ is key.
We need to consider an appropriate probability space on which we will construct solutions of the SDE \eqref{eq:graphon-SDE}. We will use Fubini extensions, as was done by \citet{aurell-carmona-2021}, to be able to keep both measurability and independence. To this end, let us beginning by describing the probabilistic setting on which the graphon SDE is studied.

\subsection{Probabilistic setting and general graphon SDEs}
Inspired by \cite{aurell-carmona-2021}, we will work on a Fubini extension of a given probability space.
Let us describe such an extension and state a crucial property that justifies it as a natural probability space to work on.
\label{sec:probabilistic-setting}
\begin{definition}~\cite{Sun-2006}
\label{def:e.p.i}
    Let $(I,\mathcal{I}',\lambda')$ and $(\Omega', \mathcal{F}',\mathbb{P}')$ be two probability spaces. 
    A process  $f: I\times \Omega' \rightarrow E$, where $E$ is a separable metric space is essentially pairwise independent if for $\lambda'$-almost every $x\in I$ the random variables $f_x$ and $f_y$ are independent for $\lambda'$-almost every $y\in I$. 
\end{definition}

\begin{definition}
Let $(\Omega', \mathcal{F}',\mathbb{P}')$ and $(I,\mathcal{I}',\lambda')$ be two probability spaces. 
A probability space $(I \times \Omega' ,\mathcal{W}, \mathbb{Q})$ extending the product space $(I \times \Omega' , \mathcal{I}\otimes \mathcal{F}', \lambda'\otimes \mathbb{P}')$ is a Fubini extension if for any real valued $\mathbb{Q}$-integrable function $f$ on $(I \times \Omega' ,\mathcal{W}, \mathbb{Q})$, 
\begin{itemize}
    \item $f_x: \omega \mapsto f(x,\omega)$ is integrable on $(\Omega', \mathcal{F}',\mathbb{P}')$ for $\lambda'$-a.e. $x\in I$ and $f_\omega: x\mapsto f(x,\omega)$ is integrable on $(I,\mathcal{I}', \lambda')$ for $\mathbb{P}'$-a.e. $\omega \in \Omega'$.
    \item $\int_{\Omega'} f_x(\omega)\mathbb{P}'(d\omega)$ and $\int_I f_\omega(x) \lambda'(dx)$ are integrable respectively on $(I,\mathcal{I}',\lambda')$ and $(\Omega', \mathcal{F}',\mathbb{P}')$ with:
    \begin{equation*}
        \int_{I \times \Omega'} f(x,\omega) \mathbb{Q}(dx,d\omega) = \int_I \left(\int_{\Omega'} f_x(\omega)\mathbb{P}'(d\omega)\right) \lambda'(dx) = \int_{\Omega'} \left(\int_I f_\omega(x)d\lambda'(x)\right) \mathbb{P}'(d\omega).
    \end{equation*}
    \end{itemize}
    We denote $ (I \times \Omega', \mathcal{I}' \boxtimes \mathcal{F}',\lambda' \boxtimes \mathbb{P}')= (I \times\Omega' ,\mathcal{W}, \mathbb{Q})$ and by $\mathbb{E}^{\boxtimes}$ the expectation with respect to $\lambda'\boxtimes \mathbb{P}'$. 
\end{definition}
We will colloquially refer to the Fubini extension as the \textit{box product} of the two spaces.
Now, let $T>0$ be a fixed finite time horizon and denote by $\mathcal{C}=C([0,T],\mathbb{R}^n)$ the space of continuous functions from $[0,T]$ to $\R^n$ equipped with the uniform norm $\left|\omega\right|_{\mathcal{C}} = \sup_{t\in [0,T]} |\omega(t)|$.
\begin{remark}
    As it is remarked in \cite{aurell-carmona-2021}, we can extend the definition to maps $f:I\times \Omega' \rightarrow \mathcal{C}$. 
    The only difference is that in this case we should interpret measurability as strong measurability and the integral as a Bochner integral.
\end{remark}

We are now ready to state the existence of the probability space that will carry the family of uncountable essentially pairwise independent Brownian motions $(B^x)_{x\in I}$ 
(See also \cite{aurell-carmona-2021, Sun-2006}). 
The construction follows from the following Theorem of Sun and Zhang~\cite{sun2008individual}, where $\cP(E)$ denotes the set of probability measures on $E$.

\begin{theorem}\label{thm:construction-spaces}~\cite{sun2008individual} 
    Let $I'$ be the unit interval and $E$ a Polish space. There exists a probability space $(I',\mathcal{I'},\lambda')$ that is an extension of the Lebesgue measure on $I'$, a probability space $(\Omega', \mathcal{F}',\mathbb{P}')$ and a Fubini extension $(I'\times\Omega', \mathcal{I'}\boxtimes \mathcal{F'},\lambda' \boxtimes \mathbb{P}')$ such that if $\varphi: (I',\mathcal{I}',\lambda') \rightarrow \mathcal{P}(E)$ is measurable there is a $\mathcal{I'}\boxtimes \mathcal{F}'$-measurable process $Y: I'\times \Omega' \rightarrow E$ such that the random variables $Y^x$ are essentially pairwise independent and $\mathbb{P}'\circ (Y^x)^{-1} = \mathcal{L}(Y^x) = \varphi(x)$ for $\lambda'$-almost every $x \in I'$.  
\end{theorem}

Consider the space $\mathcal{C}\times \mathbb{R}^n$ equipped with the product $\sigma$-algebra. 
Then, we can apply Theorem \ref{thm:construction-spaces} for this space so it holds that there exists: 
    $(i)$ a probability measure that extends the Lebesgue measure on the unit interval $(I,\mathcal{I},\lambda)$, 
    $(ii)$ a probability space  $(\Omega, \mathcal{F},\mathbb{P})$ and 
    $(iii)$ a Fubini extension $(I \times \Omega, \mathcal{I}\boxtimes \mathcal{F},\lambda \boxtimes \mathbb{P})$ satisfying the property of Theorem \ref{thm:construction-spaces}.

Let the initial condition of the SDE \eqref{eq:graphon-SDE} be a family of probability measures $\{\gamma^x\}_{x\in I}$ denoting the laws of $(X^x_0)_{x\in I}$, and let $\Q$ denote the Wiener measure on $\mathcal{C}$. 
Assume:
\begin{condition} \label{cond:initial-conditions}
    \begin{enumerate}
        \item $ (I,\mathcal{I},\lambda) \ni x \mapsto \mathbb{Q} \otimes \gamma^x \in \mathcal{P}(\mathcal{C}\times \mathbb{R}^n)$ is measurable.
        \item For all $x\in I$, $\gamma^x\in \mathcal{P}_2(\mathbb{R}^n)$, and if $\xi^x\sim \gamma^x$, then 
        \begin{equation*}
            \int_I \mathbb{E}\left[|\xi^x|^2\right]\lambda(dx)<\infty.
        \end{equation*}

    \end{enumerate}
\end{condition}

Under Condition \ref{cond:initial-conditions}, consider the measurable map $\varphi: x \in (I,\mathcal{I},\lambda) \mapsto \mathbb{Q} \otimes \gamma^x \in \mathcal{P}(\mathcal{C}\times \mathbb{R}^n)$. Then, there exists a measurable function $\Tilde{B}: (I \times \Omega, \mathcal{I}\boxtimes \mathcal{F},\lambda\boxtimes \mathbb{P}) \rightarrow \mathcal{C}\times \mathbb{R}^n$ such that, if we denote $\Tilde{B}^x=(B^x,\xi^x)$, we have that $\{\Tilde{B}^x\}_{x\in I}$ are essentially pairwise independent, $\mathcal{L}(\xi^x) = \gamma^x$ and $\mathcal{L}(B^x) = \mathbb{Q}$ for every $x\in I$. 
Finally we define $\mathbb{F}^x$ to be the $\P$-competed filtration generated by $\Tilde{B}^x$, that is, the $\mathbb{P}$-completion of $\sigma(\{B_s^x, s\leq t\}, \xi^x)$.
We thus summarize the notation and conventions that we will adopt as follows: 
\begin{itemize}
    \item $(I,\mathcal{I}, \lambda)$ where $\mathcal{I}$ is an extension of the Borel $\sigma$-algebra of $I$ and $\lambda$ is an extension of the Lebesgue measure on the interval. $(\Omega, \cF, \P)$ is a sample space and we can construct the Fubini extension of these two spaces $(I \times \Omega, \cI\boxtimes \cF, \lambda \boxtimes \P)$. We will use the convention that probability measures are complete\footnote{ This is the same convention that \citet{Sun-2006} uses}.
    \item The map $ (I\times \Omega, \mathcal{I}\boxtimes \mathcal{F}, \lambda\boxtimes \mathbb{P})\ni (x,\omega) \mapsto \Tilde{B}^x(\omega) \in \mathcal{C}\times \mathbb{R}^n$ is measurable.
    Thus by coordinate-wise projection, $ (I\times \Omega, \mathcal{I}\boxtimes \mathcal{F}, \lambda\boxtimes \mathbb{P}) \ni (x,\omega) \mapsto B^x(\omega) \in \mathcal{C}$ and $ (I\times \Omega, \mathcal{I}\boxtimes \mathcal{F}, \lambda\boxtimes \mathbb{P})\ni (x,\omega) \mapsto \xi^x(\omega) \in  \mathbb{R}^n$ are measurable. 
    \item $\{(B^x,\xi^x)\}_{x\in I}$ are essentially pairwise independent.
\end{itemize}
We denote by $L^2_\boxtimes(I\times \Omega, \mathcal{C})$ the space of square integrable processes with respect to the box product defined above. 
That is,
\begin{equation*}
\begin{split}
    L^2_{\boxtimes} (I \times \Omega , \mathcal{C}) = \bigg\{X:I\times \Omega \rightarrow \mathcal{C} : \text{$\mathcal{I}\boxtimes \mathcal{F}$-$\mathcal{B}(\mathcal{C})$-measurable  and $\int_{I\times\Omega } \|X^x(\omega)\|_\mathcal{C}^2  \lambda \boxtimes \mathbb{P}(dx,d\omega)<\infty$ }\bigg\}.
\end{split}
\end{equation*}

As usual, we identify elements of $L^2_{\boxtimes}(I\times \Omega,\cC)$ that are equal $\lambda\boxtimes \P$-a.s.
\begin{remark}
    The space $L^2_\boxtimes(I\times\Omega, \cC)$ should be interpreted as the Bochner space and the correct measurability notion is $\lambda\boxtimes\P$-strongly measurable. However, as $\cC$ is a separable space, the notions of measurable and strongly measurable are equivalent (\cite[Corollary 1.1.10]{Analysis-in-Banach}) and using \cite[Proposition 1.1.16]{Analysis-in-Banach} we get that in this case we can use the usual measurability notion. 
\end{remark}

\subsection{Main results of the paper}

With these measure-theoretic preliminaries out of the way, let us come back to the subject of the present paper and describe our main results.
\subsubsection{Well-posedness}
We start by well-posedness of the graphon SDE.
\begin{definition}
\label{def:solution-sde}
    We say that a family of stochastic processes $(X^x)_{x\in I}$ is a (strong) solution of the SDE $\eqref{eq:graphon-SDE}$ if:
    \begin{enumerate}
        \item $X\in L^2_\boxtimes(I\times\Omega,\mathcal{C})$.   
        \item $X^x$ is $\mathbb{F}^x$-adapted for $\lambda$-almost every $x\in I$. 
        \item For $\lambda$-almost every $x\in I$, $X^x$ satisfies the equation \eqref{eq:graphon-SDE} $\P$-a.s. 
    \end{enumerate}
\end{definition}
For any solution $(X^x)_{x\in I}$ of the graphon SDE, the integral in \eqref{eq:graphon-SDE} is understood as a non-negative measure, see Lemma \ref{lemma:nu-defines-a-measure}.
Denote $\cM_+(\R^n)$ to set of non-negative measures and $d_{BL}$ to the bounded Lipschitz distance both defined in Section \ref{sec:notation}. Now, consider the following conditions: 

\begin{condition} \label{cond:coefficients}
The coefficients $b : [0,T] \times \mathbb{R}^n \times \mathcal{M}_+(\mathbb{R}^n) \rightarrow \mathbb{R}^n$ and $\sigma:  [0,T] \times \mathbb{R}^n \times \mathcal{M}_+(\mathbb{R}^n) \rightarrow \mathbb{R}^{n\times m} $satisfy\footnote{ We will denote with $|\cdot|$ to both the Euclidean norm in $\R^n$ and the Euclidean norm on $\R^{n\times m}$}:
\begin{enumerate} 
    \item There exists a constant $C>0$ such that for $t\in [0,T]$, $x,y\in \mathbb{R}^n$, $\mu,\nu \in \mathcal{M}_+(\mathbb{R}^n)$:
    \begin{equation}
    \label{eq:b-lipschitz-condition}
        |b(t,x,\mu) - b(t,y,\nu)| + |\sigma(t,x,\mu) - \sigma(t,y,\nu)| \leq C \left( |x-y| + d_{\mathrm{BL}}(\mu, \nu)\right),
    \end{equation}
    \item for every $x\in \mathbb{R}^n$, $\nu \in \mathcal{M}_+(\mathbb{R}^n)$ the maps $t \mapsto b(t,x,\nu)$ and $t \mapsto \sigma(t,x,\nu)$ are continuous. 
    \item the functions $b$ and $\sigma$ grow at most linearly. Specifically, there exists a constant $C>0$ such that for every $t\in [0,T]$, $x\in \mathbb{R}^n$, $\nu\in \mathcal{M}_+(\mathbb{R}^n)$:
    \begin{equation}
    \label{eq:b-growth-condition}
        |b(t,x,\nu)| + |\sigma(t,x,\nu)| \leq C \Big( 1 + |x| + \|\nu\|_{\mathrm{BL}} \Big).
    \end{equation}
\end{enumerate}
\end{condition}

Our first contribution is the following well-posedness result.

\begin{theorem}
\label{thm:existence-uniqueness}
    Let $g$ be a square integrable graphon. 
    If Conditions \ref{cond:initial-conditions} and \ref{cond:coefficients} are satisfied, then there exists a unique solution to the graphon SDE $\eqref{eq:graphon-SDE}$.
\end{theorem}
In the above statement, by ``unique'' we mean that for any two solutions $X$ and $\bar X$ it holds $X^x_t = \bar X^x_t$ $\lambda\boxtimes \P$-a.s. for all $t\in [0,T]$.
To prove Theorem \ref{thm:existence-uniqueness} we will adapt the classical proof of the existence and uniqueness of strong solutions of the McKean-Vlasov SDE. 
The main difficulties and differences with the classical proof are of technical nature.
In fact, we need to deal with subtle measurability issues related to the measurability with respect to the $\sigma$-albegra of the Fubini extension. 
For instance, in the current framework it is not direct that the stochastic integral retains the desired measurability when both the driving Brownian motion and the integrand depend on the parameter $x$.

Measurability of SDEs depending on a parameter has been studied in the literature, refer for instance to or \citet{Skorokhod-sde-depending-parameter,Stricker1978CalculSD}. Notably, \citet{Stricker1978CalculSD} consider measurability of SDE solutions on the product space $I\times \Omega$ when the driving noise depends measurably on the parameter. This result is not strong enough for our purpose as we need to show measurability with respect to the box product.

\begin{remark}
    Even though existence is proved in $L^2_\boxtimes(I\times \Omega,\cC)$, it is not hard to see that if the integral $\nu^x:=\int_I g(x,y)\mathcal{L}(X_t^y) \lambda(dy)$ is defined for all $x$, then $X^x$ satisfies equation \eqref{eq:graphon-SDE} for every $x\in I$.
    But in the general unbounded graphon case, the integral $\nu^x$ is defined $\lambda$-almost surely.
    In the following we will choose the representative of the solution of the SDE \eqref{eq:graphon-SDE} to be such that if $\nu^x$ is a finite measure then $X^x$ satisfies the SDE \eqref{eq:graphon-SDE} for $x$.
\end{remark}

Theorem \ref{thm:existence-uniqueness} provides well-posedness of the graphon SDE \eqref{eq:graphon-SDE} in the space $L^2_\boxtimes(I\times \Omega,\cC)$ and with unbounded graphons.
Only two papers have addressed this issue: \cite{aurell-carmona-2021} shows existence and uniqueness in the linear case
\begin{equation*}
    dX^x_t = \Big(a(x)X^x_t + c(x)\int_I g(x,y)X^y_t\lambda(dy)\Big)dt + dB_t^x, \quad X^x_0 = \xi^x.
\end{equation*}
This equation falls within our setting thanks to the exact law of large numbers suggesting that $\int_I g(x,y)X^y_t\lambda(dy) = \int_Ig(x,y)\E[X^y_t]\lambda(dy)$.
More recently, the non-linear case was studied by \citet{coppini2024nonlineargraphonmeanfieldsystems}, who consider the SDE
\begin{equation}
\label{eq:SDE-nonlinear}
    dX^x_t = b\bigg(X^x_t,  \int_I\frac{g(x,y)}{\int_I g(x,y)dy}\cL(X^{y}_t)(dz)dy\bigg)dt  + \sigma \bigg(X^x_t, \int_I\frac{g(x,y)}{\int_I g(x,y)dy}\cL(X^{y}_t)(dz)dy\bigg)dB^x_t.
\end{equation}
Specifically, these authors show existence and uniqueness of \eqref{eq:SDE-nonlinear} under Lipschitz and linear growth conditions on the coefficients. 
The SDE \eqref{eq:SDE-nonlinear} has three main differences with our proposed graphon SDE \eqref{eq:graphon-SDE}: First, we allow dependence on the time parameter which forces us to add Condition \ref{cond:coefficients}$(2)$.
Secondly, we allow unbounded graphons, whereas \cite{coppini2024nonlineargraphonmeanfieldsystems} focuses on graphons taking values on the unit interval. 
Finally, in equation \eqref{eq:SDE-nonlinear} the last argument of the coefficients is normalized, hence making it a probability measure.
This allows to work with the Wasserstein metric. 
In contrast, the interaction term in \eqref{eq:graphon-SDE} only a positive measure.
Moreover, \cite{coppini2024nonlineargraphonmeanfieldsystems} assumes 
\begin{equation}
\label{eq:Assum.Coppini}
    \|g\|_1>0 \quad \text{and}\quad \int_I \|g(u,\cdot)\|_1^{-1}du < \infty.
\end{equation}
While the condition $\|g\|_1>0$ can be argued to be natural from graph-theoretic perspective, the second condition unfortunately excludes many natural examples as the uniform attachment graphon $g(x,y) = 1 - \max(x,y)$ or the graphon $g(x,y) = xy$.

\subsubsection{Propagation of chaos}

The main objective of this work is to show convergence of the particle system \eqref{eq:intro.part.system} to the graphon SDE \eqref{eq:graphon-SDE}.
To this end, we need to understand the convergence of the (random) graph sequence $\cG_N$ with adjacency matrix $\zeta^N_{ij}$ to the graphon $g$.  

The theory for convergence of sequences of graphs is well understood in the case of dense graphs\footnote{By a dense graph we mean graphs with $N$ vertices and $\Omega(N^2)$ edges, see \cite{lovazs-limits-dense-graph-seq}.}. 
Classical references on the theory include works by \citet{lovazs-limits-dense-graph-seq} and \citet{Large-networks-and-graph-limits-lovasz}, see also the references therein. 
However, most examples of graphs arising in applications (e.g. in statistical physics and social sciences) are sparse.
A prominent example is the power law graph which is a graph on $N$ vertices with the probability of having an edge between $i,j$ being $1\wedge \frac{N^{b}}{(ij)^a}$ where $a\in (0,1)$ and $b\in (0,2a)$. Unfortunately, the convergence of sequences of sparse graphs is not satisfactory if we use the dense graph limit framework. 
This is because under the classical dense theory sequences of sparse graphs converge to zero \cite{lovazs-limits-dense-graph-seq}. 
On the other hand, the theory of sparse graph limits (known as local weak convergence) introduced by \citet{Benj-Schr01,Aldous-Steele04} focuses on ``very sparse graphs'', as it covers graphs with $O(N)$ edges, where $N$ is the number of vertices. 
 We also refer to \citet{Lacker-local-weak-convergence-sparse-networks} for result along these lines for (dynamic) particle systems.

An important step in bridging the gap between the above two extremes is the works by \citet{Borgs2014AnT,Borgs-Chayes.etal18} who develop an $L^p$-theory of convergence for sparse graphs (not including the bounded average degree regime).
One central idea of the approach is to re-scale the sequences of graphs to recover a non trivial limit.
Naturally, both the re-scaled graph and the limiting graphon are unbounded.
(We will discuss the details of this construction in Section \ref{sec:graph-theoretic-formulation}.)
This makes it essential to study SDEs with unbounded graphons in order to cover the realm of sparse network interactions.
We remark at this point that the $L^p$ theory described in \cite{Borgs2014AnT} as well as the setting of the present work encompasses the dense/bounded case.
Before presenting the general conditions we impose on the graph sequence $\cG_N$, let us observe that
the most natural class of models included in our framework is that of \emph{$W$-random graphs} which we describe for the reader's convenience.

\textbf{Sequence of sparse $W$-random graphs:} 
Let $(\beta_N)_{N\in \N}$ be a fixed sequence of elements of $(0,1]$ , let $x_1,\dots,x_N$ form a partition of $I$ and let $g$ be a given (possibly unbounded) graphon.
A $W$-random graph is the graph $\cG_N(N, g,\beta_N )$ with $N$ nodes such that $(i,j)$ (with $i\neq j$) is an edge with probability $\beta_Ng(x_i,x_j)\wedge 1$, with no loops.
Moreover, the points $x_1,\dots,x_N$ can be chosen random. See e.g \cite[Section 10,1]{Large-networks-and-graph-limits-lovasz} or \cite[Section 2.7]{Borgs2014AnT} for details.
The $W$-random graph $\cG_N(N,g,\beta_N)$ is a sparse graph arising from the graphon $g$.
If $\beta_N\to0$ and $N\beta_N\to\infty$ as $N$ goes to infinity, then it converges to the $g$ in the appropriate\footnote{Namely, the if $g_N$ is the graphon associated to $\cG_N(N,g,\beta_N)$, then $g_N/\beta_N$ converges to $g$ in the cut metric, see \cite[Theorem 2.14]{Borgs2014AnT}.} sense.
Studying $W$-random graphs should allow to cover a wide class of sparse graphs.
In fact, \cite[Proposition 2.16]{Borgs2014AnT} shows that every sequence of sparse (simple) graphs converging to $g$ is close (in the cut metric) to $\cG(N,g,\beta_N)$.
Note in passing that the Erdös-R\'enyi graph $G(N,\beta_N)$ is recovered by taking $g\equiv 1$ and the power law graph corresponds to $g(x,y) = (xy)^{-a}$ and $\beta_N = N^{b-2a}$, with $b\in (2a-1, 2a)$.

Let us now address the issue of convergence of the particle system \eqref{eq:intro.part.system} to the graphon SDE \eqref{eq:graphon-SDE}.
 We will do so under different sets of assumptions that we detail in the following. 

\begin{condition}
\label{cond:graphon-cond-1}
    For every $N\in \N$ consider the sequence of points $x_i = \frac{i-1}{N}$ for $1\leq i\leq N+1$\footnote{Notice that we want to choose $N$ intervals of length $1/N$ and this implies we will need to define $N+1$ points. However, in the definition of the particle system we consider $N$ points and hence we will sum $i$ from $1$ to $N$.}. Assume:
    \begin{enumerate}
        \item For every $1\leq i,j\leq N$, $x\in I$, $\zeta_{ij}^N$ and $\Tilde{B}^x$ are independent.
        \item There exists a sequence of graphons $(g_N)_{N\in \mathbb{N}}$ that are constant over the intervals $\{[x_j,x_{j+1})\times [x_k,x_{k+1})\}_{j,k}$ \footnote{Notice that these points depend on $N$ but we will often omit this dependence as in general it will be clear from the context.}. 
        There also exists a sequence of real numbers $(\beta_N)_{N\in \mathbb{N}}$, such that $\beta_N\in (0,1]$, $\beta_Ng_N(x_i,x_j)\leq 1$ for every $N\in \N, 1\leq i,j\leq N$ and $N\beta_N \rightarrow \infty$ as $N\rightarrow \infty$.
        \item $g_N\rightarrow g$ in the $2$-norm. That is, $\|g_N-g\|_2^2 \rightarrow 0$ as $N\rightarrow \infty$. 
        \item For every $N$, the random variables $(\zeta^N_{ij})_{1\leq i,j\leq N}$ are $\cF_0$-measurable, independent and with distribution $\zeta^N_{ij}\sim \mathrm{Bernoulli}(g_N(x_i,x_j)\beta_N)$. 
    \end{enumerate}
\end{condition}

In Condition \ref{cond:graphon-cond-1} we need to have the condition $\beta_Ng_N(x_i,x_j)\leq 1$ or $\beta_Ng(x_i,x_j)\leq 1$ for $\zeta_{ij}^N$ to be well defined. In the dense case we can take for example $\beta_N = 1$ and $\|g\|_{\infty}\leq 1$ or $\|g_N\|_\infty\leq 1$. However, in this case the assumption is reasonable as it only involves a normalization of the graphons involved. In this work we have in mind the sparse graph case $\beta_N\to0$. To this end, we need to extra condition $N\beta_N \rightarrow \infty$ to be able to contextualize our work in the framework of the $L^p$ theory of convergence of ($W$-random) graphs, see \cite[Theorem 2.14]{Borgs2014AnT}. 

\begin{remark}
    The parameter $\beta_N$ represents the proportion of neighbors a node has compared to the total number of nodes. 
    Hence, the condition $N\beta_N\rightarrow \infty$ as $N\rightarrow \infty$ means that the degree of each node diverges. 
    The behavior of $N\beta_N$ is a relevant parameter as first observed by \citet{Oliveira-Reis-interacting-diffusions-on-random-graphs} where it is shown that $N\beta_N \rightarrow \infty$ is the optimal condition in a Kuramoto model on an Erdös-R\'enyi random graph for the dense and the sparse analogue to have the same hydrodynamic limit. 
    Moreover, it is proved by \citet{oliveira-interacting-on-sparse-graphs} that the condition $N\beta_N\rightarrow c \in \R_+$ changes the behavior of the system. We also refer to \citet{Lacker-local-weak-convergence-sparse-networks} for an in-dept study of the regime $N\beta_N\to c$ where it is showed that the particle system will converge in the weak local sense to a Galton-Watson tree, and the associated particle system will not become independent in the large population limit as the interactions remain strong. Hence, the condition $N\beta_N\rightarrow \infty$ adopted in this work is sharp if we want to get a propagation of chaos type of result.  
\end{remark}

For notational simplicity, let us denote the particle system under consideration $(Y^{x_i,N})_{i=1,\dots,N}$ introduced in \eqref{eq:intro.part.system} and the graphon system $(X^x)_{x\in I}$ introduced in \eqref{eq:graphon-SDE} respectively by 
 $$Y^{i,N}: =Y^{x_i,N}\quad \text{and}\quad X^{x,g}:= X^x.$$
We further assume that $Y^{i,N}_0 = \xi^{i}$. The form of the interaction term $m^{i,N}_{\Y}$ of the particle system \eqref{eq:part.system.B} is rather natural in view of the work of \cite{Borgs2014AnT} on sparse networks.
First of all, the mean field regime is recovered by choosing $\beta_N=1=g_N$, in which case $\zeta_{ij}^N=1$ and $m^{i,N}_{\Y} = m^{j,N}_{\Y} = \sum_{k=1}^N\delta_{Y^{k,N}}/N$ for all $i,j$.
In the network regime to allow sparse graphs we allow interaction only among neighbors; hence, for a particle $i$ with neighbors set $N(i)$ with cardinality $|N(i)|$, the interaction parameter becomes
    \begin{equation*}
        m^{i,N}_{\Y} = \frac{1}{|N(i)|} \sum_{j\in N(i)}\delta_{Y^{j,N}} =  \frac{1}{M} \sum_{j=1}^N\zeta^N_{ij}\delta_{Y^{j,N}} =  \frac{1}{N} \sum_{j=1}^N\frac{\zeta^N_{ij}}{|N(i)|/N}\delta_{Y^{j,N}} ,
    \end{equation*}
    for a matrix $\zeta^N_{ij}$ which is one when $j$ is a neighbor and zero otherwise.
    Thus assuming (the degree)  $\beta_N=|N(i)|/N$ to be constant leads to our interaction model and naturally requires us to work with re-scaled, unbounded graphons. A similar re-scaled graph model was notably considered by \citet{bayraktar-graphon-mean-field-systems} in the case of \textit{linear dependence} on the graphon and assuming the limiting graphon to be \emph{bounded}; and by \citet{Oliveira-Reis-interacting-diffusions-on-random-graphs} who consider a Kuramoto model on a sparse Erdös-Renyi graph. Most of the works cited in the previous section and in subsection \ref{sec.literature} focus on linear dependence on the graph terms.  Extending these results to the nonlinear case is a natural mathematical question also relevant in modeling application, see e.g. \citet{Bhamidi-weakly-interacting-part-syst-inhomogeneoud-random-graphs}.

In addition to Conditions \ref{cond:graphon-cond-1} and \ref{cond:graphon-cond-2} relating the finite network to the limiting graphon, we further need the following regularity and integrability conditions. Condition \ref{cond:almost-continuity} is adapted from \cite{bayraktar-graphon-mean-field-systems} and is clearly satisfied if $g$ is continuous or if it is piecewise constant.
\begin{condition}
\label{cond:almost-continuity}
    For $\lambda$-almost every $x\in I$, there exists a set $I_x$ such that $\lambda(I_x) = 0$ and $g$ is continuous in $(x,y)$ for every $y\in I_x^c$. 
\end{condition}
\begin{condition}
\label{cond:bound-initial-cond}
     There exists a constant $C>0$ independent of $N$ such that
     \begin{equation*}
         \sup_{ i\ge 1}\E[|\xi^i|^{n+1}] \leq C.
     \end{equation*}
\end{condition}

By a synchronous coupling argument, we will consider the particle system $\X = (X^{i,N},\dots, X^{N,N})$ solving

\begin{equation}
\label{eq:part.system.B}
    X_t^{i,N} =  \xi^{i} + \int_0^t b(s,X_s^{i,N}, m_{\mathbb{X}_s}^{i,N})ds + \int_0^t \sigma(s,X_s^{i,N}, m_{\mathbb{X}_s}^{i,N}) dB_s^{i}, \quad
    X^{i,N}_0 =  \xi^i. \\
\end{equation}
Notice that the system \eqref{eq:part.system.B} is the same system as \eqref{eq:intro.part.system} but driven by the Brownian motion $B^i:= B^{x_i}$ and with initial conditions given by $\xi^i \sim \gamma^i$.  
Hence, $\cL(X^i_t) = \cL(Y^i_t)$ for every $1\leq i\leq N$ and every $t\in [0,T]$.
The important point to keep in mind here is that the particle system \eqref{eq:intro.part.system} is driven by independent Brownian motions whereas the system \eqref{eq:part.system.B} is driven by e.p.i Brownian motions. 

We will assume that the Brownian motions $B^{x_i}$ and $B^{x_j}$ are pairwise independent for every pair $i,j$.
This does not restrict the generality because by uniqueness of solutions we can couple the systems with independent Brownian motions and retain the same laws. Both particle systems are well-posed as showed in Proposition \ref{prop:existence-part-system}.

We present two main convergence results. The first one is an convergence in the $L^1$ sense in the case of a bounded graphon:
\begin{theorem} 
\label{thm:convergence-bounded-graphon-L1}
    Assume Conditions \ref{cond:initial-conditions}, 
    \ref{cond:coefficients}, \ref{cond:graphon-cond-1}, \ref{cond:almost-continuity}, \ref{cond:bound-initial-cond}. If $g$ and the sequence $(g_N)_{N\in \N}$ are uniformly bounded and $\sigma$ does not depend on the measure argument. Then, 
    \begin{equation}
    \label{eq:first.main.conv}
         \frac{1}{N}\sum_{i=1}^N\E\Big[ \sup_{t\in [0,T]}\big|X^{i,N}_{t}-X^{i,g}_t\big| \Big] +\frac1N\sum_{i=1}^N\E\Big[d_{\mathrm{BL}}\Big(m^{i,N}_{\X_t}, \int_I g(x_i,y)\mathcal{L}(X_t^y) \lambda(dy)\Big) \Big] \xrightarrow{N\rightarrow \infty} 0 \,\,\text{for all } t\in [0,T]. 
    \end{equation}
\end{theorem}

One immediate Corollary of Theorem \ref{thm:convergence-bounded-graphon-L1} is obtained by letting the graph structure of the adjacency matrix $\zeta_{ij}^N$ be given by a $W$-random graph. This is the content of Corollary \ref{coro:bounded-L1-coro-Wrand}.

\begin{condition}
\label{cond:graphon-cond-2}
    For every $N\in \N$ consider the sequence of points $x_i = \frac{i-1}{N}$ for $1\leq i\leq N+1$. Assume Condition \ref{cond:graphon-cond-1}$(1)$ holds, there exists a sequence of real numbers $(\beta_N)_{N\in \mathbb{N}}$, such that $\beta_N\in (0,1]$, $\beta_Ng(x_i,x_j)\leq 1$ for every $N\in \N, 1\leq i,j\leq N$ and $N\beta_N \rightarrow \infty$ as $N\rightarrow \infty$ and the random variables $(\zeta^N_{ij})_{1\leq i,j\leq N}$ are $\cF_0$-measurable, independent and with distribution $\zeta^N_{ij}\sim \mathrm{Bernoulli}(g(x_i,x_j)\beta_N)$. 
\end{condition}

\begin{corollary}
    \label{coro:bounded-L1-coro-Wrand}
    Assume Conditions \ref{cond:initial-conditions}, 
    \ref{cond:coefficients}, \ref{cond:graphon-cond-2}, \ref{cond:almost-continuity}, \ref{cond:bound-initial-cond}. If $g$ is bounded and $\sigma$ does not depend on the measure argument. 
    Then equation \eqref{eq:first.main.conv} holds.
\end{corollary}

The proof of Corollary \ref{coro:bounded-L1-coro-Wrand} is straightforward if we notice that under the extra assumption that $g$ is bounded Condition \ref{cond:graphon-cond-2} is included in Condition \ref{cond:graphon-cond-1} as we can define $g_N(x,y) = g(x_i,x_j)$ for $(x,y)\in [x_i,x_{i+1})\times[x_j, x_{j+1})$ and the convergence $\|g-g_N\|_2\to0$ follows from the dominated convergence Theorem. Then, in this case we can apply Theorem \ref{thm:convergence-bounded-graphon-L1}.

The proof of Theorem \ref{thm:convergence-bounded-graphon-L1} begins with a standard coupling argument.
But in addition to the fact that the limiting particle system is not i.i.d., the proof is made more challenging by the nonlinearity of the interaction, the possibility of vanishing degree $\beta_N$ and the fact that the interaction term is not a probability measure.
Some of the main steps of the proof involve deriving stability results that allow to reduce the problem to piece-wise constant graphons $h_N$ for which we can choose a sequence of points $(y_i)_{1\leq i\leq N}$ such that the family $(X^{y_i,h_N}_u)_{1\leq i\leq N}$ is pairwise independent. Getting around the fact that the interaction term is not a probability measure is mainly done using a standard Gaussian regularization argument. 

The price to pay for this argument to hold is to require higher order moments of the initial distribution of the particle system.
We can extend Theorem \ref{thm:convergence-bounded-graphon-L1} to the case where the diffusion term $\sigma$ depends on the measure argument. However, this will require the stronger assumption $\beta_N^2 N \rightarrow \infty$, see Theorem \ref{thm:convergence-bounded-graphon}.
The proof follows the same lines of arguments.

We will present various by-products of our main convergence result in Section \ref{sec:corollaries.conve}.
For instance, under an additional Lipschitz continuity condition on $g$, on the one hand we obtain a quantitative rate for the convergence in Theorem \ref{thm:convergence-bounded-graphon-L1} (see Corollary \ref{coro:rate-convergence}).
Moreover, we will extend the result to the case where the sequence of points $(x_i)_{1\leq i\leq N}$ is chosen randomly, i.e. replaced by independent uniformly distributed random variables $(U_i)_{1\le i\le n}$, the convergence still holds (see Corollary \ref{coro:random-points}). 
In this setting, we will derive a weak law of large numbers.
That is, we will show that 
 \begin{equation*}
    \frac{1}{N}\sum_{i=1}^N\E\bigg[\Big|\frac{1}{N}\sum_{i=1}^N X_t^{U_i,N}-\int_I X_t^{y,g}\lambda(dy)\Big| \bigg|  (U_i)_{1\leq i\leq N} \bigg] \xrightarrow{N\rightarrow \infty}  0\quad \text{in probability}.
\end{equation*}
We stress that constructing solutions of the graphon SDE in $L^2_\boxtimes(I\times\Omega,\mathcal{C})$ is essential to deriving this result, see Corollary \ref{coro:WLLN}. 
Another consequence of Theorem \ref{thm:convergence-bounded-graphon-L1} is a propagation of chaos result for the particle systems driven by arbitrary (independent) Brownian motions:
\begin{corollary}
\label{coro:convergence-laws-L1}
Under the assumptions of Theorem \ref{thm:convergence-bounded-graphon-L1} we have that
\begin{equation*}
    \frac{1}{N}\sum_{i=1}^N \E\Big[\sup_{t\in [0,T]}d_{\mathrm{BL}}(\cL(Y^{i,N}_t), \cL(X^{i,g}_t))\Big] \xrightarrow{N\rightarrow \infty} 0 .
\end{equation*}
\end{corollary}

Our second main convergence result is an $L^2$-type convergence result that under some extra assumptions will allow us to consider unbounded graphons. The main restriction we will need to impose for this result is that the interaction graph will need to be deterministic. 
We will make use of the function $f_N$ defined as
\begin{equation*}
    f_N(x) := \sum_{i=1}^{N-1} x_i {\mathbf{1}_{\{x\in [x_i,x_{i+1})\}}} + x_N \mathbf{1}_{\{x\in [x_N,1]\}}.
\end{equation*}

Formally, we will modify Condition \ref{cond:graphon-cond-1} as follows:

\begin{condition}
\label{cond:graphon-cond-deterministic}
    For every $N\in \N$ consider the sequence of points $x_i = \frac{i-1}{N}$ for $1\leq i\leq N+1$. Assume:
    \begin{enumerate}
        \item Condition \ref{cond:graphon-cond-1} $(1)$-$(2)$.
        \item For every $N$, the adjacency matrix $(\zeta^N_{ij})_{1\leq i,j\leq N}$ is defined as $\zeta^N_{ij} = g_N(x_i,x_j)\beta_N$.
        \item $g\circ (f_N,\mathrm{id})$ is dominated by a square-integrable function.
        \item $\|g\|_4<\infty$ and  $\|g_N-g\|_4\rightarrow 0$.
    \end{enumerate}
\end{condition}

\begin{theorem}
    \label{thm:convergence-unbounded-graphon-deterministic}
    Assume Conditions \ref{cond:initial-conditions}, 
    \ref{cond:coefficients}, \ref{cond:almost-continuity}, \ref{cond:bound-initial-cond} and \ref{cond:graphon-cond-deterministic} hold. Then for every $t\in [0,T]$ it holds
    \begin{equation}
    \label{eq:second.main.conv}
         \frac{1}{N}\sum_{i=1}^N\E\Big[ \sup_{t\in [0,T]}\big|X^{i,N}_{t}-X^{i,g}_t\big|^2 \Big] +  \frac1N\sum_{i=1}^N\E\Big[d^2_{\mathrm{BL}}\Big(m^{i,N}_{\X_t},  \int_I g(x_i,y)\mathcal{L}(X_t^y) \lambda(dy)\Big) \Big] \xrightarrow{N\rightarrow \infty} 0. 
    \end{equation}
\end{theorem}

As above, we can state an immediate Corollary by modifying Condition \ref{cond:graphon-cond-deterministic}.

\begin{corollary}
    Assume Conditions \ref{cond:initial-conditions}, 
    \ref{cond:coefficients} and \ref{cond:almost-continuity}, \ref{cond:bound-initial-cond}.
    Further assume Condition \ref{cond:graphon-cond-deterministic} with $g_N$ therein given by $g_N(x,y):=g(x_i,x_j)$ for all $x\in [x_i,x_{i+1})$, $y\in [x_{j}, x_{j+1})$ and $g(f_N,f_N)$ dominated by a square integrable function.
    Then \eqref{eq:second.main.conv} holds.
\end{corollary}

Again, the result follows by applying Theorem \ref{thm:convergence-unbounded-graphon-deterministic} with the sequence of functions $g_N$ given in the statement.
Notice that Condition \ref{cond:graphon-cond-deterministic}.$(4)$ follows from the domination assumption in the statement of the Corollary.

\subsection{Literature review}
\label{sec.literature}
The convergence of particle systems to their large population limit is a classical problem in applied probability with a long history and several important contributions.
Starting with the work of \citet{Kac56}, with classical references and surveys written by \citet{sznitman} and \citet{Chaos-review-2,Chaos-review-1}. 
To cite a few milestones among many others we can mention \cite{Durmus-Eberle20,Graham-sticky-1989,Hao-Roeckner-zhang24,Jabin18,Jabir-Wang16,Lacker18ECP,Meleard96,Mishler-Mou13,Misch-Mou-Wenn15,Laur-Tang22,Lacker23,Leonard-LLN,Serfaty20}. Due to its wider range of applications (in statistical physics, biology, finance and social sciences) and mathematical interests, the study of particle systems on graphs is quickly gaining traction.
This is also motivated by a growing interest in graphon games, see \cite{Caines-Huang18,Carmona-stochastic-graphon-games,Lacker-Soret23,Laur-Tangp-Zhou24,Parise23,Tangpi-Zhou24}.

The case of dense networks has been more extensively studied.
In the case of particle systems on Erd\"os-R\'enyi graphs, the usual mean field limit obtained for interacting particle systems persists and the limiting particle system is completely independent.
We refer for instance to \cite{Zoe22,Bet_Coppini_Nardi_2024,Bhamidi-weakly-interacting-part-syst-inhomogeneoud-random-graphs,Chiba-Kuramoto-graphs,Medvedev2019-nonlinear-heat-W-random} for works on limits of particle systems on dense graphs.
The paper \cite{coppini2024nonlineargraphonmeanfieldsystems} is more closely related to ours insofar that it treats the case of nonlinear interaction.
In this work, the authors study the propagation of chaos of a particle system with Brownian motions sampled from an e.p.i process $B\in L^2(I\times\Omega, \cC)$ to the graphon SDE \eqref{eq:SDE-nonlinear}. 
In contrast to the present paper, in \cite{coppini2024nonlineargraphonmeanfieldsystems} the limiting graphon SDE \eqref{eq:SDE-nonlinear} depends on a probability measure and more importantly, \cite{coppini2024nonlineargraphonmeanfieldsystems} focuses on dense and non-random graph models. 
Furthermore, the second condition in \eqref{eq:Assum.Coppini} excludes examples of graph sequences converging to a zero graphon. 
For instance, this would correspond to the case of a weakly interacting particle system on a (dense) weighted graph on $N$ vertices such that each edge has weight $1/N$. 
For such a graph sequence, the appropriate step graphon is $g_N \equiv \frac{1}{N}$ and it converges to the zero graphon.

In the (very sparse) case where the average degrees of vertices are uniformly bounded in $N$, the particles have strong interactions and particles remain correlated large population limit. Works in this setting typically rely on the notion of local weak convergence of sparse graphs.
See for instance \cite{Oliveira-Reis-interacting-diffusions-on-random-graphs,Budhiraja2017SupermarketMO,Lacker-local-weak-convergence-sparse-networks,ganguly2022hydrodynamic-limits, oliveira-interacting-on-sparse-graphs,Bhamid-sly21,Ganter-Sch20,Huang-Durrett}.
We also refer to \citet{interacting-stochastic-proceses-sparse} for a survey of recent results on the topic.
Much less work has been on the analysis of interacting particle systems on graphs between the two extrems of the spectrum.
Namely, between dense graphs with $\Omega(N^2)$ edge and very sparse graphs with $O(N)$ edges.
This is for instance the case of $W$-random graphs induced by unbouded graphons such as the power law graphon.
It is thus important to develop a convergence theory bridging these two extremes.
Works along these lines include \cite{Coppini-Luc-Poq23,Oliveira-Reis-interacting-diffusions-on-random-graphs} and \cite[Section 4]{bayraktar-graphon-mean-field-systems} who use a re-normalization of the graph as we do in \eqref{eq:definition.empirical.measure}.
However, these papers assume the limiting graphon to be bounded, thus covering  cases like the Erd\"os-R\'enyi graph, but not unbounded graph limits like the power law. 
In addition, this paper is restricted to the linear interaction case.
\citet{Medvedev-semilinear-heat-sparse} study differential equations interacting through a $W$-random graph. 
It is worth remarking that in this work the notion of $L^p$ convergence as defined by \citet{Borgs2014AnT} is used and the graphons considered are unbounded. 

\subsection{Outline of the rest of the paper}
The remainder of the paper is concerned with the proofs of the main results of the article.
The proof of the well-posedness Theorem \ref{thm:existence-uniqueness} is presented in Section \ref{sec:construction-graphon-integral} where, in particular, we address the measurability issues in full details for the sake of completeness and for further references.
In Section \ref{sec:proof-convergence} we prove the main convergence results stated in Theorem \ref{thm:convergence-bounded-graphon-L1}, Theorem \ref{thm:convergence-bounded-graphon}.
In this section, we also discuss the case where $\sigma$ is allowed to depend on the interaction term and derive various consequences of our main convergence results; such as conditions guaranteeing a convergence rate in Theorem \ref{thm:convergence-bounded-graphon-L1} and an analysis of the case of random points $(x_i)_{1\le i\le N}$.

\subsection{Frequently used notation}
\label{sec:notation}
For ease of reference, let us define here some frequently used notation.
We will consider the interval $I=[0,1]$, a finite time horizon $T\geq 0$ and dimensions $m,n\in \N$.
Denote by $\mathcal{C}=C([0,T],\mathbb{R}^n)$ the space of continuous functions from $[0,T]$ to $\R^n$ and equip this space with the uniform norm  $\|\omega\|_{\mathcal{C}} = \sup_{t\in [0,T]} |\omega(t)|$.
For a Polish space $E$ we will denote $\mathcal{P}_2(E)$ to the set of square integrable probability measures on $E$, that is, such that  $\mu\in \mathcal{P}_2(E)$ if $\int_{E} |x|_2^2 d\mu(x)< \infty$. We will also define $\mathcal{M}_+(E)$ to the set of non-negative and finite Borel measures on $E$, $\mathcal{M}(E)$ to the set of Borel measures on $E$,  $\mathcal{P}(E)$ to the set of Borel probability measures on $E$ and $\mathcal{B}(E)$ to the set of Borelians in $E$. We will equip all these spaces with the topology of weak convergence.
We denote:
\begin{equation*}
    1\text{-Lip}_b = \{f \in C(\mathbb{R}^n): |f(x)-f(y)|\leq |x-y|, \sup_{x\in \mathbb{R}^n}|f(x)|\leq 1\},
\end{equation*}
where $C(\mathbb{R}^n)$ is the set of continuous functions from $\mathbb{R}^n$ to $\mathbb{R}$. We also introduce the bounded Lipschitz distance and the bounded Lipschitz norm, this distance metricizes the set $\mathcal{M}_+(E)$. Hence, for $\mu,\nu \in \mathcal{M}_+(\mathbb{R}^n)$ as:
    \begin{equation*}
        d_{\mathrm{BL}} (\mu,\nu) = \sup_{f \in 1\text{-Lip}_b} \Big( \int_{\mathbb{R}^n} f d\mu - \int_{\mathbb{R}^n} f d\nu\Big), \quad \text{ and } \quad \|\mu\|_{BL} = \sup_{f\in 1\text{-Lip}_b } \int_{\mathbb{R}^n} f d\mu.
    \end{equation*}
Furthermore, we denote by $\cL(Z)$ the law of a random variable $Z$. Finally, we denote the $p$-norm of a graphon as
    \begin{equation*}
        \|g\|_p = \left(\int_0^1\int_0^1 |g(x,y)|^pdxdy\right)^{1/p}.
    \end{equation*}

\section{Existence, uniqueness and stability of graphon SDEs}
\label{sec:construction-graphon-integral}

This section is dedicated to the proof of Theorem \ref{thm:existence-uniqueness} as well as stability results for the graphon SDE.
We will start by discussing measure theoretic issues after which the proof of well-posedness will follow by a standard fixed point construction.

\subsection{Measure theoretic preliminaries}

We begin by making clear the meaning of $\int_I g(x,y) \mathcal{L}(X^y_t)\lambda(dy)$. 
This will be understood as a measure defined for every $A\in \cB(\mathbb{R}^n)$ as:
\begin{equation}
\label{def:nu-measure}
    \int_I g(x,y) \mathcal{L}(X^y_t)\lambda(dy)(A) = \int_I g(x,y) \mathcal{L}(X^y_t)(A)\lambda(dy).
\end{equation}
Let us clarify conditions under which this measure is well-defined.
More generally, let us be given a fixed family of probability measures $\mu^x \in \cP(E)$. 
We will prove a useful characterization of measurability for functions taking values on $\cP(E)$.
This result has appeared in the literature (see for example \cite{Sun-2006}), but as we could not find a proof we will include one for completeness. 
Based on this characterization, we will ensure existence of the measure  
\begin{equation}
\label{eq:graphon.integral}
    \nu^{\mu,x}:=\int_I g(x,y)\mu^{y}\lambda(dy)\quad \text{for almost all } x\in I
\end{equation}
and its relation with $\mu^y$.

\begin{proposition}
\label{prop:characterization-measurability-maps-taking values-on-positive-measures}
Let $E$ be a Polish space and let $(Y,\mathcal{Y},\lambda)$ be a measure space. Then,  $F : (Y,\mathcal{Y}) \rightarrow (\mathcal{P}(E), \mathcal{B}(\mathcal{P}(E)))$ is measurable if and only if  for every $A\in \mathcal{B}(E)$ the map $F_A:(Y,\mathcal{Y}) \rightarrow ([0,1], \mathcal{B}([0,1]))$ defined as $F_A(y) = F(y)(A)$ is measurable.
\end{proposition}

\begin{proof}
$\Rightarrow)$ We can write $F_A$ as a composition: 
\begin{equation}
    y \xmapsto{F} F(y) \xmapsto{ev_A} F(y)(A),
\end{equation}
where $ev_A: \mathcal{P}(E) \rightarrow [0,1]$ is given by $ev_A(\mu) = \mu(A)$. The first map is measurable by assumption and it will be enough to show that the second one is also measurable for every $A\in \mathcal{B}(E)$. 

First take $A$ to be an open set. We will show that $ev_A$ is sequentially lower-semicontinuous. 
Let $(\eta_k)_{k\in \mathbb{N}}$ be a sequence in $\mathcal{P}(E)$ such that $\eta_k \rightarrow \eta \in \mathcal{P}(E)$, in the sense of weak convergence.
By Portmanteau's Theorem we have $\liminf_{k\rightarrow \infty} \eta_k(A) \geq \eta(A)$, which implies that $ev_A$ is lower semicontinuous and, in particular, measurable. 

Consider the set $\mathcal{A} = \{V \in \mathcal{B}(E) : ev_V(\cdot) \text{ is measurable}\}$. 
We claim that $\mathcal{A}$ is a $\lambda$-system. 
As $ev_E(\mu) = \mu(E) = 1$ for all $\mu$, we have $E\in \mathcal{A}$. 
If $A,B\in \mathcal{A}$ then $B\setminus A \in \mathcal{A}$. In fact, we can write
\begin{equation*}
    ev_{B\setminus A}(\mu) = \mu(B\setminus A) = \mu(B)-\mu(A) = ev_B(\mu) - ev_A(\mu).
\end{equation*}
Then $ev_{B\setminus A}$ is measurable. 
Finally, take a sequence $(A_k)_{k\geq 1}$ with $A_k \in \mathcal{A}$ and $A_k\subseteq A_{k+1}$ for every $k\in \mathbb{N}$.
Then, we have that if $A=\cup_{k\geq 1}A_k$, then
\begin{equation*}
    ev_A(\mu) = \mu(\cup_{k\geq 1} A_k) = \lim_{k\rightarrow \infty} \mu(A_k) = \lim_{k\rightarrow \infty} ev_{A_k}(\mu),
\end{equation*}
which shows that $ev_{A}$ is measurable. 
This proves the claim.

Then we have that $\{O \in \mathcal{B}(E): O \text{ is open}\}\subseteq \mathcal{A}$ and as the open sets form a $\pi$-system then by the $\pi-\lambda$ theorem we have that $\mathcal{B}(E) \subseteq \mathcal{A}$.

$\Leftarrow)$ By \cite[Proposition 7.25]{Stochastic-optimal-control-discrete-time-case}, we have that
\begin{equation*}
    \mathcal{B}(\mathcal{P}(E)) = \sigma\left(\{\mu \mapsto \mu(D)\}_{D\in \mathcal{B}(E)}\right).
\end{equation*}
To show that $F$ is measurable we need to show that if $B\in \mathcal{B}(\mathcal{P}(E))$ then $F^{-1}(B)\in \mathcal{Y}$. As $\mathcal{B}(\mathcal{P}(E))$ is generated by $\{\text{ev}_D : \mu \mapsto \mu(D)\}_{D\in \mathcal{B}(E)}$ it is enough to show that for $A\in \mathcal{B}([0,1])$ and $D\in \mathcal{P}(E)$, $F^{-1}(\text{ev}_D^{-1}(A)) \in \mathcal{Y}$ but this holds by hypothesis. 
\end{proof}

This proposition allows to give a precise meaning to the integral defining $\nu^{\mu,x}$ in equation \eqref{eq:graphon.integral}.
\begin{lemma}
\label{lemma:nu-defines-a-measure}
    If $(\mu^y)_{y\in I}$ is a measurable family of probability measures on $\cB(E)$, then for almost every $x\in I$,
    it holds $\nu^{\mu,x} \in \cM_+(E)$ and for every $A \in \cB(E)$ the map $x\mapsto \nu^{\mu,x}(A)$ is measurable.
\end{lemma}
\begin{proof}
    The measure $\nu^{\mu,x}$ is understood as
    \begin{equation*}
        \nu^{\mu,x}(A) = \int_{I}g(x,y)\mu^y(A)\lambda(dy), \quad A \in \cB(E).
    \end{equation*}
    By Proposition \ref{prop:characterization-measurability-maps-taking values-on-positive-measures}, for every $A\in \mathcal{B}(E)$ the map $y \mapsto \mu^y(A)$ is measurable. 
    This is because we can write this map as the composition $y\mapsto \mu^y \mapsto \mu^y(A)= ev_A(\mu^y)$.  
    Therefore, the map $(x,y)\mapsto g(x,y)\mu^{y}(A)$ is measurable as well.
    Moreover, for every $x,y \in I$ and $A\in \mathcal{B}(E)$, we have that
    \begin{equation*}
        \int_{I\times I} |g(x,y)|^2\mu^y(A) \lambda\otimes\lambda(dx,dy) \leq \int_{I\times I} |g(x,y)|^2 \lambda\otimes\lambda(dx,dy) < \infty.
    \end{equation*}
    Hence, by Fubini's Theorem the map $y\mapsto g(x,y)\mu^y(A)$ is integrable for almost every $x\in I$ and every $A \in \mathcal{B}(E)$.
    In particular, 
    \begin{equation*}
        x\mapsto \nu^{\mu,x}(A) = \int_I g(x,y) \mu^{y}(A) \lambda(dy)
    \end{equation*}
    is measurable for all $A \in \cB(E)$.

    Checking that $\nu^{\mu,x}$ is a measure follows from the fact that $\mu^y$ is a measure. The non-negativity follows from the fact that $g(x,y)\geq 0$ and the finiteness from the fact that $g$ is integrable. 
\end{proof}
If $(\mu^y)_{y\in I}$ is a family of measures on $\cB(\cC)$, then we will define the finite dimensional marginal on $\cB(\R^n)$ as usual.
In fact, considering projection $\pi_t : \mathcal{C} \rightarrow \mathbb{R}^n$ defined as $\pi_t(\omega) = \omega_t$,
for every $A\in \mathcal{B}(\mathbb{R}^n)$ we have
\begin{equation*}
    \nu_t^{\mu,x}(A) := \nu^{\mu,x}\left( \pi_t^{-1}(A)\right) = \int_I g(x,y)\mu^y(\pi_t^{-1}(A)) \lambda(dy) = \int_I g(x,y)\mu^y_t(A)\lambda(dy),
\end{equation*}
where $\mu^y_t(A) := \mu\circ \pi_t^{-1}(A)$.

Next, we describe integration with respect to the measure $\nu^{\mu,x}$. 
\begin{proposition}
\label{prop:graphon.integration}
   Let $(\mu^y)_{y\in I}$ be a measurable family of probability measures on $\cB(\cC)$.
   Let $f: \mathcal{C} \rightarrow \mathbb{R}$ be $\mu^y$-integable for almost every $y\in I$.
   Then $f$ is $\nu^{\mu,x}$-integrable for almost every $x \in I$ and we have that
    \begin{equation}
    \label{eq:transform-lambda-nu}
        \int_\mathcal{C} f(\omega) \nu^{\mu,x}(d\omega) = \int_I\int_\mathcal{C} g(x,y) f(\omega) \mu^y(d\omega) \lambda(dy).
    \end{equation}
    Similarly, if  $f: \mathbb{R}^n \rightarrow \mathbb{R}$ is $\mu^y_t$-integrable for almost every $y\in I$, then $f$ is $\nu^{\mu,x}_t$-integrable for almost every $x\in I$ and we have 
    \begin{equation*}
        \int_{\mathbb{R}^n} f(z) \nu^{\mu,x}_t(dz) = \int_I\int_{\mathbb{R}^n} g(x,y) f(z) \mu^y_t(dz) \lambda(dy).
    \end{equation*}
\end{proposition}

\begin{proof} 
    The proof uses classical arguments.
    We will give only the main steps.

    If $f$ is a simple function, i.e. $f=\sum_{i=1}^N a_i \mathbf{1}_{A_i}$ for some $a_i\in \mathbb{R}$ and $A_i\in \mathcal{B}(\mathcal{C})$, then the result is clear by definition of $\nu^{\mu,x}$.
    If $f\geq 0$ is measurable, then there exists a sequence of positive simple functions $(f_k)_{k\ge1}$ increasing to $f$ pointwise. 
    Then, we have
    \begin{equation*}
    \begin{split}
        \int_\mathcal{C} f(\omega) \nu^{\mu,x}(d\omega)
        =  \lim_{k\rightarrow \infty} \int_I \int_\mathcal{C} f_k(\omega)g(x,y) \mu^y(d\omega)\lambda(dy) 
        =  \int_I \int_\mathcal{C} f(\omega) g(x,y) \mu^y(d\omega) \lambda(dy), \\
    \end{split}
    \end{equation*}
    where we used the monotone convergence theorem and the fact that $g$ is positive. 
    If $f$ is a measurable function we can write it as $f= f^{+}-f^{-}$ where $f^{\pm}(x)=\max(\pm f(x),0)$, these are measurable and positive so the result holds from linearity of the integral. 

    For the second part we notice that: 
    \begin{equation*}
    \begin{split}
        \int_{\mathbb{R}^n} f(z)\nu^{\mu,x}_t(dz) = &  \int_\mathcal{C} f(\omega_t) \nu^{\mu,x}(d\omega) = \int_I \int_\mathcal{C} g(x,y) f(\omega_t) \mu^{y}(d\omega) \lambda(dy)
        =  \int_I \int_{\mathbb{R}^n} f(z) \mu^{y}(dz) \lambda(dy). 
    \end{split}
    \end{equation*}
\end{proof}

\subsection{Measurability of SDEs depending on a parameter}

The aim of this section is to discuss measurability (in the box product probability space) of SDEs depending on a given parameter.
This will be an essential building bloc of the existence proof.

Let $\nu^\cdot:I\ni x\mapsto \nu^x \in C([0,T],\cM_+(\R^n))$ be a measurable map and consider the SDE
\begin{equation}
\begin{split}
\label{eq:graphon-SDE-mu}
    dX_t^{x} = & b(t,X_t^{x}, \nu^x_t)dt + \sigma(t,X_t^{x}, \nu^x_t)dB_t^x, \quad 
    X_0^{x} =  \xi^x.
\end{split}
\end{equation}
The following result shows measurability of the family $\{X^x\}_{x\in I}$.
\begin{proposition}
 \label{prop:measurable.SDE}
    If Conditions \ref{cond:initial-conditions} and \ref{cond:coefficients} are satisfied and $\nu$ is such that
        \begin{equation*}
        \int_I \sup_{s\in[0,T]}\|\nu_{s}^{x}\|_{\mathrm{BL}}\lambda(dx) < \infty,
    \end{equation*} 
    then there is a unique family $X = (X^x)_{x\in I}$ of processes such that $X \in L^2_{\boxtimes}(I\times \Omega, \cC)$ and for almost every $x\in I$, $X^x$ satisfies \eqref{eq:graphon-SDE-mu} and is adapted to $\F^x$.
\end{proposition}
It is well-known, see e.g. \cite[Theorem 5.2.9]{Karatzas-Shreve} that for each $x$ such that $B^x$ is defined, the SDE \eqref{eq:graphon-SDE-mu} admits a unique $\F^x$-adapted solution $X^x$.
It remains to show that $X = (X^x)_{x\in I}$ belongs to $L^2_\boxtimes(I\times\Omega, \cC)$.
To this end, let $\S^2_\boxtimes(I\times\Omega, \cC)$ as the subset of $\F^x$-adapted processes in $L^2_\boxtimes(I\times\Omega, \cC)$.
That is,
\begin{equation*}
    \S^2_\boxtimes(I\times\Omega, \cC) := \left\{X\in L^2_\boxtimes(I\times\Omega, \cC): \text{ for $\lambda$-almost every $x\in I$, $X^x$ is $\F^x$-adapted}\right\}.
\end{equation*}

\begin{lemma}
\label{lem:complete}
    The space $\S^2_\boxtimes(I\times\Omega, \cC)$ is a closed subset of $L^2_\boxtimes(I\times\Omega, \cC)$ and hence complete. 
\end{lemma}
\begin{proof}
    Consider a sequence of processes $(X^k)_{k\in \N}$ in $\S^2_\boxtimes(I\times\Omega, \cC)$ converging to $X$ in $L^2_{\boxtimes}(I\times\Omega, \cC)$.
    As $L^2_\boxtimes(I\times\Omega, \cC)$ is closed then $X\in L^2_\boxtimes(I\times\Omega, \cC)$. 
    We claim that for $\lambda$-almost every $x\in I$ we have that $X^{k,x} \rightarrow X^x$ in $L^2(\Omega,\cC)$. 
    This follows from the fact that by definition of the sequence and by Fatou's lemma we have
    \begin{equation*}
         \int_I \liminf_{k\rightarrow\infty}\E\bigg[\sup_{t\in [0,T]}|X^{k,x}_t-X_t^x|^2\bigg]\lambda(dx) \leq  \lim_{k\rightarrow\infty} \int_I \E\Big[\sup_{t\in [0,T]}|X^{k,x}_t-X_t^x|^2\Big]\lambda(dx) = 0.
    \end{equation*}
    We thus that for $\lambda$-almost every $x\in I$
    \begin{equation*}
        \lim_{k\rightarrow\infty}\E\Big[\sup_{t\in [0,T]}|X^{k,x}_t-X_t^x|^2\Big] = 0.
    \end{equation*}
    Hence, for $\lambda$-almost every $x$ there exists a subsequence again denoted $(X^k)_{k\in \N}$ such that $X^{k,x} \rightarrow X^x$ $\P$-almost surely. 
    Thus, as $X^{k,x}$ is 
    $\F^x$-adapted, so is $X^x$.
\end{proof}

For any $\cX \in \S^2_{\boxtimes}(I\times\Omega, \cC)$, consider the family of processes $F(\cX)$ given by
\begin{equation}
\label{eq:def.F.fixedpoint}
    F(\cX)_t^x := \xi^x + \int_0^t b(s,\cX_s^{x},\nu_s^x)ds + \int_0^t \sigma(s,\cX_s^{x},\nu_s^x)dB^x_s. 
\end{equation}
Let us show that the function $F$ maps $\S^2_{\boxtimes}(I\times\Omega, \cC)$ into itself.

\begin{lemma}
\label{lemma:good-definition-of-psi} 
   Under the assumptions of Proposition \ref{prop:measurable.SDE}, the map 
    \begin{equation*}
        I \times \Omega \ni(x,\omega) \mapsto F(\cX)^{x}(\omega) \in \cC  
    \end{equation*}
    is $\mathcal{I}\boxtimes \mathcal{F}$-measurable and for almost every $x\in I$, $F(\cX)^x$ is $\F^x$-adapted. 
\end{lemma}

\begin{proof}
    The fact that $F(\cX)^x$ is $\F^x$-adapted is clear. Let us focus on the measurability in the box-product.

    We need to show that each of the summands defining $F(\cX)$ are measurable. 
    By definition, the map $(x,\omega) \mapsto \xi^{x}(\omega)$ is $\cI\boxtimes \cF$-measurable. 
    
    We first look at the Lebesgue integral. 
    Since for every $(x,\omega)$ the maps $t\mapsto \cX_t^{x}(\omega)$ and $t\mapsto \nu^x_t$ are continuous and for every $t\in [0,T]$ we have that $(x,\omega)\mapsto \cX^x_t(\omega)$ and $x\mapsto \nu_t^x$ are measurable  (this is because the evaluation on a fixed point is a continuous mapping and we can decompose the maps in a measurable map and a continuous map). 
    It follows by continuity of $b$ that $(s,x,\omega)\mapsto b(s,\cX_s^{x}(\omega), \nu_s^{x})$ is measurable with respect to $\mathcal{B}([0,T])\otimes (\mathcal{I}\boxtimes \mathcal{F})$~\cite[Lemma $4.51$]{aliprantis2007infinite}. 
    Thus by Fubini's theorem we have that $\int_0^t b(s,X_s^{x}(\omega), \nu_s^{x}) ds$ is $\mathcal{I}\boxtimes\mathcal{F}$-measurable.
    Notice that we are allowed to use Fubini's theorem because by the growth condition on $b$ and properties of integration under $\lambda\boxtimes \P$ we have that
        \begin{align*}
            \int_{I\times\Omega} \int_0^t &|b(s,X_s^{x}(\omega), \nu_s^{x})| ds \lambda\boxtimes \mathbb{P}(dx,d\omega)\\
            &  \leq  CT \bigg(1 + \int_I \mathbb{E}\Big[\sup_{s\in [0,T]}|\cX_s^{x}|\Big]\lambda(dx) + \int_I\sup_{s\in [0,T]}\|\nu_{s}^{x}\|_{\mathrm{BL}} \lambda(dx)\bigg)<\infty.
        \end{align*}

    Measurability of the stochastic integral is a bit more complex as it is not defined pathwise.
    Denote by $\mathfrak{P}_l := \{0=t_0 <t_1< \dots < t_l = T\}$ a partition of $[0,T]$ with $l$ elements, with $\|\mathfrak{P}_l\| = \max_{1\leq i\leq l}|t_i-t_{i-1}|$ and define
    \begin{equation*}
        \sigma^{\mathfrak{P}_l}(s,\cX_s^{x}, \nu_s^{x}) := \sum_{i=1}^l \sigma(t_{i-1}, \cX_{t_{i-1}}^{x},\nu_{t_{i-1}}^{x})\mathbf{1}_{s\in [t_{i-1},t_i)}.
    \end{equation*}
    For every fixed $t_i \in [0,T]$ we have that $(x,\omega) \mapsto \sigma(t_{i},\cX_{t_i}^x(\omega), \nu^x_{t_i})$ and $(x,\omega)\mapsto B_{t_i}^x(\omega)$ are $\cI\boxtimes\cF$-measurable. 
    Thus, the finite sum
    \begin{equation*}
        \int_0^t \sigma^{\mathfrak{P}_l}(s,\cX_s^{x},\nu_s^{x}) dB_s^{x} = \sum_{i=1}^l \sigma(t_{i-1}, \cX_{t_{i-1}}^{x},\nu_{t_{i-1}}^{x}) (B^x_{t_{i}}-B^x_{t_{i-1}})
    \end{equation*}
    is  $\mathcal{I}\boxtimes \mathcal{F}$-measurable as well.
    For $\lambda$-almost every $x\in I$ we have that
    \begin{equation}
    \label{eq:bounded.f}
        \mathbb{E}\Big[\sup_{t\in [0,T]} |\sigma(t,\cX_t^{x},\nu^{x}_t)|^2\Big] \leq CT \Big(1 + \E\Big[\sup_{t\in[0,T]}|\cX_t^{x}|^2\Big] + \sup_{t\in [0,T]}\|\nu^x_t\|^2_{\mathrm{BL}} \Big) < \infty
    \end{equation}
    and for every $(x,\omega)$ the map $t\mapsto f(t):=\sigma(t,\cX_t^{x}(\omega),\nu_t^{x})$ is continuous.
    Hence, we obtain by Cauchy-Schwarz inequality
    \begin{equation*}
        \begin{split}
            |\mathbb{E}[f(s_j)f(t_j)]-\mathbb{E}[f(s)f(t)]| &\leq  \mathbb{E}[|f(s_j)f(t_j)-f(s)f(t)|]\\
            &\leq  \mathbb{E}[|f(s_j)f(t_j)- f(s_j)f(t)|] + \mathbb{E}[|f(s_j)f(t)-f(s)f(t)|]\\
            &=  \mathbb{E}[|f(s_j)|^2]^{1/2}\mathbb{E}[|f(t_j)-f(t)|^2]^{1/2} \\
            & \quad + \mathbb{E}[|f(t)|^2]^{1/2}\mathbb{E}[|f(s_j)-f(s)|^2]^{1/2} \rightarrow 0,
    \end{split}
    \end{equation*}
    where the convergence uses $(s_j,t_j) \rightarrow (s,t)$, continuity of $f$, equation \eqref{eq:bounded.f} and the dominated convergence theorem. Thus, by \cite[Theorem 4.7.1]{Introduction-stochastic-integration}, for $\lambda$-almost every $x\in I$,
    \begin{equation*}
        \int_0^t \sigma^{\mathfrak{P}_l}(s,\cX_s^{x},\nu_s^{\mu,x}) dB_s^{x} \xrightarrow{\mathfrak{P}_l\to0} \int_0^t \sigma(s,\cX_s^{x},\nu_s^{x}) dB_s^{x} 
    \end{equation*}
    in the sense of $L^2(\Omega,\mathbb{R}^n)$-convergence. The idea would be to use the fact that $L^2$ convergence implies almost sure convergence along a subsequence and for the purpose of showing that the limit is measurable this is enough. However, the subsequence could depend on $x$ so we need to get that the convergence holds in $L^2_\boxtimes (I\times\Omega, \R^n)$. 
    
    A priori we know that $\E\left[ \int_0^t \sigma(s,\cX_s^{x},\nu_s^{x}) dB_s^{x}\right]$ is well defined for $\lambda$-almost every $x\in I$, but we still need to show that we can integrate with respect to $x$. 
   By It\^o isometry,
    \begin{equation*}
    \begin{split}
        \E\bigg[\bigg|\int_0^t \sigma(s,\cX_s^{x},\nu_s^{x})dB_s^x - \int_0^t \sigma^{\mathfrak{P}_l}(s,\cX_{s}^{x},\nu_{s}^{x})dB_s^x\bigg|^2\bigg] =  \E\bigg[  \int_0^t |\sigma^{\mathfrak{P}_l}(s,\cX_{s}^{x},\nu_{s}^{x})-|\sigma(s,\cX_s^{x},\nu_s^{x})|^2 ds\bigg].
    \end{split}
    \end{equation*}
    Since $\sigma$ satisfies the same conditions as $b$ the argument\footnote{We would only need to modify that $\sigma^{\mathfrak{P}_l}$ is not continuous but is piecewise constant in $s$ so we can claim measurability in $\cB([0,T])\otimes(\cI \boxtimes \cF)$ regardless.} used to obtain measurability of the Lebesgue integral gives us that  $(x,\omega) \mapsto \int_0^t |\sigma^{\mathfrak{P}_l}(s,\cX_{s}^{x},\nu_{s}^{x})-|\sigma(s,\cX_s^{x},\nu_s^{x})|^2 ds$ is $\cI\boxtimes\cF$-measurable and hence 
    \begin{equation*}
        x\mapsto \E\bigg[ \int_0^t |\sigma^{\mathfrak{P}_l}(s,\cX_{s}^{x},\nu_{s}^{x})-|\sigma(s,\cX_s^{x},\nu_s^{x})|^2 ds\bigg]
    \end{equation*}
    is $\cI$-measurable. In particular we can integrate with respect to the $x$ variable. This yields
    \begin{align*}
    \label{eq:A}
        &\int_I \mathbb{E}\bigg[\bigg|\int_0^t \sigma(s,\cX_s^{x},\nu_s^{x})dB_s^x - \int_0^t \sigma^{\mathfrak{P}_l}(s,\cX_{s}^{x},\nu_{s}^{x})dB_s^x\bigg|^2\bigg]\\
        &\le   \int_I \Big\|\int_0^t \sigma(s,\cX_s^{x},\nu_s^{x})dB^x_s - \int_0^t\sigma^{\mathfrak{P}_l}(s,\cX_{s}^{x},\nu_{s}^{x})dB_s^x\Big\|^2_{L^2(\Omega,\P)}d\lambda(y)\xrightarrow{\mathfrak{P}_l\to0}0,
    \end{align*}
    where the convergence follows by equation \eqref{eq:bounded.f}, the choice of $\cX$ and $\nu^x$, and the dominated convergence theorem. 
    We conclude that the convergence holds in $L^2_\boxtimes(I\times\Omega, \R^n)$, and this yields that the stochastic integral term is  $\cI\boxtimes\cF$-measurable.
    
    The above steps show that for each $t$, the random variable $F(\cX)^x_t$ is  $\cI\boxtimes\cF$-measurable.
    Now we conclude by extending this measurability to a the map $(x,\omega) \mapsto F(X)^x(\omega)\in \cC$. 
    This follows from the fact that $\mathcal{B}(\mathcal{C}) = \sigma(\text{cylindrical sets})$, where the cylindrical sets are intersection of a finite number of sets of the form
    \begin{equation*}
        C_{t,A} = \big\{(x,\omega) \in I\times \Omega : F(\cX)_t^{x}(\omega) \in A\big\}, \quad t\in [0,T], A\in \mathcal{B}(\mathbb{R}^n).
    \end{equation*}
    But by the first part of the proof we know that $C_{t,A}$ is $\mathcal{F}\boxtimes\mathcal{I}$-measurable. 
    This concludes the proof.
\end{proof}
Let us now turn to the measurability of the solution $X^x$ of the SDE
\eqref{eq:graphon-SDE-mu}.
\begin{proof}[Proof of Proposition \ref{prop:measurable.SDE}]
    Consider the map $F$ defined in \eqref{eq:def.F.fixedpoint}. 
    We know that $I\times \Omega \ni (x,\omega)\mapsto F(\cX)^x(\omega)\in\cC$ is measurable. 
    Moreover using integrability of $\xi^x$, Burkholder-Davis-Gundy inequality, and the bound \eqref{eq:bounded.f} (with a similar property holding for $b$) we have
    \begin{equation*}
    \begin{split}
        \E^{\boxtimes}\Big[\sup_{s\in [0,t]}|F(\cX)^x_s|^2\Big] 
        &\leq  \E^\boxtimes\left[\left|\xi^x\right|^2\right] + C \mathbb{E}^{\boxtimes}\bigg[\int_0^t |b(u,\cX_u^x, \nu^x_u)|^2du\bigg] + C_{BDG} \E^{\boxtimes}\bigg[\int_0^t |\sigma(u,\cX_u^x, \nu^x_u)|^2du\bigg]\\
        & \leq  \E^\boxtimes\left[\big|\xi^x\right|^2\big] + C \mathbb{E}^{\boxtimes}\bigg[\sup_{u\in[0,T]} |\cX_u^x|^2 + \sup_{u\in[0,T]}\|\nu_u^x\|_{BL}^2 du \bigg] < \infty.
    \end{split}
    \end{equation*}
    Notice that we are using the fact that  $\int_0^t \sigma(u,X_u^x,\nu_u^x)dB^x_u $ is a true martingale for $\lambda$-almost every $x\in I$, and then use the same argument as in the previous claim to show that we can still integrate with respect to $x$. 
    We conclude that $F(\cX)\in \S^2_\boxtimes(I\times\Omega, \cC)$.

    We now show that $F^k= F \circ \dots \circ F$ is a contraction for the composition of $k$ times $F$. 
    Take $\cX,\cY\in \S^2_\boxtimes(I\times\Omega, \cC)$ and bound using the Lipschitz continuity of $b$ and $\sigma$ and essentially the same arguments as above, we have
    \begin{equation*}
    \begin{split}
         \E^\boxtimes\bigg[\sup_{s\in [0,t]}|F(\cX)_s^x-F(\cY)_s^x|^2\bigg] 
        & \le C \E^\boxtimes\bigg[ \int_0^t \sup_{v\in [0,s]}|\cX^x_v-\cY^x_v|^2 ds \bigg].
    \end{split}
    \end{equation*}
    Then for $k\in \N$ we can iterate to get: 
    \begin{equation*}
    \begin{split}
        \E^\boxtimes\bigg[\sup_{s\in [0,t]}|F^k(\cX)_s^x-F^k(\cY)_s^x|^2\bigg] 
        \leq & C^k \E^\boxtimes\left[\sup_{v\in [0,T]}|\cX^x_v- \cY^x_v|^2\right]\int_0^t \int_0^{u_1}\dots \int_0^{u_{k-1}} du_1du_2 \dots du_k\\
        &\leq  \frac{C^k}{k!}\E^\boxtimes\bigg[\sup_{t\in [0,T]}|\cX^x_t - \cY^x_t|^2\bigg].\\
    \end{split}
    \end{equation*}
    Thus, for a large enough $k\in N$ we have that $F^k$ is a contraction.
    As $\S^2_\boxtimes(I\times\Omega, \cC)$ is complete (Lemma \ref{lem:complete}), it follows that there exists a unique $(X^x)_{x\in I}$ in $\S^2_\boxtimes(I\times\Omega, \cC)$ satisfying equation \eqref{eq:graphon-SDE-mu}. 
    This concludes the proof.
\end{proof}

In the proof of the existence theorem, we will need to use laws of processes like $X^x$ described above.
This will be done using the following remark:
    \begin{remark}
    Given $X\in L^2_\boxtimes(I\times \Omega, \mathcal{C})$, for every $y\in I$ we can define a function $X^y:\Omega \rightarrow \mathcal{C}$. 
    However, this map doesn't need to be a random variable, but it induces the map $I \ni y \mapsto X^y\in L^2(\Omega,\mathcal{C})$ that is almost surely well defined.
    That is, there exists a set $I_X\subseteq I$ such that for all $y\in I_X $, we have $X^y\in L^2(\Omega,\mathcal{C})$ and $\lambda(I_X^c) = 0$. 
As changing the map $y \mapsto X^y$ on a set of measure zero will not change the integral, we can (and will) define a modified map as follows:
given $X\in L^2_\boxtimes(I\times\Omega, \mathcal{C})$, we define the map $\Tilde{X} : I \rightarrow L^2(\Omega,\mathcal{C})$ as:
\begin{equation}
\label{eq:version}
    \Tilde{X}^y = \begin{cases}
    X^y \text{ if $X^y\in L^2(\Omega,\mathcal{C})$}\\
    0 \text{ otherwise.}
    \end{cases}
\end{equation}
We have $\tilde X^x\in L^2(\Omega, \cC)$ for every $x\in I$ and $\tilde X^x = X^x$ for $\lambda$-almost every $x$.
\end{remark}

\begin{lemma}
\label{eq:Lemma.integ.graphon.integr}
    Let $X\in L^2_{\boxtimes}(I\times \Omega, \mathcal{C})$. Then, the map $I\ni x \mapsto \mathcal{L}(X^x) \in \mathcal{P}(\mathcal{C})$ is measurable. 
    Consequently, the graphon integral $\nu^{ \mathcal{L}(X),x}$ defined in \eqref{eq:graphon.integral} belongs to $\cM_+(\cC)$ and satisfies
    \begin{equation*}
        \int_I \sup_{t \in[0,T]}\|\nu_{t}^{ \mathcal{L}(X),x}\|_{\mathrm{BL}}\lambda(dx) < \infty.
    \end{equation*}
\end{lemma}
\begin{proof}
    For every $B\in \mathcal{B}(\mathcal{C})$ we have that
    \begin{equation*}
    \begin{split}
        \mathcal{L}(X^y)(B)
        = \int_{\Omega} \mathbf{1}_{\{X^y(\omega)\in B\}}\mathbb{P}(d\omega).
    \end{split}
    \end{equation*}
    As we know that the function $\mathbf{1}_{\{X^y(\omega)\in B\}}$ is $\mathcal{I}\boxtimes\mathcal{F}$-measurable and bounded, it is integrable.
    By the definition of a Fubini extention we have that $y\mapsto \mathcal{L}(X^y)(B)$ is integrable and hence $\mathcal{I}$-measurable. 
    The second claim is immediate from Lemma \ref{lemma:nu-defines-a-measure}.
    
    In view of Proposition \ref{prop:graphon.integration}, the integrability property follows from    
    \begin{equation*}
    \begin{split}
        \int_I\|\nu_{s}^{ \mathcal{L}(X),x}\|_{\mathrm{BL}} \lambda(dx) &= \int_I\sup_{\phi \in 1\text{-Lip}_b} \int_{\mathbb{R}^n} \phi(z) \nu_{s}^{ \mathcal{L}(X) ,x}(dz) \lambda(dx)\\
        & \leq \int_I \sup_{\phi \in 1\text{-Lip}_{b}} \int_I  g(x,y)\mathbb{E}\left[|\phi(X_s^{y})|\right] \lambda(dy)\lambda(dx)\\
        & \leq \int_I\int_I  g(x,y) \lambda(dy)\lambda(dx) <\infty,
    \end{split}
    \end{equation*}
    where we used that $g$ is non-negative and square integrable on $I\times I$.
\end{proof}

\subsection{Proof of Theorem \ref{thm:existence-uniqueness}}
\label{sec:proof-existence-uniqueness}

The proof of the existence result will use a Banach fixed point argument in the Banach space $\S^2_{\boxtimes}(I\times \Omega, \mathcal{C})$. 
In fact, let us be given a process $Z\in \S^2_{\boxtimes}(I\times \Omega, \mathcal{C})$.
Denote again by $Z$ the version defined in \eqref{eq:version} and consider the mapping $\mu = \mathcal{L}({Z}) : I \rightarrow \mathcal{P}_2(\mathcal{C})$.
The SDE
\begin{equation}
\begin{split}
\label{eq:graphon-SDE-mu.proof}
    dX_t^{\mu,x} &=  b\bigg(t,X_t^{\mu,x}, \int_I g(x,y)\mu_t^y \lambda(dy)\bigg)dt + \sigma\bigg(t,X_t^{\mu.x}, \int_I g(x,y)\mu_t^y \lambda(dy)\bigg)dB_t^x,\quad
    X_0^{\mu,x}= \xi^x,
\end{split}
\end{equation}
parameterized by $\mu$ admits a unique solution.
Thus by Proposition \ref{prop:measurable.SDE}, the map
\begin{equation}
    \psi :  \S^2_{\boxtimes}(I\times \Omega, \mathcal{C}) \rightarrow \S^2_{\boxtimes}(I\times \Omega, \mathcal{C}), \quad Z \mapsto X^{\mu,\cdot},\quad \text{with} \quad \cL(Z) =\mu, 
\end{equation}
 is well-defined.
It remains to show that this mapping admits a unique fixed point in $\S^2_{\boxtimes}(I\times \Omega,\cC)$. 
We will show that there is $k\ge1$ such that the $k$-fold composition $\psi^k = \psi\circ \dots\circ\psi$ is a contraction. 
Then, by Banach's fixed point theorem we will conclude the proof.

\begin{proof}[Proof of Theorem \ref{thm:existence-uniqueness}]
    Let $Z,Z'\in \S^2_{\boxtimes}(I\times \Omega, \mathcal{C})$ be fixed. For simplicity denote $\mu = \mathcal{L}(Z)$ and $\mu' = \mathcal{L}(Z')$. Also let $\psi(Z)=X^{\mu,\cdot}$ and $\psi(Z') = X^{\mu',\cdot}$ with $X^{\mu,\cdot}$ and $X^{\mu',\cdot}$ satisfying equation \eqref{eq:graphon-SDE-mu.proof} with $\mu$ and $\mu'$ respectively.
    Thus, for $\lambda$-almost every $x\in I$, by Lipschitz continuity of $b,\sigma$ and Burkholder-Davis-Gundy inequality (recalling that $ \int_0^s \sigma(u,X_u^{\mu,x},\nu_u^{\mu,x})-  \sigma(u,X_u^{\mu',x},\nu_u^{\mu',x}) dB^x_u$ is a square integrable martingale in view of the growth condition of $\sigma$ and Lemma \ref{eq:Lemma.integ.graphon.integr}) we have that
    \begin{equation*}
    \begin{split}
        \mathbb{E}\bigg[\sup_{s\in [0,t]}|X^{\mu,x}_s-X^{\mu',x}_s|^2\bigg]
         & \leq  C\bigg(\int_0^t \mathbb{E}\Big[\sup_{v\in [0,u]}|X_v^{\mu,x}-X_v^{\mu',x}|^2 \Big] du  +  d_{\mathrm{BL}}^2(\nu_u^{\mu,x},\nu_u^{\mu',x})du \bigg).
    \end{split} 
    \end{equation*}
    Then, by Gronwall's inequality and Proposition \ref{prop:graphon.integration}, we have
    \begin{equation*}
    \begin{split}
        \mathbb{E}\bigg[\sup_{s\in [0,t]}|X^{\mu,x}_s-X^{\mu',x}_s|^2\bigg] 
        & \leq   C \int_0^t d_{\mathrm{BL}}^2(\nu_u^{\mu,x},\nu_u^{\mu',x}) du \\
         & =  C \int_0^t \bigg(\sup_{\phi\in 1\text{-Lip}_b} \int_I g(x,y)\mathbb{E}\Big[\phi(Z_u^{y})-\phi((Z')_u^{y})\Big]\lambda(dy)\bigg)^2 \\
         & \leq  C \bigg(\int_I g(x,y)^2 \lambda(dy)\bigg) \int_0^t\int_I \mathbb{E}\Big[\big|Z_u^{y}-(Z')_u^{y}\big|^2\Big]\lambda(dy) du ,\\
    \end{split}
    \end{equation*}
    where in the last line we used Cauchy-Schwarz inequality and Jensen's inequality. 
    Finally, to recover the expectation $\mathbb{E}^{\boxtimes}$ we integrate with respect to $\lambda(dx)$ on both sides and use again Cauchy-Schwarz inequality to obtain
    \begin{equation*}
    \begin{split}
        \mathbb{E}^{\boxtimes}\bigg[\sup_{s\in [0,t]}|\psi(Z)_s-\psi(Z')_s|^2\bigg] 
        &\leq  C \left(\int_{I\times I} g(x,y)^2 \lambda(dy)\lambda(dx)\right) \int_0^t \mathbb{E}^\boxtimes\left[\left|Z_u-Z'_u\right|^2\right] du \\ 
        & \leq  C  \int_0^t \mathbb{E}^\boxtimes\Big[\sup_{v\in [0,s]}\left|Z_v-Z'_v\right|^2\Big] ds .
    \end{split}
    \end{equation*}

    We will now iterate this bound $k$ times to obtain an estimation for $\psi^k$:
    \begin{equation*}
    \begin{split}
        \mathbb{E}^{\boxtimes}\bigg[\sup_{s\in [0,T]}|\psi^k(Z)_s-\psi^k(Z')_s|^2\bigg] 
        &\leq  C^k \int_0^T \int_0^{u_1} \dots \int_0^{u_{k-1}}  \mathbb{E}^\boxtimes\bigg[\sup_{v\in [0,u_k]}\left|Z_v-Z'_v\right|^2\bigg]du_k \dots du_1 \\
        &\leq   \frac{C^kT^k}{k!}  \mathbb{E}^{\boxtimes}\bigg[\sup_{v\in [0,T]}\left|Z_v-Z'_v\right|^2\bigg].
    \end{split}
    \end{equation*}
    Since for $k$ large enough we have $\frac{C^kT^k}{k!}<1$, it follows that $\psi^k$ is a contraction.
    Therefore, we conclude the proof using Banach fixed point theorem.
\end{proof}

\subsection{Stability of graphon SDEs}
We conclude this section by discussing stability of the graphon SDE with $L^2$-graphons.
This result will serve us well when proving the propagation of chaos result in the next section.
\begin{proposition}
\label{prop:stability-of-g-integral}
    Let Conditions \ref{cond:initial-conditions} and \ref{cond:coefficients} be satisfied.
    Let $g$ and $h$ be two square integrable graphons, and denote by $(X^{x,g})_{x\in I}$ and $(X^{x,h})_{x\in I}$ the solutions of the SDE \eqref{eq:graphon-SDE} with respective graphons $g$ and $h$. 
    Then, it holds that 
    \begin{equation*}
        \int_I \mathbb{E}\Big[\sup_{t\in [0,T]}|X^{y,g}_t-X^{y,h}_t|^2\Big] \lambda(dy) \le  C\|g-h\|^2_{2},
    \end{equation*}
    for a constant $C$ that depends only on the Lipschitz constants of $b$ and $\sigma$ and on $\|h\|_2$.
\end{proposition}

\begin{proof}
    We will use the notation $\nu^{y,g}_{X_t}:=\int_Ig(x,y)\cL(X^{g,y}_t)\lambda(dy)$.
    Since $X^{x,h}$ and $X^{x,g}$ satisfy Equation \eqref{eq:graphon-SDE} for the respective graphons $h$ and $g$, using Burkholder-Davis-Gundy inequality and Condition \ref{cond:coefficients} we have:
    \begin{equation}
    \label{eq:stability.estim}
    \begin{split}
        \int_I \mathbb{E}\Big[\sup_{s\in [0,T]}|X^{y,g}_{s} - X^{y,h}_s|^2\Big]\lambda(dy)
        & \leq  C \int_I \mathbb{E}\bigg[ \int_0^T |X^{y,g}_u-X^{y,g}_u|^2 + d^2_{\mathrm{BL}}(\nu^{y,g}_{X_u}, \nu^{y,h}_{X_u}) du\bigg] \lambda(dy) \\
        & \leq C \int_0^T \int_I\Big( \E \Big[\sup_{r\in [0,u]}\big| X_r^{y,g}-X_r^{y,h}\big|^2\Big] +  d_{\mathrm{BL}}^2(\nu^{y,g}_{X_u}, \nu^{y,h}_{X_u})  \Big) \lambda (dy) du. 
    \end{split}
    \end{equation}
    Observe that by definition of the bounded Lipschitz distance we have
    \begin{equation*}
    \begin{split}
        \int_I  d^2_{\mathrm{BL}}(\nu^{y,g}_{X_u}, \nu^{y,h}_{X_u})  \lambda(dy) 
        & =  \int_I \sup_{\phi\in 1\text{-Lip}_{b}} \left(\int_I \mathbb{E}[\phi(X^{z,g}_u)]g(y,z) - \mathbb{E}[\phi(X^{z,h}_u)]h(y,z)  \lambda(dz)\right)^2 \lambda(dy)\\
         & \leq 2 \int_I  \sup_{\phi\in 1\text{-Lip}_{b}}  \left( \int_I \mathbb{E}[ \phi(X^{z,g}_u)]\Big(g(y,z) - h(y,z)\Big)\lambda(dz) \right)^2  \\
        & \quad + \sup_{\phi\in 1\text{-Lip}_{b}}\left( \int_I \mathbb{E}[  \phi(X^{z,g}_u)] - \mathbb{E}[\phi(X^{z,h}_u)] h^2(y,z)\lambda(dz)\right)^2 \lambda(dy)\\
        & \leq \int_I  \sup_{\phi\in 1\text{-Lip}_{b}}  \int_I \mathbb{E}[\left| \phi(X^{z,g}_u)\right|^2]\lambda(dz) \int_I \Big(g(y,z) - h(y,z)\Big)^2\lambda(dz)\\
        & \quad +  \int_I \sup_{\phi\in 1\text{-Lip}_{b}} \int_I \left(\mathbb{E}[  \phi(X^{z,g}_u)] - \phi(X^{z,h}_u)] \right)^2 \lambda(dy) \int_{I} h^2(y,z)\lambda(dz) \lambda(dy) \\
        & \le \|g-h\|_2^2 + \|h\|_2^2 \int_I \mathbb{E}\Big[\sup_{r\in [0,u]}\big|X^{z,g}_r - X^{z,h}_r\big|^2\Big] \lambda(dy).
    \end{split}
    \end{equation*}
    Notice that we can take the square inside the supremum in the first line because if $\phi$ is $1$-Lip$_{b}$, then $-\phi$ is also in $1$-Lip$_{b}$ and thus we could use $-\phi$ to get a larger value, so the supremum is always positive. Furthermore, we used Cauchy-Schwarz inequality and the fact that $h$ is square integrable. 

    Coming back to equation \eqref{eq:stability.estim}, we then have that 
    \begin{align*}
        \int_I \mathbb{E}\Big[\sup_{s\in [0,T]}|X^{y,g}_{s} - X^{y,h}_s|^2\Big]\lambda(dy) & \leq  C \int_0^T \int_I \mathbb{E}\Big[  \sup_{r\in [0,u]}|X^{y,g}_r-X^{y,h}_r|^2\Big]\lambda(dy) du +  C\|g-h\|_{2}^2.   
    \end{align*}   
    Applying Gronwall's inequality allows to finish the proof.
\end{proof}

\section{Propagation of chaos and related convergence results}
\label{sec:proof-convergence}

Let us now turn our attention to the proof of the main propagation of chaos results stated in the introduction.
We will also derive a few by-products of the main results.
The proofs build upon several steps and is thus divided into subsections.
We begin by deriving yet another stability result for the graphon SDE with respect to the graphon.
In contrast to Proposition \ref{prop:stability-of-g-integral}, here we look at stability with respect to another distance. 
Then, we will prove moment estimates.
In the ensuing part we clarify that the usual synchronous coupling argument of \cite{sznitman} extends to graphon SDEs.
We will first prove convergence of the particle system when the underlying graphon $g$ is bounded, then use the stability results to extend the result to the case of $L^2$-graphons.

We will present the proof only for the case where Condition \ref{cond:graphon-cond-1} is satisfied.
The case where Condition \ref{cond:graphon-cond-2} holds follows exactly the same but replacing $g_N$ by $g$ in all the computations.

\subsection{Stability and moment estimates}

To begin with, let us clarify that the particle system considered in this paper is well-defined.
\begin{proposition}
\label{prop:existence-part-system}
    Assume that for each $i=1,\dots,N$,  $\E[|\xi^i|^2]<\infty$.
    Let Condition \ref{cond:coefficients} and Condition \ref{cond:graphon-cond-1} be satisfied and assume that the graphon $g$ is square integrable. 
    Then for every $N\ge1$ there exist unique solutions to the systems of SDEs \eqref{eq:intro.part.system} and \eqref{eq:part.system.B}.
\end{proposition}

\begin{proof}
    The proof is obvious and follows from the standard theory of SDEs with (random) Lipschitz coefficients.
    It suffices to verify that the functions 
    $$
        (x_1\dots,x_N)\mapsto b\Big(t, x_i, \frac{1}{N}\sum_{j=1}^N\frac{\zeta^N_{ij}}{\beta_N}\delta_{x_j} \Big)
        \quad \text{and}\quad 
        (x_1\dots,x_N)\mapsto \sigma\Big(t, x_i, \frac{1}{N}\sum_{j=1}^N\frac{\zeta^N_{ij}}{\beta_N}\delta_{x_j} \Big)
    $$
    are Lipschitz continuous.
    This can be verified using Condition \ref{cond:coefficients} as well as boundedness of the random variables $\zeta^N_{ij}$ and definition of $d_{\mathrm{BL}}$.
    Notice that the Lipschitz-constant will depend on $N$.
\end{proof}

Next, we show that the moment bounds on the initial position of the particle system imply a (uniform in $N$) moment bound on the whole particle system.
In what follows, let $(X^{x,g})_{x\in I}$ and $(X^{i,N})_{i=1,\dots,N}$ be respectively the solution of the graphon SDE \eqref{eq:graphon-SDE} and of the particle system \eqref{eq:part.system.B}.
We will use the notation 
\begin{equation}
\label{eq:def.Xi.nu}
    X^{i,g}:= X^{x_i,g}\quad \text{and} \quad \nu^{i,g}_{X}(\cdot):=\int_Ig(x_i,y)\cL(X^{y,g})(\cdot)\lambda(dy), \quad i=1,\dots,N.
\end{equation}

\begin{lemma}
\label{lemma:L2-boundness-solutions}
    Assume Conditions \ref{cond:initial-conditions}, \ref{cond:coefficients} and \ref{cond:graphon-cond-1} hold. 
    Moreover, assume that $g$ is bounded and there exists a constant $C'>0$ independent of $N$ such that 
    $$
        \sup_{1\le i\leq N}\E\big[|\xi^i|^k\big]\leq C'
    $$
    for some integer $k\ge 2$. 
    Then, there exists a constant $C>0$ independent of $N$ such that 
    \begin{equation*}
        \sup_{1\leq i\leq N}\E\Big[\sup_{t\in [0,T]}|X^{x_i,g}_t|^{k}\Big] \leq C.
    \end{equation*}
    Moreover, if $g_N$ is  uniformly bounded, then there exists $C>0$ independent of $N$ such that
    \begin{equation*}
        \sup_{1\leq i\leq N} \E\Big[\sup_{t\in [0,T]}|X^{i,N}_t|^{k}\Big] \leq C.
    \end{equation*}
\end{lemma}

\begin{proof}
    Fix $N\in \N$ and $0\leq i\leq N$.
    Using that $X^{i,g}_t$ satisfies the graphon SDE \eqref{eq:graphon-SDE}, by Condtion \ref{cond:coefficients} and Burkholder-Davis-Gundy inequality we have
    \begin{equation*}
    \begin{split}
        \E\Big[\sup_{s\in [0,t]}|X^{i,g}_s|^{k}\Big] 
        & \leq  C \E\bigg[1+|\xi^i|^{k} + \int_0^t \sup_{r\in [0,u]}\left|X^{i,g}_r\right|^{k} + \sup_{\phi \in 1\text{-Lip}_b} \bigg(\int_{\R^n} \phi(z)\nu^{i,g}_{X_u}(dz)\bigg)^{k}du \bigg]\\
        & \quad + C \E\bigg[ \bigg(\int_0^t |X^{i,g}_u|^2 + \|\nu_{X_u}^{i,g}\|_{\mathrm{BL}}^{2} du\bigg)^{k/2}\bigg]\\
        & \leq  C \E\bigg[1+|\xi^i|^{k} + \int_0^t \sup_{r\in [0,u]}\left|X^{i,g}_r\right|^{k} + \sup_{\phi \in 1\text{-Lip}_b} \bigg(\int_I\E[\phi(X^{y,g}_u) ]g(x_i,y)\lambda(dy)\bigg)^{k} du \bigg] \\
        & \leq  C \E\bigg[1+|\xi^i|^{k} + \int_0^t \sup_{r\in [0,u]}\left|X^{i,g}_r\right|^{k} + \|g\|_{\infty}^{k} \bigg], \\
    \end{split}
    \end{equation*}
    where in the second line we used the fact that $k/2\geq 1$ to apply Hölder's inequality. Then, using Gronwall's inequality we get the result. 

    The proof of the moment bound of the particle system is a bit harder because $\beta_N$ can become arbitrarily small.
    Observe that if Condition \ref{cond:graphon-cond-1} holds, then we have that $\E[\zeta^{N}_{ij}]= \beta_Ng_N(x_i,x_j)$.
    Notice as above that 
    \begin{equation}
    \label{eq:sol-bounded-intermetiate}
    \begin{split}
         \E\Big[\sup_{s\in [0,t]}|X^{i,N}_t|^{k}\Big] 
         & \leq C\E\bigg[1+|\xi^i|^{k} + \int_0^t \sup_{r\in [0,u]}|X^{i,N}_r|^{k} + \sup_{\phi \in 1\text{-Lip}_b}\Big( \frac{1}{N}\sum_{j=1}^N \frac{\zeta^{N}_{ij}}{\beta_N}\phi(X^{j,N}_u)\Big)^{k} du \bigg] \\
         & \leq C\E\bigg[1+|\xi^i|^{k} + \int_0^t \sup_{r\in [0,u]}|X^{i,N}_r|^{k} du + \bigg( \frac{1}{N}\sum_{j=1}^N \frac{\zeta^{N}_{ij}}{\beta_N}-g_N(x_i,x_j)\bigg)^{k}  \\
         &  \quad + \bigg( \frac{1}{N}\sum_{j=1}^N g_N(x_i,x_j)\bigg)^{k}  \bigg].
    \end{split}
    \end{equation}

    Notice that the third term converges to $0$. We get this result by expanding the $k$-th power:
    \begin{equation*}
        \frac{1}{N^{k}\beta_N^{k}}\E\bigg[\Big(\sum_{j=1}^N\Gamma_{ij}\Big)^{k}\bigg] = \frac{1}{N^{k}\beta_N^{k}}\E\Bigg[\sum_{\substack{k_1,\dots,k_N\\ k_1+\dots+k_N = k}}\Gamma_{i1}^{k_1}\dots\Gamma_{iN}^{k_N}\Bigg]\quad \text{with } \Gamma_{ij}=\zeta_{ij}^N - \beta_N g_N(x_i,x_j).
    \end{equation*}
    If $k_j=1$ we obtain that $\E[\Gamma_{ij}]=0$. Hence, we only need to sum over the combinations such that $k_1,\dots, k_N \neq 1$. Then, using that the centered moments of a Bernoulli random variable are bounded by its expectation we can write
    \begin{equation*}
    \begin{split}
        \frac{1}{N^{k}\beta_N^{k}}\E\bigg[\Big(\sum_{j=1}^N\Gamma_{ij}\Big)^{k}\bigg]&  = \frac{1}{N^{k}\beta_N^{k}} \E\left[\sum_{j=1}^{\lfloor\frac{k}{2}\rfloor} \sum_{\substack{k_1,\dots,k_N\neq 1\\ k_1+\dots+k_N = k\\ |\{l:k_l\neq 0\}|=j}}\Gamma_{i1}^{k_1}\dots\Gamma_{iN}^{k_N}\right]\\
        & \leq \frac{C}{N^{k}\beta_N^{k}} \sum_{j=1}^{\lfloor\frac{k}{2}\rfloor} \sum_{\substack{k_1,\dots,k_N\neq 1\\ k_1+\dots+k_N = k \\ |\{l:k_l\neq 0\}|=j}}\beta_N^j \\
        & = \frac{C}{N^{k}\beta_N^{k}} \sum_{j=1}^{\lfloor\frac{k}{2}\rfloor}\binom{N}{j}\binom{k+j-1}{j-1}\beta_N^j \\
        & \leq \frac{C}{N^{k}\beta_N^{k}} \sum_{j=1}^{\lfloor\frac{k}{2}\rfloor} \binom{k+j-1}{j-1}N^{j}\beta_N^j \\
        & \leq \frac{C}{N^{\lfloor\frac{k}{2}\rfloor}\beta_N^{\frac{k}{2}}} \sum_{j=1}^{\frac{k}{2}} \binom{k+j-1}{j-1}.
    \end{split}
    \end{equation*}

    Thus, going back to equation \eqref{eq:sol-bounded-intermetiate} and using that $g_N$ is uniformly bounded we obtain
    \begin{equation*}
    \begin{split}
        \E\Big[\sup_{s\in [0,t]}|X^{i,N}_t|^{k}\Big] & \leq C\E\big[|\xi^i|^{k}\big] + C \int_0^t\E\Big[ \sup_{r\in [0,u]}|X^{i,N}_r|^{k} \Big] du+ \frac{C}{N^{\frac{k}{2}}\beta_N^{\frac{k}{2}}} \sum_{j=1}^{\lfloor\frac{k}{2}\rfloor} \binom{k+j-1}{j-1} + C.  \\
    \end{split}
    \end{equation*}
    Using Gronwall inequality and that $N\beta_N\rightarrow \infty$, we get the result.
\end{proof}

\begin{lemma}
\label{lemma:L2-boundness-solutions-g-unbounded-part-1}
    Assume Conditions \ref{cond:initial-conditions}, \ref{cond:coefficients} and Condition \ref{cond:graphon-cond-1} hold. 
    Moreover, assume that $g$ is $k$-integrable and there exists a constant $C'>0$ independent of $N$ such that 
    $$
        \sup_{1\le i\leq N}\E\big[|\xi^i|^k\big]\leq C'
    $$
    for some integer $k\ge 2$. 
    Then, there exists a constant $C>0$ independent of $N$ such that 
    \begin{equation*}
        \frac{1}{N}\sum_{i=1}^N\E\Big[\sup_{t\in [0,T]}|X^{x_i,g}_t|^{k}\Big] \leq C.
    \end{equation*}
\end{lemma}

\begin{proof}
    For the proof of Lemma \ref{lemma:L2-boundness-solutions} we obtain: 
    \begin{equation*}
    \begin{split}
        \frac{1}{N}\sum_{i=1}^N\E\Big[\sup_{s\in [0,t]}|X^{i,g}_s|^{k}\Big] 
        & \leq  \frac{C}{N}\sum_{i=1}^N \E\bigg[1+|\xi^i|^{k} + \sup_{\phi \in 1\text{-Lip}_b} \bigg(\int_I\E[\phi(X^{y,g}_u) ]g(x_i,y)\lambda(dy)\bigg)^{k} du \bigg] \\
        & \leq  C\E\bigg[1+\frac{1}{N}\sum_{i=1}^N |\xi^i|^{k} +  \|g\|^k_k du \bigg]. \\
    \end{split}
    \end{equation*}
\end{proof}

We now present a stability result for graphon SDE in terms of the average.
Recall the function $f_N$ was defined as
\begin{equation*}
    f_N(x) := \sum_{i=1}^{N-1} x_i {\mathbf{1}_{\{x\in [x_i,x_{i+1})\}}} + x_N \mathbf{1}_{\{x\in [x_N,1]\}}.
\end{equation*}

\begin{proposition}
\label{prop:stability-of-g-average}
    Let Conditions \ref{cond:initial-conditions}, \ref{cond:coefficients} and \ref{cond:almost-continuity} be satisfied. 
    For every $N\in \N$ consider an arbitrary sequence of points $x_1=0 < x_2 < \dots < x_N < 1$. 
    For any two square integrable graphons $g$ and $h$ there exists a constant $C>0$ independent of $N,h$ but dependent on $\|g\|_2$ such that 
     \begin{equation}
     \label{eq:average-stability-discrete}
     \begin{split}
        \frac{1}{N}\sum_{i=1}^N \mathbb{E}\bigg[\sup_{t\in [0,T]}|X^{i,g}_t-X^{i,h}_t|^2\bigg] & \leq  \frac{C}{N}\sum_{i=1}^N \int_I\left(h(x_i,y)-g(x_i,y)\right)^2\lambda(dy)\\
        & \quad +  \frac{C}{N}\sum_{i=1}^N\int_I|g(x_i,y)|^2\lambda(dy)  \|g-h\|_2^2.
    \end{split}
    \end{equation}
    If Condition \ref{cond:graphon-cond-1} is satisfied and $g$ is bounded,
    then
    \begin{equation}
    \label{eq:average-stability-g_N}
        \frac{1}{N}\sum_{i=1}^N \mathbb{E}\bigg[\sup_{t\in [0,T]}|X^{i,g}_t-X^{i,g_N}_t|^2\bigg] 
        \xrightarrow{N\rightarrow \infty} 0.
    \end{equation}
\end{proposition} 

\begin{remark}
    Notice that if we want to consider the $L^1$ norm then equation \eqref{eq:average-stability-discrete} will take the form: 
    \begin{equation*}
         \frac{1}{N}\sum_{i=1}^N \mathbb{E}\bigg[\sup_{t\in [0,T]}|X^{i,g}_t-X^{i,h}_t|\bigg] \leq  \frac{C}{N}\sum_{i=1}^N \int_I\left|h(x_i,y)-g(x_i,y)\right|\lambda(dy) +  \frac{C}{N}\sum_{i=1}^N\int_I|g(x_i,y)| \lambda(dy)  \|g-h\|_2.
    \end{equation*}
\end{remark}

In the proof of this proposition, Condition \ref{cond:almost-continuity} will be used only to claim the convergence in the last step. 
In particular, we will use it through the following lemma :

\begin{lemma}
\label{lemma:cond-implies-convergence}
    If $g$ is a graphon that satisfies Condition \ref{cond:almost-continuity} and there exists a square integrable function $h$ that dominates $f\circ (f_N, \mathrm{id})$\footnote{in particular, this assumption holds when $g$ is bounded.}, then
    \begin{equation*}
        \lim_{N\rightarrow 0}\int_I\int_I |g(f_N(x),y)-g(x,y)|^2\lambda(dx)\lambda(dy) = 0.
    \end{equation*}
\end{lemma}

\begin{proof}
    For every $x\in I$ Condition \ref{cond:almost-continuity} gives a set $I_x$ with $\lambda(I_x)=0$ such that $g$ is continuous at $(x,y)$ for all $y \in I^c_x$.
    Thus,
    \begin{equation*}
    \begin{split}
        \int_I\int_I |g(f_N(x),y)-g(x,y)|^2\lambda(dy)\lambda(dx) & = \int_I\int_{I_x} |g(f_N(x),y)-g(x,y)|^2\lambda(dy)\lambda(dx) \\
        & \quad +\int_I\int_{I_x^c} |g(f_N(x),y)-g(x,y)|^2\lambda(dy)\lambda(dx) \\
        & = \int_I\int_{I_x^c} |g(f_N(x),y)-g(x,y)|^2\lambda(dy)\lambda(dx).
    \end{split}
    \end{equation*}
    Then, we have that $g$ is continuous at every pair $(x,y)$ in the domain of integration.
    Hence, because $f_N(x) \rightarrow x$ uniformly as $N\to\infty$, we have $g(f_N(x),y) \rightarrow g(x,y)$. 
    Since $g$ is bounded, by dominated convergence theorem we have that
    \begin{equation*}
        \lim_{N\rightarrow \infty} \int_I\int_{I_x^c} |g(f_N(x),y)-g(x,y)|^2\lambda(dy)\lambda(dx)  = 0.
    \end{equation*}   
\end{proof}

\begin{proof}[Proof of Proposition \ref{prop:stability-of-g-average}]
By the same argument as in the proof of Proposition \ref{prop:stability-of-g-integral} we have:
    \begin{equation}
    \label{eq:lemma:average-stability-estimate-0}
    \begin{split}
        \frac{1}{N}\sum_{i=1}^N\mathbb{E}\Big[\sup_{s\in [0,T]}|X^{i,g}_{s} - X^{i,h}_s|^2\Big] 
        & \leq C \int_0^T  \frac{1}{N}\sum_{i=1}^N\E \Big[\sup_{r\in [0,u]}\big| X_r^{i,g}-X_r^{i,h}\big|^2\Big] +  \frac{1}{N}\sum_{i=1}^N d_{\mathrm{BL}}^2(\nu^{i,g}_{X_u}, \nu^{i,h}_{X_u}) du. \\
    \end{split}
    \end{equation}
    To bound the second term on the right hand side of equation \eqref{eq:lemma:average-stability-estimate-0} we use the definition of $d_{\mathrm{BL}}$, boundedness of $g$ and Proposition \ref{prop:stability-of-g-integral} 
    \begin{equation}
    \label{lemma:estimate-1-average-stability-lemma}
    \begin{split}
          \frac{1}{N}\sum_{i=1}^N d^2_{\mathrm{BL}}(\nu^{i,g}_{X_u}, \nu^{i,h}_{X_u})  
         & \leq  \frac{2}{N}\sum_{i=1}^N \sup_{\phi\in 1\text{-Lip}_{b}}  \left( \int_I \mathbb{E}\left[ \phi(X^{y,h}_u)\right]\Big(h(x_i,y) - g(x_i,y)\Big)\lambda(dy) \right)^2  \\
        & \quad + \frac{2}{N}\sum_{i=1}^N\sup_{\phi\in 1\text{-Lip}_{b}}\left( \int_I \mathbb{E}\big[  \phi(X^{y,h}_u)- \phi(X^{y,g}_u)\big] g(x_i,y)\lambda(dy)\right)^2 \\
        & \le \frac{2}{N}\sum_{i=1}^N \int_I \Big(h(x_i,y) - g(x_i,y)\Big)^2\lambda(dy)\\
        &\quad + \frac2N\sum_{i=1}^N\int_I|g(x_i,y)|^2\lambda(dy)   \int_I\E\Big[\sup_{t\in[0,T]}|X^{y,h}_t - X^{y,g}_t|^2 \Big]\lambda(dy)\\
        &\le  \frac{2}{N}\sum_{i=1}^N \int_I \Big(h(x_i,y) - g(x_i,y)\Big)^2\lambda(dy) + \frac{C}{N}\sum_{i=1}^N\int_I|g(x_i,y)|^2\lambda(dy)  \|g-h\|_2^2,
    \end{split}
    \end{equation}
        where $id$ is the identity function $id(x) = x$.
    Combining equations \eqref{eq:lemma:average-stability-estimate-0} and \eqref{lemma:estimate-1-average-stability-lemma} and applying Gronwall's inequality we get \eqref{eq:average-stability-discrete}.

    To show equation \eqref{eq:average-stability-g_N} we notice that by introducing $f_N$ and letting $h = g_N$ in equation \eqref{lemma:estimate-1-average-stability-lemma} we obtain
    \begin{equation}
    \label{eq:from.average.to.integral}
    \begin{split}
        \frac{1}{N}\sum_{i=1}^N d^2_{\mathrm{BL}}(\nu^{i,g}_{X_u}, \nu^{i,h}_{X_u})  
         & \leq \frac{1}{N}\sum_{i=1}^N \int_I \Big(g_N(x_i,y) - g(x_i,y)\Big)^2\lambda(dy) + C\|g\circ (f_N,\mathrm{id})\|_2^2 \|g-g_N\|_2^2 \\
         &  =\int_I \int_I \Big(g_N(x,y) - g(f_N(x),y)\Big)^2\lambda(dx)\lambda(dy)  + C\|g\circ (f_N,\mathrm{id})\|_2^2 \|g-g_N\|_2^2 \\
        & \leq C\big( \|g\circ (f_N, \mathrm{id}) - g\|^2_2 + \|g-g_N\|^2_2\big) +  + C\|g\circ (f_N,\mathrm{id})\|_2^2 \|g-g_N\|_2^2. 
    \end{split}
    \end{equation}
    Finally using Lemma \ref{lemma:cond-implies-convergence} we obtain  \eqref{eq:average-stability-g_N}.
\end{proof}

\subsection{Synchronous coupling of graphon SDEs}
The proof of Theorem \ref{thm:convergence-bounded-graphon-L1} will make use of the synchronous coupling idea of \citet{sznitman} adapted to the present case of particles in heterogenous interaction.
The construction of a synchronous coupling of the graphon SDE \eqref{eq:graphon-SDE} is based on the following lemma:
\begin{lemma}
\label{lemma:sol-sde-are-epi}
    Let the assumptions of Theorem \ref{thm:existence-uniqueness} be satisfied and let $X\in L^2_{\boxtimes}(I\times \Omega,\cC)$ be the unique strong solution of the SDE \eqref{eq:graphon-SDE}. 
    Then, $(X^y)_{y\in I}$ forms a family of essentially pairwise independent processes.
\end{lemma}

\begin{proof}
    Let $x,y \in I$ be fixed and consider $\tilde B^x:=(\xi^x,B^x)$ and $\tilde B^y= (\xi^y, B^y)$.
    Assume that $\tilde B^x$ and $\tilde B^y$ are independent. 
    We know by Theorem \ref{thm:existence-uniqueness} that $X^x$ and $X^y$ are unique strong solutions of the SDE \eqref{eq:graphon-SDE-mu} with $\nu^x$ therein replaced by $\nu^{x,g}_X:=\int_Ig(x,z)\cL(X^z)\lambda(dz)$ and $\nu^{y,g}_{X}:=\int_{I}g(y,z)\cL(X^z)\lambda(dz)$, respectively.  
    Thus, by \cite[Definition $4.1.6$]{Ikeda-Watanabe-SDE-and-diffusion-processes}, there exists a measurable function $F_{\nu^{z,g}_X}:\R^n\times \cC \to\R^n$ such that $X^{z} = F_{\nu^{z,g}_X}(\xi^z, B^z) = F_{\nu^{z,g}_X}(\tilde B^z)$ for $z=x,y$. Then, we get that $X^{x}$ and $X^{y}$ are independent. 
    Now, because the family $(\tilde B^x)_{x\in I}$ is essentially pairwise independent, it follows that $(X^x)_{x\in I}$ is also essentially pairwise independent.
\end{proof}

By Lemma \ref{lemma:sol-sde-are-epi}, the family $(X^{i,g})_{i=1,\dots,N}$ is independent. 
We also have the following immediate result:
\begin{lemma}
\label{lemma:independence-zeta-X}
   For every $1\leq i,j \leq N$ and $x\in I$. 
   We have that $\zeta_{ij}^N$ and $X^{x,g}$ are independent.
\end{lemma}

\begin{proof}
    Fix $1\leq i,j\leq N$ and $x\in I$. 
    As in the proof of Lemma $\ref{lemma:sol-sde-are-epi}$, there exists a measurable function $F_{\nu^{x,g}_X}: \mathbb{R}^n \times \mathcal{C} \rightarrow \cC$ such that $X^x = F_{\nu^{x,g}_X}(\xi^x, B^x)$. 
    As $\zeta_{ij}^N$ and $\Tilde{B}^x$ are independent then $X^{x,g}$ and $\zeta_{ij}^N$ are independent. 
\end{proof}

\subsection{Convergence in the bounded case}
With the above preliminaries out of the way, we are now ready to present the proof of the first propagation of chaos result. Recall that the measure $m^{i,N}_{\X_t}$ was defined as
\begin{equation*}
    m_{\mathbb{X}_t}^{i,N} := \frac{1}{N} \sum_{j=1}^N \frac{\zeta_{ij}^N}{\beta_N} \delta_{X_t^{j,N}}.
\end{equation*}

The whole of this subsection is dedicated to the proof of Theorem \ref{thm:convergence-bounded-graphon-L1}, the assumptions of this result are thus in force in the rest of the subsection.
Let us introduce the vector $\X^g$  and the (weighted) empirical measure $m_{\X^g}^{i,N}$ defined as
\begin{equation}
\label{eq:def-coupled-measure}
    m_{\X^g}^{i,N}: = \frac1N\sum_{j=1}^N\frac{\zeta^N_{ij}}{\beta_N}\delta_{X^{j,g}}\quad \text{and}\quad \X^g:= (X^{x_1,g},\dots, X^{x_N,g}).
\end{equation}
Using the dynamics of $X^{i,N}$ and $X^{i,g}$, Condition \ref{cond:coefficients} and Burkholder-Davis-Gundy inequality we have
    \begin{equation}
    \begin{split}
        \frac{1}{N}\sum_{i=1}^N \E\bigg[\sup_{s\in[0,T]}\big|X^{i,N}_s-X^{i,g}\big|\bigg]
        & \leq   \frac{C}{N}\sum_{i=1}^N   \mathbb{E}\bigg[ \int_0^T |X^{i,N}_u-X^{i,g}_u| + d_{\mathrm{BL}}( m^{i,N}_{\X_u}, \nu^{i,g}_{X_u}) du\bigg] \\
        &\quad +  \frac{C}{N}\sum_{i=1}^N  \mathbb{E}\bigg[ \bigg(\int_0^T |X^{i,N}_u-X^{i,g}_u|^2 du\bigg)^{1/2}\bigg]\\
        & \leq   \frac{C}{N}\sum_{i=1}^N   \mathbb{E}\bigg[ \int_0^T |X^{i,X}_u-X^{i,g}_u| + d_{\mathrm{BL}}(m^{i,N}_{\X_u}, \nu^{i,g}_{X_u}) du\bigg] \\
        &\quad +  \frac{C}{N}\sum_{i=1}^N  \mathbb{E}\bigg[ \Big(\sup_{s\in[0,T]} |X^{i,N}_s-X_s^{i,g}| \Big)^{1/2}\bigg(\int_0^T |X^{i,N}_u - X^{i,g}_u|du \bigg)^{1/2} \bigg].
    \end{split}
    \end{equation}

    Fix $\eta>0$ and using that for any $a,b\in \R_{+}$ we have $ab\leq \eta^2a^2 + \eta^{-2}b^2$, we obtain
    \begin{equation*}
    \begin{split}
        \frac{1}{N}\sum_{i=1}^N \mathbb{E}\Big[\sup_{s\in [0,T]}|X^{i,N}_s-X^{i,g}_{s}|\Big] & \leq  \frac{C}{N}\sum_{i=1}^N   \mathbb{E}\left[ \int_0^t |X^{i,N}_u-X^{i,g}_u| + d_{\mathrm{BL}}(m^{i,N}_{\X_u}, \nu^{i,g}_{X_u}) du\right] \\
        &\quad +  \frac{C}{N}\sum_{i=1}^N  \mathbb{E}\bigg[ \eta^2\sup_{s\in[0,T]} |X_s^{i,N}-X_s^{i,g}| +  \frac{1}{\eta^2}\int_0^T |X^{i,N}_u-X^{i,g}_u|du \bigg]. 
    \end{split}
    \end{equation*}

    By rearranging and choosing $\eta$ sufficiently small, we obtain
    \begin{equation}
    \label{eq:first-estimate-L1}
    \begin{split}
        \frac{1}{N}\sum_{i=1}^N \mathbb{E}\Big[\sup_{s\in [0,T]}|X^{i,N}_{s} - X^{i,g}_s|\Big] & \leq  
        \frac{C}N\sum_{i=1}^N   \mathbb{E}\bigg[ \int_0^T |X^{i,N}_u-X^{i,g}_u| du\bigg] \\
         & \quad + \frac{C}{N}\sum_{i=1}^N   \mathbb{E}\bigg[ \int_0^t  d_{\mathrm{BL}}(\nu^{i,g}_{X_u} , m^{i,N}_{\X_u}) du\bigg].
    \end{split}
    \end{equation}
    
We will divide the proof into intermediate Lemmas.
\begin{lemma} \label{lem:First.estim.conf.proof-L1} 
For every $t\in [0,T]$ it holds that: 
\begin{equation}
\label{eq:estimation.proof.1-L1}
\begin{split}
    \frac1N\sum_{i=1}^N\E\bigg[d_{\mathrm{BL}}(m^{i,N}_{\X_t},  \nu^{i,g}_{X_t}) \bigg] & \le \frac{C}{\sqrt{N\beta_N}} + \frac{C}{N}\sum_{j=1}^N  \E\bigg[\sup_{s\in [0,t]}\left|X^{j,N}_s - X^{j,g}_s\right|\bigg] \\
    & \quad + \frac{1}{N}\sum_{i=1}^N\E\Big[d_{\mathrm{BL}}(\nu^{i,g}_{X_t}, m^{i,N}_{\X_t^g})\Big].\\
\end{split}
\end{equation}
    for some constant $C>0$ that does not depend on $N$.
\end{lemma}

\begin{proof}
    By triangular inequality we have
    \begin{equation}
    \label{eq:first-lemma-estim-1-L1}
    \begin{split}
        \frac{1}{N}\sum_{i=1}^N\E\Big[d_{\mathrm{BL}}(m^{i,N}_{\X_t}, \nu^{i,g}_{X_t})\Big] & \le \frac{1}{N}\sum_{i=1}^N \E\Big[d_{\mathrm{BL}}(m^{i,N}_{\X_t}, m^{i,N}_{\X_t^g})\Big] +  \frac{1}{N}\sum_{i=1}^N \E\Big[d_{\mathrm{BL}}(m^{i,N}_{\X_t^g} , \nu^{i,g}_{X_t})\Big]\\
        &\le \frac{1}{N}\sum_{i=1}^N \E\bigg[\sup_{\phi\in 1\text{-Lip}_{b}}\frac1N\sum_{j=1}^N\frac{\zeta^N_{ij}}{\beta_N}\left(\phi(X^{j,N}_t) - \phi(X^{j,g}_t)\right)\bigg]\\
        & \quad + \frac{1}{N}\sum_{i=1}^N\E\Big[d_{\mathrm{BL}}(m^{i,N}_{\X_t^g} , \nu^{i,g}_{X_t})\Big]\\
        &\le \frac{1}{N}\sum_{i=1}^N \E\bigg[\frac1N\sum_{j=1}^N\frac{\zeta^N_{ij}}{\beta_N}\left|X^{j,N}_t - X^{j,g}_t\right|\bigg]+ \frac{1}{N}\sum_{i=1}^N\E\Big[d_{\mathrm{BL}}(\nu^{i,g}_{X_t}, m^{i,N}_{\X_t^g})\Big]\\
        &\le \frac{1}{N\beta_N}\sum_{j=1}^N \E\bigg[\frac1N\sum_{i=1}^N \big( \zeta^N_{ij} - \beta_Ng_N(x_i,x_j) \big)\left|X^{j,N}_t - X^{j,g}_t\right|\bigg] \\
        & \quad + \frac{1}{N}\sum_{i=1}^N \E\bigg[\frac1N\sum_{j=1}^N g_N(x_i,x_j)\left|X^{j,N}_t - X^{j,g}_t\right|\bigg]  \\
        & \quad + \frac{1}{N}\sum_{i=1}^N\E\Big[d_{\mathrm{BL}}(m^{i,N}_{\X_t^g} , \nu^{i,g}_{X_t} )\Big]. \\
    \end{split}
    \end{equation}

    To bound the first term we use Hölder's inequality:
    \begin{equation}
    \label{eq:first-lemma-eq-1}
    \begin{split}
        & \frac{1}{N\beta_N}\sum_{j=1}^N \E\bigg[\frac1N\sum_{i=1}^N (\zeta^N_{ij} - \beta_Ng_N(x_i,x_j))\big|X^{j,N}_t - X^{j,g}_t\big|\bigg]\\
        & \le \frac{1}{N^2\beta_N}\sum_{j=1}^N \E\bigg[\bigg(\sum_{i=1}^N \zeta^N_{ij} - \beta_Ng_N(x_i,x_j)\bigg)^2 \bigg]^{1/2}\E\Big[\big|X^{j,N}_t - X^{j,g}_t\big|^2\Big]^{1/2} \\
        & \le \frac{C}{N^2\beta_N}\sum_{j=1}^N\Big(\sum_{i=1}^N  \E\Big[(\zeta^N_{ij} - \beta_Ng_N(x_i,x_j))^2\Big]\Big)^{1/2} \\
        & \le \frac{C}{N^2\beta_N}\sum_{j=1}^N\Big(\sum_{i=1}^N\beta_Ng_N(x_i,x_j) \Big)^{1/2} \\
        & \le \frac{C\|g_N\|_2}{\sqrt{N\beta_N}},
    \end{split}
    \end{equation}
    where in the third line we used Lemma \ref{lemma:L2-boundness-solutions} and the fact that when expanding the square the crossed terms vanish as $\E[\zeta_{ij}^N-g_N(x_i,x_j)]=0$ and $\zeta_{ij}^N$ and $\zeta_{ik}^N$ are independent for $j\neq k$. 
    Finally, in the fourth line we used that $\E[(\zeta_{ij}^N-g_N(x_i,x_j))^2] \leq \beta_Ng_N(x_i,x_j)\le C\beta_N$.
    Combining equations \eqref{eq:first-lemma-estim-1-L1} and \eqref{eq:first-lemma-eq-1} we get the desired result. 
\end{proof}

\begin{remark}
    In the proof of Lemma \ref{lem:First.estim.conf.proof-L1} we use that $g$ is bounded to estimate the second term in equation \eqref{eq:first-estimate-L1} and to estimate the first term through Lemma \ref{lemma:L2-boundness-solutions}.
\end{remark}

Observe that equation \eqref{eq:first-estimate-L1} and Lemma \ref{lem:First.estim.conf.proof-L1} together with Gronwall's inequality allows to write 
\begin{equation}
\label{eq:partial-result-convergence-L1}
\begin{split}
    \frac{1}{N}\sum_{i=1}^N\E\Big[ \sup_{t\in [0,T]}\left|X^{i,N}_{t}-X^{i,g}_t\right|^2 \Big] & \leq  \frac{C}{\sqrt{N\beta_N}} +\int_0^T\frac{C}{N}\sum_{i=1}^N \E\left[ d_{\mathrm{BL}}( \nu^{i,g}_{X_u}, m_{\X_u^g}^{i,N}) \right]du.\\
\end{split}
\end{equation}

The rest of the proof consists in estimating the second term on the right hand side of equation \eqref{eq:partial-result-convergence-L1}. To this end, let $h_N:[0,1]\times [0,1] \to \R_{\geq 0}$ be a sequence of measurable and piecewise constant functions such that 
\begin{equation*}
    \int_I\int_I |g(x,y)-h_N(x,y)|^2 dxdy \rightarrow 0.
\end{equation*}
\begin{remark}
\label{rk:def-h_N-in-continuous-case-L1}
    It is clear that if $g$ is continuous the sequence $h_N$ exists as we can define $h_N(x,y) = g(x_i,x_j) = g(f_N(x),f_N(y))$ where $x\in [x_i,x_{i+1})$ and $y\in [x_j,x_{j+1})$. As $f_N$ converges uniformly to the identity and $g$ is continuous then $g\circ (f_N,f_N) \xrightarrow{L^2} g$.  Under Condition \ref{cond:almost-continuity} we can do the same construction and by Lemma \ref{lemma:cond-implies-convergence} we have the condition we require.
\end{remark}
Let $X^{x,h_N}$ be the solution of the graphon SDE \eqref{eq:graphon-SDE} with graphon $h_N$ and let us put $\X^{h_N}:= (X^{1,h_N},\dots, X^{N,h_N})$.

Using triangular inequality, we have
\begin{equation}
\label{eq:proof.trans.to.stability-L1}
\begin{split}
    \frac{1}{N}\sum_{i=1}^N\E\Big[d_{\mathrm{BL}}(\nu^{i,g}_{X_u}, m^{i,N}_{\X_u^g})\Big] & \le \frac{1}{N}\sum_{i=1}^N \left\{d_{\mathrm{BL}}(\nu^{i,g}_{X_u}, \nu^{i,h_N}_{X_u}) + \E\Big[d_{\mathrm{BL}}(\nu^{i,h_N}_{X_u}, m^{i,N}_{\X_u^{h_N}})\Big]\right.\\
    & \quad\left. + \E\Big[d_{\mathrm{BL}}(m^{i,N}_{\X^{h_N}_u}, m^{i,N}_{\X^{g}_u})\Big]\right\}. \\
\end{split}
\end{equation}

The idea is now to use the stability results to bound the first and third terms. The first term is estimated by using Hölder's inequality and the bounded Lipschitz properties of $\phi$: 
\begin{align}
    \notag
     \frac{1}{N}\sum_{i=1}^N d_{\mathrm{BL}}(\nu^{i,g}_{X_u}, \nu^{i,h_N}_{X_u})
    & = \frac{1}{N}\sum_{i=1}^N \sup_{\phi \in     1\text{-Lip}_b}\bigg( \int_Ig(x_i,y)\E[\phi(X^{y,g}_u)]\lambda(dy) - \int_Ih_N(x_i,y)\E[\phi(X^{y,h_N}_u)]\lambda(dy)\bigg) \\\notag
    &\le \frac{1}{N}\sum_{i=1}^N \sup_{\phi \in     1\text{-Lip}_b}\int_I\big( g(x_i,y) - h_N(x_i,y)\big)\E[\phi(X^{y,g}_u)] \lambda(dy) \\\notag
    & \quad +  \frac{1}{N}\sum_{i=1}^N \sup_{\phi \in     1\text{-Lip}_b}\int_I h_N(x_i,y)\E\big[|\phi(X^{y,g}_u) - \phi(X^{y,h_N}_u)| \big]\lambda(dy) \\\notag
    & \le   \int_I \left(\int_I |g(f_N(x),y) - h_N(x,y)|^2 \lambda(dy)\right)^{1/2}\lambda(dx) \\\label{eq:before.intro.norms}
    & \quad +  \|h_N\|_2 \left(\int_I\E\big[|X^{y,g}_u - X^{y,h_N}_u|^2\big]\lambda(dy)\right)^{1/2} \\\notag
    & \le  \|g\circ(f_N,\mathrm{id}) - h_N\|_2 +  C \|g-h_N\|_2 \\\label{eq:d-nu-g-h_N-intermediate-L1}
    & \le  \|g\circ(f_N,\mathrm{id}) - g\|_2 +  C \|g-h_N\|_2. 
\end{align}
Where we used Hölder's inequality and the stability result (see Proposition \ref{prop:stability-of-g-integral}).

The third term in equation \eqref{eq:proof.trans.to.stability-L1} is estimated similarly:
\begin{align*}
    \frac{1}{N}\sum_{i=1}^N \E\Big[d_{\mathrm{BL}}(m^{i,N}_{\X^{h_N}_u}, m^{i,N}_{\X^{g}_u})\Big]
     &  \le \frac{1}{N}\sum_{i=1}^N  \E\bigg[\frac1N\sum_{j=1}^N\frac{\zeta^N_{ij}}{\beta_N}\big|X^{j,h_N}_u - X^{j,g}_u\big| \bigg] \\
    & \leq  \frac{1}{N}\sum_{i=1}^N  \E\bigg[\frac1N\sum_{j=1}^N g_N(x_i,x_j)\big|X^{j,h_N}_u - X^{j,g}_u\big|\bigg]\\
    & \leq \frac{1}{N}\sum_{i=1}^N \frac{1}{N}\sum_{j=1}^N g_N(x_i,x_j) \frac{1}{N}\sum_{j=1}^N \E\left[|X^{j,h_N}_u - X^{j,g}_u\big| \right]\\
    & \leq\frac{\|g_N\|_2}{N}\sum_{j=1}^N \E\left[|X^{j,h_N}_u - X^{j,g}_u\big| \right].
\end{align*}
By Proposition \ref{prop:stability-of-g-average} we can further bound:  
\begin{equation}
\label{eq:dist-m-h_N-g-intermediate-L1}
\begin{split}
    \frac{1}{N}\sum_{i=1}^N \E\Big[d_{\mathrm{BL}}(m^{i,N}_{\X^{h_N}_u}, m^{i,N}_{\X^{g}_u})\Big] \leq & C \|g_N\|_2\left( \|h_N - g\circ(f_N,\mathrm{id})\|_1 + \|g\circ(f_N,\mathrm{id})\|_2\|g-h_N\|_2\right)\\
    \leq & C \|g_N\|_2\left( \|h_N - g\|_2 + \|g-g\circ(f_N,\mathrm{id})\|_2 + \|g\circ(f_N,\mathrm{id})\|_2\|g-h_N\|_2\right).
\end{split}
\end{equation}
Hence, using equations \eqref{eq:dist-m-h_N-g-intermediate-L1}, \eqref{eq:d-nu-g-h_N-intermediate-L1} and \eqref{eq:partial-result-convergence-L1} we obtain:
\begin{equation}
\label{eq:reduction.2-L1}
\begin{split}
    \frac{1}{N}\sum_{i=1}^N\E\Big[ \sup_{s\in [0,t]}\big|X^{i,N}_{t}-X^{i,g}_t\big| \Big] & \leq  \frac{C}{\sqrt{N\beta_N}} + (1+\|g_N\|_2)\|g\circ (f_N, \mathrm{id}) - g\|_2 + C(1+\|g_N\|_2)\|h_N-g\|_2 \\
    & \quad + C \|g_N\|_2 \|g\circ(f_N,\mathrm{id})\|_2\|g-h_N\|_2\\
    & \quad + \frac{C}{N}\sum_{i=1}^N\int_0^t\E\Big[d_{\mathrm{BL}}(\nu^{i,h_N}_{X_u}, m^{i,N}_{\X_u^{h_N}})\Big]du. \\
\end{split}
\end{equation}

Let us now focus on bounding the last term in equation \eqref{eq:reduction.2-L1}. Fix $\varepsilon>0$ and let $\cN^{\eps} \sim \cN(0, \eps^2\textrm{Id}_n)$ be the multivariate normal distribution on $\R^n$ with mean vector $0$ and covariance matrix $\eps^2\textrm{Id}_n$ where $\textrm{Id}_n$ is the $n\times n$ identity matrix.
For a generic measure $\mu \in \cM_+(\R^n)$, denote by $\mu^\eps:= \cN^\eps\star \mu$, the convolution of $\mu$ and $\cN^\eps$.
This is the positive measure defined as
\begin{equation*}
    \int_{\R^n}\phi(x)\mu^\eps(dx) := \int_{\R^n}\int_{\R^n}\phi(y + z)\mu(dy)\cN^\eps(dz)
 \end{equation*} 
 for any bounded Borel measurable function $\phi:\R^n \to \R$.
Observe that for such a function, we have
\begin{equation}
\label{eq:def-regularized-measure}
    \int_{\R^n}\phi(x)\mu^\eps(dx) = \int_{\R^n}\phi(x)h^\eps(x)dx \quad \text{with}\quad h^\eps(x):= \int_{\R^n}\varphi_\eps(x-y)\mu(dy) = :\varphi_\eps\star \mu(x),
\end{equation}
where $\varphi_\eps$ is the density function of $\cN^\eps$.
We call (with some abuse) $h^\eps$ the density function of $\mu\star \cN^\eps$.
In particular, the density function of 
\begin{equation}
\label{eq:def-regularized-measure-density}
m^{i,N,\eps}_{ \X^{h_N}_u}: = m^{i,N}_{\X^{h_N}_u}\star \cN^\eps \quad \text{is}\quad h^{i,N,\eps}_u(x):= \frac1N\sum_{j=1}^N\frac{\zeta^N_{ij}}{\beta_N}\varphi_\eps(x - X^{j,h_N}_u)
\end{equation}

and the density function of
\begin{equation*}
    \nu^{i,g,\eps}_{X_u} := \nu^{i,g}_{X_u}\star \cN^{\eps} \quad \text{is}\quad h^{i,\eps}_u(x) := \int_I g(x_i,y)\E[\varphi_\eps(x - X^{y,g}_u)]\lambda(dy).
\end{equation*}

By triangular inequality we have that
\begin{equation}
\label{eq:reduction-reg-L1}
\begin{split}
    \E\Big[d_{\mathrm{BL}}(\nu^{i,h_N}_{X_u}, m^{i,N}_{\X^{h_N}_u})\Big]
     &\le  \E\Big[d_{\mathrm{BL}}(\nu^{i,h_N}_{X_u}, \nu^{i,h_N,\eps}_{X_u})\Big] + \E\Big[d_{\mathrm{BL}}(\nu^{i,h_N,\eps}_{X_u}, m^{i,N,\eps}_{\X^{h_N}_u})\Big]\\
    &\quad + \E\Big[d_{\mathrm{BL}}(m^{i,N,\eps}_{\X_u^{h_N}}, m^{i,N}_{\X_u^{h^N}} )\Big]. \\
\end{split}
\end{equation}

We bound the first and last terms on the right hand side thanks to the following lemma:

\begin{lemma}
\label{lemma:distance-regularised-measures-L1}
    Let $\eps>0$. For every $N\geq 1$ we have that there exists a constant $C>0$ independent of $N$ such that: 
    \begin{equation*}
        \frac{1}{N}\sum_{i=1}^N\E\Big[d_{\mathrm{BL}}( m^{i,N}_{\X_u^{h_N}}, m^{i,N,\eps}_{\X_u^{h_N}})\Big] \le C\sqrt{n}\eps \quad \text{ and } \quad \E\Big[d_{\mathrm{BL}}(\nu^{i,h_N}_{X_u}, \nu^{i,h_N,\eps}_{X_u})\Big] \le C\sqrt{n}\eps.
    \end{equation*}
\end{lemma}

\begin{proof}
    By definition of the bounded Lipschitz metric and the fact that the weights $\zeta^N_{ij}$ are positive, we have
    \begin{align*}
        & \E\Big[d_{\mathrm{BL}}( m^{i,N}_{\X_u^{h_N}}, m^{i,N,\eps}_{\X_u^{h_N}})\Big]\\
        & = \E\bigg[\sup_{\phi \in     1\text{-Lip}_b} \frac{1}{N}\sum_{j=1}^N \frac{\zeta^N_{ij}}{\beta_N}\left(\int_{\R^n} \phi(z)\delta_{X^{j,h_N}_u}(dz) - \int_{\R^n}\int_{\R^n} \phi(y+z)\delta_{X^{j,h_N}_u}(dy)\cN^\eps(dz)\right)\bigg]\\
        & \le \E\bigg[ \frac{1}{N}\sum_{j=1}^N \frac{\zeta^N_{ij}}{\beta_N} d_{\mathrm{BL}}\left( \delta_{X_u^{j,h_N}}, \delta_{X_u^{j,h_N}}\star \cN^{\eps}\right)\bigg]\\
        & \le \E\bigg[ \frac{1}{N}\sum_{j=1}^N \frac{\zeta^N_{ij}}{\beta_N} \cW_2\left( \delta_{X_u^{j,h_N}}, \delta_{X_u^{j,h_N}}\star \cN^{\eps}\right)\bigg],
    \end{align*}
    where $\cW_p$ is the $p$-th order Wasserstein distance. The bound on the last inequality holds as by the Kantorovich-Rubinstein duality $d_{\mathrm{BL}}\leq \cW_1$ and $\cW_1\leq \cW_2$~\cite[Remark $6.6$]{villani-optimal-transport}. We know from \cite[Lemma 3.1]{HK94} that $\cW_2^2(\delta_{X^{j,h_N}_u},\delta_{ X^{j,h_N}_u}\star \cN^\eps) \le n\eps^2$ and thus, 
    using the facts that $E[\zeta^N_{ij}]=\beta_Ng_N(x_i,x_j)$ and $\beta_N>0$ we obtain
    \begin{align*}
        \frac1N\sum_{i=1}^N\E\Big[d_{\mathrm{BL}}( m^{i,N}_{\X_u^{h_N}}, m^{i,N,\eps}_{\X_u^{h_N}})\Big] \le \sqrt{n}\varepsilon\|g_N\|_2.
     \end{align*}
     This concludes the argument for the first claimed bound.

    The proof of the second bound in the statement is the same, replacing the weighted average with weights $\zeta^{N}_{ij}/\beta_N$ by the integral $\int_I h_N(x_i,y)\lambda(dy)$ and the measure $\delta_{X^{j,h_N}_u}$ by $\cL(X^{y,h_N}_u)$. Notice that $X^{y,h_N}_u$ is square integrable $\lambda$-almost surely and hence by \cite[Lemma $3.1$]{HK94} we recover the bound. 
    We omit the proof.
\end{proof}

The following Lemma will allow us to bound the second term in \eqref{eq:reduction-reg-L1}:

\begin{lemma}
\label{lemma:middle-term-reg-L1}
    There is a constant $C(\varepsilon)>0$ which does not depend on $N$, but may depend on $\eps>0$ such that for every $N\geq 1$ it holds that:
    \begin{equation*}
    \begin{split}
        \frac1N\sum_{i=1}^N\E\Big[d_{\mathrm{BL}}(\nu^{i,h_N,\eps}_{X_u}, m^{i,N,\eps}_{\X^{h_N}_u})\Big] & \le \frac{C}{\eps^{n/2}\sqrt{N\beta_N}}\|g_N\|^{1/4}_2 + \frac{C\|g_N\|^{1/2}_2}{\eps^{n/2} \sqrt{N\beta_N}} \left(\frac{1}{N}\sum_{j=1}^N\E[|X^{j,h_N}_u|^{2n+2}]\right)^{1/4}   \\
        & \quad + \frac{C}{\eps^{n/2}\sqrt{N}}\|h_N-g_N\|_2 + \frac{C\|h_N-g_N\|_2}{\eps^{n/2} \sqrt{N}} \left(\frac{1}{N}\sum_{j=1}^N\E[|X^{j,h_N}_u|^{n+1}]\right)^{1/2}\\
        & \quad + \frac{C\|h_N\|_2}{\eps^{n/2}\sqrt{N}} + \frac{C\|h_N\|_4}{\eps^{n/2}\sqrt{N}} \left(\frac{1}{N}\sum_{j=1}^N \E\left[|X_u^{y_j,h_N}|^{2n+2}\right]\right)^{1/4}.
    \end{split}
    \end{equation*}

    In particular, we obtain:
    \begin{equation*}
        \limsup_{N\rightarrow \infty} \frac{C}{N}\sum_{i=1}^N\int_0^t\E\Big[d_{\mathrm{BL}}(\nu^{i,h_N}_{X_u}, m^{i,N}_{\X_u^{h_N}})\Big]du = \sqrt{n}\eps \|g_N\|_2.
    \end{equation*}
\end{lemma}
\begin{proof}
     We introduce the notation:
    \begin{equation*}
        \alpha_j(z) = \varphi_\eps(z - X_u^{j,h_N}), \quad \alpha_y(z) = \varphi_\eps(z - X_u^{y,h_N})\quad \text{and} \quad \Gamma_{ij} = \zeta^N_{ij} - \beta_N g_N(x_i,x_j).
    \end{equation*} 
    We can first notice that $x\mapsto \cL( X^{x,h_N})$ are piecewise constant over intervals $[x_i,x_{i+1})$. 
    This follows from the fact that as $h_N$ is piecewise constant if $x,y\in [x_i,x_{i+1})$ then $\nu^{x,h_N}_X = \nu^{y,h_N}_X$, so that $X^{x,h_N}$ and $X^{y,h_N}$ satisfy the same SDE and by (the proof of) Lemma \ref{lemma:sol-sde-are-epi}, their laws are the same. 
    
    Let us subsequently use the definition of $d_{\mathrm{BL}}$, definition of the convolution, and the fact that $h_N$ and $x\mapsto \cL( X^{x,h_N})$ are piecewise constant to obtain the following inequalities:
    \begin{align*}
        & \E\Big[d_{\mathrm{BL}}(\nu^{i,h_N,\eps}_{X_u}, m^{i,N,\eps}_{\X^{h_N}_u})\Big]\\
        & = \E\bigg[ \sup_{\phi \in     1\text{-Lip}_b} \int_{\R^n}\phi(z)\frac1N\sum_{j=1}^N\frac{\zeta^{N}_{ij}}{\beta_N}\alpha_j(z) dz  - \int_{\R^n}\int_{I}\phi(z)h_N(x_i,y)\E[\alpha_y(z)]\lambda(dy)dz \bigg]\\
        &\le \E\bigg[\sup_{\phi \in 1\text{-Lip}_b} \int_{\R^n}\phi(z)\frac1N\sum_{j=1}^N\frac{\zeta^{N}_{ij}}{\beta_N} \alpha_j(z) dz -\int_{\R^n} \frac1N \sum_{j=1}^N h_N(x_i,x_j) \phi(z) \alpha_j(z)dz \bigg]\\
        & \quad + \E\bigg[\sup_{\phi \in 1\text{-Lip}_b}  \int_{\R^n}\frac1N\sum_{j=1}^Nh_N(x_i,x_j)\phi(z)\alpha_j(z)dz - \int_{\R^n}\frac{1}{N}\sum_{j=1}^N\phi(z)h_N(x_i,x_j)\E[\alpha_y(z)]dz \bigg]\\
        &\le C\E\bigg[\sup_{\phi \in 1\text{-Lip}_b} \frac1N\sum_{j=1}^N\Big( \frac{\zeta^{N}_{ij}}{\beta_N} - g_N(x_i,x_j)\Big)\int_{\R^n}\phi(z)\alpha_j(z)dz \bigg]\\
        & \quad + C\E\bigg[\sup_{\phi \in 1\text{-Lip}_b} \frac1N\sum_{j=1}^N\Big( g_N(x_i,x_j)-h_N(x_i,x_j)\Big)\int_{\R^n}\phi(z)\alpha_j(z))dz \bigg]\\
        & \quad + C \E\bigg[\sup_{\phi \in 1\text{-Lip}_b}  \int_{\R^n}\phi(z)\frac1N\sum_{j=1}^Nh_N(x_i,x_j)\Big(\alpha_j(z) - \E[\alpha_j(z)]\Big)dz \bigg]\\
        & = C (L^i_1 + L^i_2 + L^i_3).\\
    \end{align*}

    We start by bounding $L^i_1$: 
    \begin{equation*}
    \begin{split}
        \frac{1}{N}\sum_{i=1}^N L^i_1 & =  \frac{1}{N}\sum_{i=1}^N\E\bigg[\sup_{\phi \in 1\text{-Lip}_b} \int_{\R^n}\phi(z) \frac{1}{N\beta_N}\sum_{j=1}^N \Gamma_{ij}\alpha_j(z)dz \bigg] \\
        & \le \frac{1}{N}\sum_{i=1}^N \E\bigg[\int_{\R^n}\bigg| \frac{1}{N\beta_N}\sum_{j=1}^N \Gamma_{ij}\alpha_j(z)\bigg|dz \bigg] .\\
    \end{split}
    \end{equation*}

    To bound this integral we will use a modification of Carlson's Lemma given in \cite{HK94}: 
    \begin{equation}
    \label{eq:Carson}
        \bigg(\int_{\R^n}\bigg| \frac{1}{N\beta_N}\sum_{j=1}^N \Gamma_{ij}\alpha_j(z)\bigg|dz \bigg)^2 \leq C \int_{\R^n} (1+|z|^{n+1})\left( \frac{1}{N\beta_N}\sum_{j=1}^N \Gamma_{ij}\alpha_j(z)\right)^2dz. 
    \end{equation}
    Then, we can expand the square and notice that we can factor the expectation by independence. 
    Moreover, the cross terms vanish as $\E[\Gamma_{ij}]=0$ for $i\neq j$ and we use that $\mathrm{Var}[\Gamma_{ij}] = \beta_Ng_N(x_i,x_j)(1-\beta_Ng_N(x_i,x_j))\leq C\beta_N$. 
    Then, we obtain
    \begin{equation*}
    \begin{split}
        \frac{1}{N}\sum_{i=1}^N L_1^i \leq & \frac{C}{N^2\beta_N}\sum_{i=1}^N \E\left[\left(\int_{\R^n} (1+|z|^{n+1})\left(\sum_{j=1}^N \Gamma_{ij}\alpha_j(z)\right)^2dz \right)^{1/2}\right] \\
        \leq & \frac{C}{N^2\beta_N}\sum_{i=1}^N \left(\sum_{j=1}^N\int_{\R^n} (1+|z|^{n+1}) \mathrm{Var}[\zeta^N_{ij}]\E[\alpha^2_j(z)]dz \right)^{1/2} \\
        \leq & \frac{C}{N^2\beta_N} \sum_{i=1}^N\left(\sum_{j=1}^N g_N(x_i,x_j)\beta_N  \int_{\R^n} (1+|z|^{n+1}) \E[\alpha^2_j(z)]dz \right)^{1/2} .\\
    \end{split}
    \end{equation*}

    The next step is to bound the remaining integral uniformly in $j$. We first notice that
    \begin{equation*}
        \varphi_\eps^2(x) = \frac{1}{(4\pi)^{n/2}\eps^{n}}\varphi_{\tilde\eps}(x), \quad \text{where $\tilde\eps = \frac{\eps}{\sqrt{2}}$}.
    \end{equation*}
    We can bound the integral term by noticing that for $\eps$ small we can estimate:
    \begin{equation}
    \label{eq:Li1-bound-L1}
    \begin{split}
        \E\left[\int_{\R^n} (1+|z|^{n+1})\alpha_j^2(z)dz\right]
       & =   \int_{\R^n} \int_{\R^n} (1+|z|^{n+1})\varphi_\eps^2(z-x)dz \cL(X^{j,h_N}_u)(dx)\\
        & \le  \frac{1}{(4\pi)^{n/2}\eps^{n}} \int_{\R^n} \int_{\R^n} (1+(|z-x|+|x|)^{n+1})\varphi_{\tilde\eps}(z-x)dz \cL(X^{j,h_N}_u)(dx)\\
        & \le  \frac{1}{(4\pi)^{n/2}\eps^{n}} \int_{\R^n} \int_{\R^n} \varphi_{\tilde\eps}(z-x)dz \cL(X^{j,h_N}_u)(dx) \\
        & \quad +  \frac{2^n}{(4\pi)^{n/2}\eps^{n}} \int_{\R^n} \int_{\R^n} (|z-x|^{n+1}+|x|^{n+1})\varphi_{\tilde\eps}(z-x)dz \cL(X^{j,h_N}_u)(dx)\\
        & \le  \frac{1}{(4\pi)^{n/2}\eps^{n}} +  \frac{1}{\pi^{n/2}\eps^{n}} \int_{\R^n} |v|^{n+1} \varphi_{\tilde\eps}(v) dv\\
        &\quad +  \frac{1}{\pi^{n/2}\eps^{n}} \int_{\R^n} |x|^{n+1}\cL(X^{j,h_N}_u)(dx)\\
        & = \frac{1}{(4\pi)^{n/2}\eps^{n}} +  \frac{\E[\tilde\eps^{n+1} |\cN^1|^{n+1}]}{\pi^{n/2}\eps^{n}} +  \frac{\E[|X^{j,h_N}_u|^{n+1}]}{\pi^{n/2}\eps^{n}}\\
        & \leq C\eps^{-n} + \frac{\E[|X^{j,h_N}_u|^{n+1}]}{\pi^{n/2}\eps^{n}},
    \end{split}
    \end{equation}
    where for the last line we used the fact that for $\P$-almost every $\omega\in \Omega$ $\varphi_{\tilde\eps}(z- X^{j,h_N}_u(\omega))$ is a probability density and it integrates to $1$ for the first term. For the second term we used a change of variable and that the Lebesgue measure is invariant against translations. Hence, we conclude that
    \begin{equation}
    \label{eq:bound.L1}
    \begin{split}
        \frac{1}{N} \sum_{i=1} L^i_1 \leq & \frac{C}{\eps^{n/2}\sqrt{N\beta_N}}\|g_N\|^{1/2}_2 + \frac{C}{\eps^{n/2} N^2\sqrt{\beta_N}} \sum_{i=1}^N\left(\sum_{j=1}^N g_N(x_i,x_j)\E[|X^{j,h_N}_u|^{n+1}]\right)^{1/2}\\
        \leq & \frac{C}{\eps^{n/2}\sqrt{N\beta_N}}\|g_N\|^{1/2}_2 + \frac{C}{\eps^{n/2} N^2\sqrt{\beta_N}} \sum_{i=1}^N\left(\sum_{j=1}^N g_N(x_i,x_j)\E[|X^{j,h_N}_u|^{n+1}]\right)^{1/2} \\
        \leq & \frac{C}{\eps^{n/2}\sqrt{N\beta_N}}\|g_N\|^{1/2}_2 + \frac{C}{\eps^{n/2} N^2\sqrt{\beta_N}} \sum_{i=1}^N\left(\sum_{j=1}^N g^2_N(x_i,x_j)\sum_{j=1}^N\E[|X^{j,h_N}_u|^{2n+2}]\right)^{1/4} \\
        \leq & \frac{C}{\eps^{n/2}\sqrt{N\beta_N}}\|g_N\|^{1/4}_2 + \frac{C\|g_N\|^{1/2}_2}{\eps^{n/2} \sqrt{N\beta_N}} \left(\frac{1}{N}\sum_{j=1}^N\E[|X^{j,h_N}_u|^{2n+2}]\right)^{1/4}.
    \end{split}
    \end{equation}
    Notice that the last average is bounded by Lemma \ref{lemma:L2-boundness-solutions-g-unbounded-part-1} if $g$ is such that $\|g\|_k < \infty$.

    We next bound $L^i_2$:     
    \begin{equation*}
    \begin{split}
        & \frac{1}{N}\sum_{i=1}^N L^i_2 \\
        & = \frac{1}{N}\sum_{i=1}^N\E\bigg[\sup_{\phi \in 1\text{-Lip}_b} \int_{\R^n}\phi(z) \frac1N\sum_{j=1}^N \Big(g_N(x_i,x_j)-h_N(x_i,x_j)\Big)\varphi_\eps(z - X^{j,h_N}_u)dz \bigg]\\
        & = \frac{1}{N}\sum_{i=1}^N\E\bigg[\int_{\R^n}\bigg|\frac1N\sum_{j=1}^N \Big(g_N(x_i,x_j)-h_N(x_i,x_j)\Big)\varphi_\eps(z - X^{j,h_N}_u)\bigg| dz \bigg]\\
        & \le \frac{C}{N}\sum_{i=1}^N\E\left[\left(\int_{\R^n}(1+|z|^{n+1})\left(\frac1N\sum_{j=1}^N \Big(g_N(x_i,x_j)-h_N(x_i,x_j)\Big)\varphi_\eps(z - X^{j,h_N}_u)\right)^2 dz\right)^{1/2} \right] \\
        & \le \frac{C}{N}\sum_{i=1}^N\left(\frac1N\sum_{j=1}^N \Big(g_N(x_i,x_j)-h_N(x_i,x_j)\Big)^2\frac{1}{N}\sum_{j=1}^N\E\left[\int_{\R^n}(1+|z|^{n+1})\varphi^2_\eps(z - X^{j,h_N}_u) dz \right]\right)^{1/2} \\
        & \le \frac{C}{\eps^{n/2}\sqrt{N}}\|h_N-g_N\|_2 + \frac{C\|h_N-g_N\|_2}{\eps^{n/2} \sqrt{N}} \left(\frac{1}{N}\sum_{j=1}^N\E[|X^{j,h_N}_u|^{n+1}]\right)^{1/2},
    \end{split}
    \end{equation*}
    where the sixth line follows by the same bound obtained in equation \eqref{eq:Li1-bound-L1}. 
    
    We now turn our attention to $L^i_3$. Using Carlson's Lemma we can bound $L^i_3$ as follows:

    \begin{equation}
    \label{eq:bound-Li3}
    \begin{split}
        \frac{1}{N}\sum_{i=1}^N L^i_3 & \le  \frac{C}{N}\sum_{i=1}^N\E\bigg[\sup_{\phi \in 1\text{-Lip}_b} \int_{\R^n}\phi(z)\frac1N\sum_{j=1}^N h_N(x_i,x_j)\Big(\varphi_\eps(z - X^{j,h_N}_u) - \E[\varphi_\eps(z - X^{j, h_N}_u)]\Big)dz\bigg]\\
        & \le  \frac{C}{N}\sum_{i=1}^N\E\left[  \int_{\R^n}\left|\frac1N\sum_{j=1}^Nh_N(x_i,x_j) \Big(\varphi_\eps(z - X^{j,h_N}_u) - \E[\varphi_\eps(z - X^{j, h_N}_u)]\right|dz \right]\\
        & \le  \frac{C}{N}\sum_{i=1}^N\E\left[  \int_{\R^n}(1+|z|^{n+1})\left(\frac1N\sum_{j=1}^N h_N(x_i,x_j)\Big(\varphi_\eps(z - X^{j,h_N}_u) - \E[\varphi_\eps(z - X^{j, h_N}_u)]\right)^2dz \right]^{1/2}.\\
    \end{split}
    \end{equation}
    As before we expand the square and notice that the crossed terms vanish by independence of $X^{j}$ and $X^{k}$. Hence, we obtain
    \begin{equation*}
    \begin{split}
        \frac{1}{N}\sum_{i=1}^N L^i_3 & \le  \frac{C}{N}\sum_{i=1}^N \E\left[  \int_{\R^n}(1+|z|^{n+1})\frac{1}{N^2}\sum_{j=1}^N h^2_N(x_i,x_j)\mathrm{Var}[\varphi_\eps(z - X^{j,h_N}_u)]dz \right]^{1/2}\\
        & \le  \frac{C}{N}\sum_{i=1}^N \left(\frac{1}{N^2}\sum_{j=1}^N h^2_N(x_i,x_j)\E\left[ \int_{\R^n}(1+|z|^{n+1})\E[\varphi^2_\eps(z - X^{j,h_N}_u)]dz \right]\right)^{1/2}\\
        & \leq \frac{C\|h_N\|_2}{\eps^{n/2}\sqrt{N}} + \frac{C}{\eps^{n/2}N}\sum_{i=1}^N \left(\frac{1}{N^2}\sum_{j=1}^N h^2_N(x_i,x_j)\E\left[|X_u^{j,h_N}|^{n+1}\right]\right)^{1/2} \\
        & \leq  \frac{C\|h_N\|_2}{\eps^{n/2}\sqrt{N}} + \frac{C\|h_N\|_4}{\eps^{n/2}\sqrt{N}} \left(\frac{1}{N}\sum_{j=1}^N \E\left[|X_u^{j,h_N}|^{2n+2}\right]\right)^{1/4}, \\
    \end{split}
    \end{equation*}
    where for the second line we used the fact that $\mathrm{Var}[\varphi_\eps(z - X^{j,h_N}_u)] \leq \E\left[\varphi^2_\eps(z - X^{j,h_N}_u)\right]$ and for the third line we used that the integral is bounded by the calculation done in equation \eqref{eq:Li1-bound-L1}. 
    This finishes the proof. 

    The second part follows from equation \eqref{eq:reduction-reg-L1}, Lemma \ref{lemma:distance-regularised-measures-L1} and the fact that as $g_N,h_N$ are bounded then all their norms are uniformly bounded and the averge of the moments of $X^{j,h_N}_u$ are uniformly bounded by Lemma \ref{lemma:L2-boundness-solutions}.
\end{proof}

Using equation \eqref{eq:reduction.2-L1} and Lemma \ref{lemma:middle-term-reg-L1} we obtain:

\begin{equation}
\label{eq:reduction.3-L1}
\begin{split}
    & \limsup_{N\rightarrow \infty} \frac{1}{N}\sum_{i=1}^N\E\Big[ \sup_{s\in [0,t]}\left|X^{i,N}_{t}-X^{i,g}_t\right|^2 \Big] \\
    & \leq \limsup_{N\rightarrow \infty} \frac{C}{\sqrt{N\beta_N}} + (1+\|g_N\|_2)\|g\circ (f_N, \mathrm{id}) - g\|_2 + C(1+\|g_N\|_2)\|h_N-g\|_2 \\
    & \quad + C \|g_N\|_2 \|g\circ(f_N,\mathrm{id})\|_2\|g-h_N\|_2\\
    & \quad + \frac{C}{N}\sum_{i=1}^N\int_0^t\E\Big[d_{\mathrm{BL}}(\nu^{i,h_N}_{X_u}, m^{i,N}_{\X_u^{h_N}})\Big]du \\
    & = C\sqrt{n}\eps. \\
\end{split}
\end{equation}

As equation \eqref{eq:reduction.3-L1} holds for every $\eps>0$ we conclude that 
\begin{equation*}
    \frac{1}{N}\sum_{i=1}^N\E\Big[ \sup_{s\in [0,T]}\left|X^{i,N}_{s}-X^{i,g}_s\right|^2 \Big]  \xrightarrow{N\rightarrow \infty} 0. 
\end{equation*}

 For the second part, by Lemma \ref{lem:First.estim.conf.proof-L1}, equations \eqref{eq:proof.trans.to.stability-L1},\eqref{eq:d-nu-g-h_N-intermediate-L1}, \eqref{eq:dist-m-h_N-g-intermediate-L1} and \eqref{eq:reduction-reg-L1}, Lemma \ref{lemma:distance-regularised-measures-L1} and Lemma \ref{lemma:middle-term-reg-L1} we have
\begin{equation}
\label{eq:estim.dBL.main}
\begin{split}
    \limsup_{N\rightarrow  \infty}\frac{1}{N}\sum_{i=1}^N\E\Big[d_{\mathrm{BL}}(m_{\X_t}^{i,N},\nu_{X_t}^{i,g}) \Big] 
    & \leq \limsup_{N\rightarrow  \infty} \left(\frac{C(\eta)}{N}\sum_{i=1}^N   \mathbb{E}\left[ \int_0^t \sup_{r\in [0,u]}|X^{i,N}_r-X^{i,g}_r|du\right] + \frac{C(\eps)}{\sqrt{N\beta_N}}\right.\\
    & \quad \left. + \frac{1}{N}\sum_{i=1}^N\E\Big[d_{\mathrm{BL}}(\nu^{i,g}_{X_t}, m^{i,N}_{\X_t^g})\Big]\right)\\
    & = C\sqrt{n}\eps.
\end{split}
\end{equation}
Again, as this bound holds for every $\eps>0$ we obtain that
\begin{equation*}
    \frac{1}{N}\sum_{i=1}^N\E\Big[d_{\mathrm{BL}}(m_{\X_t}^{i,N},\nu_{X_t}^{i,g}) \Big] \xrightarrow{N\rightarrow \infty} 0. 
\end{equation*}
This finishes the proof.

\subsection{Some extensions and consequences of the main convergence result}
\label{sec:corollaries.conve}
In this Section we state and prove a few by-products of Theorem \ref{thm:convergence-bounded-graphon-L1} and proof.
We start by providing the proof of Corollary \ref{coro:convergence-laws-L1}. 

\begin{proof}[Proof of Corollary \ref{coro:convergence-laws-L1}]

By definition of $d_\mathrm{BL}$ and the fact that $\cL(X^{i,N}_t) = \cL(Y^{i,N}_t)$, we have
\begin{equation*}   
\begin{split}
    & \frac{1}{N}\sum_{i=1}^N \E\Big[\sup_{t\in [0,T]}d_{\mathrm{BL}}(\cL(Y^{i,N}_t), \cL(X^{i,g}_t))\Big]
     = \frac{1}{N}\sum_{i=1}^N \E\Big[\sup_{t\in [0,T]}d_{\mathrm{BL}}(\cL(X^{i,N}_t), \cL(X^{i,g}_t))\Big] \\
    & = \frac{1}{N}\sum_{i=1}^N \E\bigg[\sup_{t\in [0,T]}\sup_{\phi \in 1\text{-Lip}_b} \int_{\R^n} \phi(z)\cL(X^{i,N}_t)(dz) - \int_{\R^n} \phi(z)\cL(X^{i,g}_t)(dz) \bigg] \\ 
    & = \frac{1}{N}\sum_{i=1}^N \E\bigg[\sup_{t\in [0,T]}\sup_{\phi \in 1\text{-Lip}_b}\big|  \phi(X^{i,N}_t) - \phi(X^{i,g}_t)\big| \bigg] \\ 
    & \le \frac{1}{N}\sum_{i=1}^N \E\Big[\sup_{t\in [0,T]}\big|  X^{i,N}_t - X^{i,g}_t\big| \Big]. 
\end{split}
\end{equation*}
By Theorem \ref{thm:convergence-bounded-graphon-L1} the last term goes to zero and we can conclude. 
\end{proof}

The next result states that the (random) empirical measure of the interacting particle system converges in probability to the integral of the laws of the solution of the graphon SDE \eqref{eq:graphon-SDE}.
This is in contrast to the complete graph (as well as the Erd\"os-R\'enyi graph) case where the empirical measure converges to the law of the limiting distribution.
Naturally, if the processes $(X^x)_{x\in I}$ have the same law we recover the standard result.

\begin{corollary}\label{coro:convergence-empirical-to-int}
    Under the assumptions of Theorem \ref{thm:convergence-bounded-graphon-L1}, for every fixed $t\in [0,T]$, it holds that
    \begin{equation*}
        \frac{1}{N}\sum_{i=1}^N \delta_{X^{i,N}_t} \rightarrow \int_I \cL(X^{y}_t) \lambda(dy) \quad \text{in probability. }
    \end{equation*}
\end{corollary}

\begin{proof}
    Let $t\in [0,T]$ be fixed.
    By Markov's inequality it is enough to show that
    \begin{equation}
    \label{eq:coro-empirical-to-int-goal}
        \E\bigg[d_{\mathrm{BL}}\bigg(\frac{1}{N}\sum_{i=1}^N \delta_{X^{i,N}_t}, \int_I \cL(X^{y,g}_t)\lambda(dy)\bigg)\bigg] \rightarrow 0. 
    \end{equation}
    To simplify the notation, let us put $m^N_t := \frac{1}{N}\sum_{i=1}^N \delta_{X^{i,N}_t}$, $m^{N,g}_t := \frac{1}{N}\sum_{i=1}^N \delta_{X^{i,g}_t}$ and $\nu^{g}_t := \int_I \cL(X^{y,g}_t)\lambda(dy)$,  and denote by $m_t^{N,\eps}, m^{N,g,\eps}$ and $\nu^{g,\eps}_t$ the respective Gaussian regularizations as defined in \eqref{eq:def-regularized-measure}.

    To show \eqref{eq:coro-empirical-to-int-goal} we proceed as in the proof of Theorem \ref{thm:convergence-bounded-graphon-L1} and consider the sequence of approximating graphons $g_N$.
    Given $\eps>0$, by triangle inequality, we have
    \begin{equation}
    \label{eq:coro-conv-law-estimate-1}
        \begin{split}
            \E\big[d_{\mathrm{BL}}\big(m^N_t, \nu^{g}_t\big)\big]  & \leq \E\big[d_{\mathrm{BL}}\big(m^N_t, m^{N,g}_t\big)\big] + \E\big[d_{\mathrm{BL}}\big(m^{N,g}_t, m^{N,g_N}_t\big)\big] + \E\big[d_{\mathrm{BL}}\big(m^{N,g_N}_t, m^{N,g_N,\eps}_t\big)\big]\\
            & \quad + \E\big[d_{\mathrm{BL}}\big( m^{N,g_N,\eps}_t, \nu^{g_N,\eps}_t\big)\big] + \E\big[d_{\mathrm{BL}}\big( \nu^{g_N,\eps}_t, \nu_t^{g_N}\big)\big] + \E\big[d_{\mathrm{BL}}\big( \nu_t^{g_N}, \nu^{g}_t \big)\big] \\
            & =: \cT_1 + \cT_2 + \cT_3 + \cT_4 + \cT_5 + \cT_6.
    \end{split}
    \end{equation}
    We can estimate $\cT_1$ as:
    \begin{equation}
    \label{eq:estimate-T1}
        \begin{split}
        \cT_1 
          &  \leq  \E\Big[\frac{1}{N}\sum_{i=1}^N  |X^{i,N}_t - X^{i,g}_t|\Big].
        \end{split}
    \end{equation}
    Thus, using Theorem \ref{thm:convergence-bounded-graphon-L1} we conclude that $\cT_1\rightarrow 0.$
    For $\cT_2$ we obtain by Proposition \ref{prop:stability-of-g-average} that
    \begin{equation}
    \label{eq:estimate-T2}
    \begin{split}
        \cT_2 \leq \E\Big[ \frac{1}{N}\sum_{i=1}^N |X_t^{i,g}-X^{i,g_N}_t|\Big]\rightarrow 0.
    \end{split}
    \end{equation}
    With a similar argument but using Proposition \ref{prop:stability-of-g-integral} we can bound $\cT_6$ by
    \begin{equation}
    \label{eq:estimate-T6}
        \begin{split}
            \cT_6 &\leq  E\bigg[\sup_{\phi\in 1\text{-Lip}_b} \int_I  \E\big[\phi(X^{y,g_N}_t) - \phi(X^{y,g}_t) \big]\lambda(dy) \bigg]\\
            & \leq \int_I  \E[|X^{y,g_N}_t - X^{y,g}_t|\lambda(dy) \leq C \|g_N-g\|_2 \rightarrow 0.
        \end{split}
    \end{equation}
    Furthermore, with the same argument used in Lemma \ref{lemma:distance-regularised-measures-L1} we obtain
    \begin{equation}
    \label{eq:estimate-T3-and-T5}
        \cT_3 + \cT_5 \leq Cn\sqrt{\eps}.
    \end{equation}
    Finally, to bound $\cT_4$ we use independence of the solutions of the graphon SDE \eqref{eq:graphon-SDE}.
    Namely,
    \begin{equation}
    \label{eq:estimate-T4}
        \begin{split}
            \cT_4 & \leq \E\bigg[ \sup_{\phi \in 1\text{-Lip}_b} \int_{\R^n}|\phi(z)|\bigg|\frac{1}{N}\sum_{i=1}^N \varphi_\eps(z-X^{i,g_N}_t) - \int_I \E[\varphi_\eps(z-X^{y,g_N}_t)]\bigg|dz\bigg]\\
            & \leq  \E\bigg[\int_{\R^n} \bigg|\frac{1}{N}\sum_{i=1}^N \left(\varphi_\eps(z-X^{i,g_N}_t) - \E[\varphi_\eps(z-X^{i,g_N}_t)]\right)\bigg|dz\bigg]\\
            & \leq  \E\bigg[\int_{\R^n} (1+|z|^{n+1})\bigg(\frac{1}{N}\sum_{i=1}^N \left(\varphi_\eps(z-X^{i,g_N}_t) - \E[\varphi_\eps(z-X^{i,g_N}_t)]\right)\bigg)^2dz \bigg]^{1/2}\\
            & \leq  \bigg(\frac{1}{N^2}\sum_{i=1}^N\int_{\R^n} (1+|z|^{n+1})\E\left[\varphi^2_\eps(z-X^{i,g_N}_t)\right]dz \bigg)^{1/2}\\
            & \leq \frac{C(\eps)}{\sqrt{N}},
        \end{split}
    \end{equation}
    where we bounded the integral term uniformly in $i$ using equation \eqref{eq:Li1-bound-L1} for the last line. 

Combining equation \eqref{eq:coro-conv-law-estimate-1} and equations \eqref{eq:estimate-T1}, $\eqref{eq:estimate-T2}$, $\eqref{eq:estimate-T3-and-T5}$, \eqref{eq:estimate-T4} and \eqref{eq:estimate-T6} we arrive at 
    \begin{equation*}
        \lim_{N\rightarrow \infty}\E\left[d_{\mathrm{BL}}\left(m^N_t, \nu^{g}_t\right)\right] = 0. 
    \end{equation*}
This concludes the proof. 
\end{proof}

The next Corollary states that if we assume stronger conditions on the graphon $g$, then we can get a rate of convergence in Theorem \ref{thm:convergence-bounded-graphon-L1}.

\begin{corollary}
\label{coro:rate-convergence}
    Assume Conditions \ref{cond:initial-conditions}, 
    \ref{cond:coefficients}, \ref{cond:graphon-cond-2}, \ref{cond:almost-continuity}, \ref{cond:bound-initial-cond}. If $\sigma$ does not depend on the measure argument and we further assume that $g$ is Lipschitz then we have that there exists a constant $C>0$ independent of $N$ such that for every $N\geq 1$ it holds that
    \begin{equation}
    \label{eq:coro-rate-1}
        \frac{1}{N}\sum_{i=1}^N\E\Big[ \sup_{t\in [0,T]}\big|X^{i,N}_{t}-X^{i,g}_t\big| \Big] \leq  \frac{C}{(N\beta_N)^{\frac{1}{2(n+1)}}}
    \end{equation}
    and that
    \begin{equation*}
    \label{eq:coro-rate-3}
        \frac1N\sum_{i=1}^N\E\Big[d_{\mathrm{BL}}(m^{i,N}_{\X_t},  \nu^{i,g}_{X_t}) \Big] + \frac{1}{N}\sum_{i=1}^N \E\Big[\sup_{t\in [0,T]}d_{\mathrm{BL}}(\cL(Y^{i,N}_t), \cL(X^{i,g}_t))\Big] \leq \frac{C}{(N\beta_N)^{\frac{1}{2(n+1)}}}.
    \end{equation*}
\end{corollary} 

\begin{proof}
    It is enough to show equation \eqref{eq:coro-rate-1}.
    The other bound follows by the same argument, see \eqref{eq:estim.dBL.main}. 
    In the proof of Theorem \ref{thm:convergence-bounded-graphon-L1} we showed that
    \begin{equation}
    \label{eq:reduction.3-coro}
    \begin{split}
        \frac{1}{N}\sum_{i=1}^N\E\Big[ \sup_{t\in [0,T]}\left|X^{i,N}_{t}-X^{i,g}_t\right|\Big] 
        & \leq   \frac{C}{\sqrt{N\beta_N}} + C(1+\|g_N\|_2)\|g\circ (f_N, \mathrm{id}) - g\|_2 + C(1+\|g_N\|_2)\|h_N-g\|_2 \\
        & \quad + C \|g_N\|_2 \|g\circ(f_N,\mathrm{id})\|_2\|g-h_N\|_2  + C\sqrt{n}||g_N||_2 \eps \\
        & \quad + \frac{C}{\eps^{n/2}\sqrt{N\beta_N}}\|g_N\|^{1/4}_2 + \frac{C\|g_N\|^{1/2}_2}{\eps^{n/2} \sqrt{N\beta_N}} \left(\frac{1}{N}\sum_{j=1}^N\E[|X^{j,h_N}_u|^{2n+2}]\right)^{1/4}   \\
        & \quad + \frac{C}{\eps^{n/2}\sqrt{N}}\|h_N-g_N\|_2 + \frac{C\|h_N-g_N\|_2}{\eps^{n/2} \sqrt{N}} \left(\frac{1}{N}\sum_{j=1}^N\E[|X^{j,h_N}_u|^{n+1}]\right)^{1/2}\\
        & \quad + \frac{C\|h_N\|_2}{\eps^{n/2}\sqrt{N}} + \frac{C\|h_N\|_4}{\eps^{n/2}\sqrt{N}} \left(\frac{1}{N}\sum_{j=1}^N \E\left[|X_u^{j,h_N}|^{2n+2}\right]\right)^{1/4}. \\
    \end{split}
    \end{equation}
    
    As $g$ is continuous we can choose $h_N(x,y)=g(f_N(x),f_N(y))$ as was mentioned in Remark \ref{rk:def-h_N-in-continuous-case-L1}. Moreover, $g_N$ is just $g_N(x,y)=g(f_N(x),f_N(y))$. Notice that by Lemma \ref{lemma:L2-boundness-solutions} and the fact that $g$ is bounded we can write:
    \begin{equation}
    \begin{split}
        \frac{1}{N}\sum_{i=1}^N\E\Big[ \sup_{t\in [0,T]}\left|X^{i,N}_{t}-X^{i,g}_t\right|\Big] 
        & \leq   \frac{C}{\sqrt{N\beta_N}} + C \|g\circ (f_N, \mathrm{id}) - g\|_2 + C \|h_N-g\|_2 \\
        & \quad  + C\sqrt{n}(\|g_N - g\|_2 + \|g\|_2)\eps \\
        & \quad + \frac{C}{\eps^{n/2}\sqrt{N\beta_N}} + \frac{C}{\eps^{n/2}\sqrt{N}}. \\
    \end{split}
    \end{equation}
    Notice that we can estimate $\|g-h_N\|_2$ by definition as
    \begin{align}
    \notag
        \|g-h_N\|_2^2 & = \int_I\int_I |g(x,y)-g(f_N(x),f_N(y))|^2 dxdy\\\label{eq:coro-lipschitz-norm}
        & \leq C \int_I\int_I |x-f_N(x)|^2 + |y-f_N(y)|^2 dxdy \leq \frac{C}{N^2}.
    \end{align}
    With a similar argument we obtain that:
    \begin{equation}
    \label{eq:coro-lipschitz-norm-2}
    \begin{split}
        \|g\circ(f_N,\mathrm{id})-g\|_2^2 \leq \frac{C}{N^2}.
    \end{split}
    \end{equation}
    
    Combining equations \eqref{eq:reduction.3-coro}, \eqref{eq:coro-lipschitz-norm} and \eqref{eq:coro-lipschitz-norm-2} we obtain: 
    \begin{equation}
    \label{eq:reduction.4-coro}
    \begin{split}
        \frac{1}{N}\sum_{i=1}^N\E\Big[ \sup_{t\in [0,T]}\left|X^{i,N}_{t}-X^{i,g}_t\right|\Big] 
        & \leq   \frac{C}{\sqrt{N\beta_N}} + \frac{C}{N} + \frac{C\eps}{N} + C\|g\|_2\eps + \frac{C}{\eps^{n/2}\sqrt{N\beta_N}} + \frac{C}{\eps^{n/2}\sqrt{N}}. \\
    \end{split}
    \end{equation}
    Notice that if $\eps <1$ then we have:
    \begin{equation*}
    \frac{\varepsilon}{N} \le \frac1N\le \frac{1}{\sqrt{N}} \le \frac{1}{\sqrt{N\beta_N}}.
    \end{equation*}
    Hence, we can bound equation \eqref{eq:reduction.4-coro} to obtain: 
    \begin{equation*}
        \frac{1}{N}\sum_{i=1}^N\E\Big[ \sup_{t\in [0,T]}\left|X^{i,N}_{t}-X^{i,g}_t\right|\Big] 
         \leq \frac{C}{\sqrt{N\beta_N}}  + C\|g\|_2\eps + \frac{C}{\eps^{n/2}\sqrt{N\beta_N}}.
    \end{equation*}
    Now, let up pick $\varepsilon = \frac{1}{(N\beta_N)^{\frac{1}{n+1}}}$. Then,
    \begin{equation*}
        \varepsilon^{n/2}\sqrt{N\beta_N} = (N\beta_N)^{\frac{1}{2(n+1)}}.
    \end{equation*}
    Therefore, we have that: 
    \begin{equation*}
        \frac{C}{\sqrt{N\beta_N}}  + C\|g\|_2\eps + \frac{C}{\eps^{n/2}\sqrt{N\beta_N}} \le \frac{C}{(N\beta_N)^{\frac{1}{2(n+1)}}}.
    \end{equation*}
    This finishes the proof.
\end{proof}

We conclude this subsection by showing that the main convergence result Theorem \ref{thm:convergence-bounded-graphon-L1} can be extended to the case where the diffusion coefficient $\sigma$ also depends on the interaction term.
This is done under a stronger sparsity condition on the random graph.
Namely, we replace the condition $N\beta_N\to\infty$ by $N\beta^2_N\to\infty$.

\begin{theorem} 
\label{thm:convergence-bounded-graphon}
    Assume Conditions \ref{cond:initial-conditions}, 
    \ref{cond:coefficients}, \ref{cond:almost-continuity}, \ref{cond:bound-initial-cond} and Condition \ref{cond:graphon-cond-1} hold. Moreover, assume that $g$ is bounded and $N\beta_N^2 \rightarrow \infty$. 
    Then for every $t\in [0,T]$ it holds
    \begin{equation*}
         \frac{1}{N}\sum_{i=1}^N\E\Big[ \sup_{t\in [0,T]}\left|X^{i,N}_{t}-X^{i,g}_t\right|^2 \Big] +  \frac1N\sum_{i=1}^N\E\Big[d^2_{\mathrm{BL}}(m^{i,N}_{\X_t},  \nu^{i,g}_{X_t}) \Big] \xrightarrow{N\rightarrow \infty} 0. 
    \end{equation*}
\end{theorem}

\begin{proof}
    The proof of this result pretty much the same as that of Theorem \ref{thm:convergence-bounded-graphon-L1}.
    Therefore we only outline the step where the stronger condition $N\beta_N^2\to\infty$ is needed.
    Using standard Lipschitz SDE estimations we have
    \begin{equation}
    \label{eq:sigma.1}
    \begin{split}
        \frac{1}{N}\sum_{i=1}^N\E\Big[ \sup_{s\in [0,T]}\left|X^{i,N}_{s}-X^{i,g}_s\right|^2 \Big] & \leq \frac{C}{N}\sum_{i=1}^N \int_0^T 
        \E\big[d^2_{\mathrm{BL}}(m^{i,N}_{\X_u},\nu_{X_u}^{i,g})\big] du\\
        &\le  \frac{C}{N}\sum_{i=1}^N \int_0^T \E\big[d^2_{\mathrm{BL}}(m^{i,N}_{\X_u}, m^{i,N}_{\X^g_u})\big] + d^2_{\mathrm{BL}}(m^{i,N}_{\X^g_u},\nu_{X_u}^{i,g})\big]du.
    \end{split}
    \end{equation}
    The estimation of the second term on the right hand side is the same as in the proof of Theorem \ref{thm:convergence-bounded-graphon-L1}.
    The difficulty is in the estimation of the first term on the right hand side.
    This is obtained as follows:
    \begin{align}
    \notag
        \E\big[d^2_{\mathrm{BL}}(m^{i,N}_{\X_u}, m^{i,N}_{\X^g_u})\big] & \le 2\E\bigg[\Big(\frac1N\sum_{j=1}^N\Big(\frac{\zeta^N_{ij}}{\beta_N}- g_N(x_i,x_j)\Big)|X^{j,N}_t - X^{j,g}_t|\Big)^2\bigg]\\\label{eq:sigma.2}
        & \quad + 2\E\bigg[\Big(\frac1N\sum_{j=1}^N g_N(x_i,x_j) |X^{j,N}_t - X^{j,g}_t|\Big)^2\bigg].
    \end{align}
    Since $g_N$ is bounded, we have
    \begin{equation}
    \label{eq:sigma.3}
        \begin{split}
        \frac{1}{N}\sum_{i=1}^N\E\bigg[\Big(\frac{1}{N}\sum_{j=1}^N g_N(x_i,x_j) |X^{j,N}_t-X^{j,g}_t|\Big)^2\bigg]
        & \le 
        \frac{C}{N}\sum_{j=1}^N \E\Big[ |X^{j,N}_t-X^{j,g}_t|^2\Big].
    \end{split}
    \end{equation}
    Now, put
    \begin{equation*}
        \Gamma_{ij} = \zeta_{ij}^N - \beta_Ng_N(x_i,x_j) \quad \text{ and }\quad \Delta X^j_t = X^{j,N}_t - X^{j,g}_t
    \end{equation*}
    and recall that as $\Gamma_{ij}$ is a centered Bernoulli, $\Gamma_{ij}$ and $\Gamma_{lj}$ are independent and $g_N$ are uniformly bounded we have
    \begin{equation}
    \label{eq:bounds-centered-ber}
        \E[\Gamma_{ij}^4] \leq  g_N(x_i,x_j)\beta_N \leq C\beta_N, \quad \E[\Gamma_{ij}]=0 \quad \text{ and }\quad\E[\Gamma_{ij}^2\Gamma_{lj}^2]\leq g_N(x_i,x_j)g_N(x_l,x_j)\beta_N^2\leq C\beta_N^2.
    \end{equation}
    Thus, using Hölder inequality and Lemma \ref{lemma:L2-boundness-solutions} we obtain
    \begin{equation}
    \label{eq:bound-zeta-variance-prelim}
        \begin{split}
        & \frac{1}{N}\sum_{i=1}^N\E\bigg[\Big(\frac1N\sum_{j=1}^N\Big(\frac{\zeta^N_{ij}}{\beta_N}- g_N(x_i,x_j)\Big)|X^{j,N}_t - X^{j,g}_t|\Big)^2\Big] \\
        & = \frac{1}{N^3\beta_N^2} \sum_{j=1}^N \E\bigg[\Big(\sum_{i=1}^N\Gamma_{ij}^2\Big)|\Delta X^j_t|^2\bigg] + \frac{1}{N^3\beta_N^2} \sum_{j\neq k = 1}^N \E\bigg[\Big(\sum_{i=1}^N\Gamma_{ij}\Gamma_{ik}\Big)|\Delta X_t^j\|\Delta X_t^{k}|\bigg]\\
        & \leq \frac{C}{N^3\beta_N^2} \sum_{j=1}^N \E\bigg[\Big(\sum_{i=1}^N\Gamma_{ij}^2\Big)^2\bigg]^{1/2} + \frac{1}{N^3\beta_N^2} \sum_{j\neq k = 1}^N \E\bigg[\Big(\sum_{i=1}^N\Gamma_{ij}\Gamma_{ik}\Big)^2\bigg]^{1/2}\E\big[|\Delta X_t^j|^4\big]^{1/4}\E\big[|\Delta X_t^{k}|^4\big]^{1/4}\\
        & \leq \frac{C}{N^3\beta_N^2} \sum_{j=1}^N \E\bigg[\sum_{i=1}^N\Gamma_{ij}^4 + \sum_{i\neq l=1}^N\Gamma_{ij}^2\Gamma_{lj}^2 \bigg]^{1/2} + \frac{C}{N^3\beta_N^2} \sum_{j\neq k = 1}^N \E\bigg[\Big(\sum_{i=1}^N\Gamma_{ij}\Gamma_{ik}\Big)^2\bigg]^{1/2}.
    \end{split}
    \end{equation}
    Replacing equation \eqref{eq:bounds-centered-ber} in equation \eqref{eq:bound-zeta-variance-prelim} we continue the estimation as
    \begin{equation}
    \label{eq:bound-zeta-variance}
        \begin{split}
        &  \frac{1}{N}\sum_{i=1}^N\E\bigg[\Big(\frac1N\sum_{j=1}^N\Big(\frac{\zeta^N_{ij}}{\beta_N}- g_N(x_i,x_j)\Big)|X^{j,N}_t - X^{j,g}_t|\Big)^2\bigg] \\
    & \leq \frac{C}{N^2\beta_N^2}\left(N\beta_N + N^2\beta_N^2\right)^{1/2}  + \frac{C}{N^3\beta_N^2}N^2\left(N\beta_N^2\right)^{1/2}\\
    & \le \frac{C}{(N\beta_N)^{3/2}} + \frac{C}{N\beta_N} + \frac{C}{\sqrt{N\beta^2_N}}.
    \end{split}
    \end{equation}
    Thus, putting together \eqref{eq:sigma.1} \eqref{eq:sigma.2} and \eqref{eq:sigma.3}, and using Gronwall's inequality we conclude the proof.
\end{proof}

\subsection{Extension to random points}

In the next corollary we allow the points $(x_i)_{1\leq i\leq N}$ to be chosen randomly on the unit interval. 
In that case, we need to state the convergence result in terms of convergence in probability with respect to the randomness associated to the sampling of the points $(x_i)_{1\leq i\leq N}$.

\begin{condition}
\label{cond:random-points}
    For every $N \in \N$ consider independent and uniformly distributed and ordered random variables $U_2,\dots, U_{N}$ on the unit interval and set $U_1=0$ and $U_{N+1} = 1$ with probability $1$. We will assume that we sample the variables $U_1,\dots, U_{N+1}$ on a probability space $\widetilde \P$. Furthermore, if $x_i$ is the realization of $U_i$, assume Condition \ref{cond:graphon-cond-1}$(1)$ hold and assume that for every $N$, there exists a set of independent random variables $(\zeta^N_{ij})_{1\leq i,j\leq N}$ with distribution $\zeta^N_{ij} = \mathrm{Bernouilli}(g(x_i,x_j)\beta_N)$ where $\beta_N$ is a sequence of non-negative reals such that $N\beta_N\rightarrow \infty$, $\beta_N\leq 1$ and $\beta_Ng(x_i,x_j)\leq 1$.
\end{condition}

\begin{corollary}
\label{coro:random-points}
    Assume Conditions \ref{cond:initial-conditions}, \ref{cond:coefficients}, \ref{cond:random-points} and \ref{cond:bound-initial-cond} holds. 
    If $g$ is Lipschitz, then for every $\eta>0$ we have
    \begin{equation*}
      \lim_{N\rightarrow \infty} \widetilde \P\Big( \frac{1}{N}\sum_{i=1}^N\E\Big[\sup_{t\in [0,T]}|X_t^{U_i,N}-X_t^{U_i,g}|\Big| (U_i)_{1\leq i\leq N} \Big] > \eta \Big) = 0.
    \end{equation*}
\end{corollary}

\begin{proof}
    We start by bounding $\frac{1}{N}\E\left[\sup_{t\in [0,T]}|X_t^{U_i,N}-X_t^{U_i,g}|\mid (U_i)_{1\leq i\leq N}\right]$. 
    We will denote by $x_i$ a given realization of $U_i$.

    Observe that in the proof of Theorem \ref{thm:convergence-bounded-graphon-L1}, the only steps where we are using that the points are deterministic is where we need to transform an average into an integral, see for instance \eqref{eq:from.average.to.integral} (among many other places). 
    In particular, equations \eqref{eq:partial-result-convergence-L1} and \eqref{eq:proof.trans.to.stability-L1} still hold in the present random case. 
    We will now explain how to continue the proof from \eqref{eq:proof.trans.to.stability-L1}, and explain only how to deal with random points.

    Let us now define $h_N(x,y) := g(f_N(x),f_N(y))$.
    By equation \eqref{eq:before.intro.norms} and Proposition \ref{prop:stability-of-g-integral} we have
\begin{equation}
\begin{split}
    \label{eq:d-nu-g-h_N-intermediate-random}
    & \frac{1}{N}\sum_{i=1}^N d_{\mathrm{BL}}(\nu^{i,g}_{X_u}, \nu^{i,h_N}_{X_u})\\
    & \le   \frac{1}{N}\sum_{i=1}^N \sup_{\phi \in  1\text{-Lip}_b}\bigg(\int_I |g(x_i,y) - h_N(x_i,y)|^2 \lambda(dy)\bigg)^{1/2}\bigg( \int_I (\E[|\phi(X^{y,g}_u)|])^2 \lambda(dy)\bigg)^{1/2} \\
    & \quad +  \frac{1}{N}\sum_{i=1}^N \sup_{\phi \in 1\text{-Lip}_b}\bigg(\int_I h_N(x_i,y)^2\lambda(dy)\bigg)^{1/2} \bigg(\int_I\E\big[|\phi(X^{y,g}_u) - \phi(X^{y,h_N}_u)|^2\big]\lambda(dy)\bigg)^{1/2} \\
    & \le \frac{1}{N}\sum_{i=1}^N \bigg(\int_I |g(x_i,y) - h_N(x_i,y)|^2 \lambda(dy)\bigg)^{1/2} +  C \bigg(\int_I\E\big[|X^{y,g}_u - X^{y,h_N}_u|^2\big]\lambda(dy)\bigg)^{1/2}. \\
     & \le   \frac{1}{N}\sum_{i=1}^N \bigg(\int_I |g(x_i,y) - h_N(x_i,y)|^2 \lambda(dy)\bigg)^{1/2} + C \|g-h_N\|_2.
\end{split}
\end{equation}
The analogue of equation \eqref{eq:dist-m-h_N-g-intermediate-L1} is obtained from the bound on \eqref{eq:first-lemma-eq-1} and Proposition \ref{prop:stability-of-g-average} before introducing the norms
\begin{equation}
    \label{eq:dist-m-h_N-g-intermediate-random}
    \begin{split}
        \frac{1}{N}\sum_{i=1}^N \E\Big[d_{\mathrm{BL}}(m^{i,N}_{\X^{h_N}_u}, m^{i,N}_{\X^{g}_u})\Big]  \leq & \frac{\|g\|_\infty}{N}\sum_{i=1}^N \E\Big[\big|X^{i,h_N}_u-X^{i,g}_u\big|\Big] \\
        \leq & \frac{C}{N}\sum_{i=1}^N \int_I\left|h_N(x_i,y)-g(x_i,y)\right|\lambda(dy) +  C \|g\circ (f_N,id)\|_{2} \|g-h_N\|_2\\
        \leq & \frac{C}{N}\sum_{i=1}^N \int_I\left|h_N(x_i,y)-g(x_i,y)\right|\lambda(dy) +  C(\|g\|_\infty) \|g-h_N\|_2.\\
    \end{split}
\end{equation}

Hence, combining equations \eqref{eq:partial-result-convergence-L1}, \eqref{eq:proof.trans.to.stability-L1}, \eqref{eq:d-nu-g-h_N-intermediate-random} and \eqref{eq:dist-m-h_N-g-intermediate-random} we obtain
\begin{equation}
\label{eq:reduction.2-L1-coro}
\begin{split}
    \frac{1}{N}\sum_{i=1}^N\E\Big[ \sup_{s\in [0,t]}\big|X^{i,N}_{t}-X^{i,g}_t\big| \Big] & \leq   \frac{C}{\sqrt{N\beta_N}}  +\frac{C}{N}\sum_{i=1}^N \int_0^t\E\Big[d_{\mathrm{BL}}(\nu^{i,h_N}_{X_u}, m^{i,N}_{\X_u^{h_N}})\Big]du + C \|g-h_N\|_2\\
    & \quad + \frac{C}{N}\sum_{i=1}^N \bigg(\int_I \left|h_N(x_i,y)-g(x_i,y)\right|^2\lambda(dy)\bigg)^{1/2}\\
    &\le   \frac{C(\varepsilon)}{\sqrt{N\beta_N}} 
    +  C \|g-h_N\|_2 + \frac{C}{N}\sum_{i=1}^N \bigg(\int_I \left|h_N(x_i,y)-g(x_i,y)\right|^2\lambda(dy)\bigg)^{1/2}\\
    &\quad +\frac{C(\eps)}{\sqrt{N}} + Cn\sqrt{\varepsilon},
\end{split}
\end{equation}
where we used \eqref{eq:reduction-reg-L1}, Lemma \ref{lemma:distance-regularised-measures-L1} and the fact that $h_N(x_i,x_j) = g(x_i,x_j)$. Notice that here we are using the same regularization of measure argument as in the proof of Theorem \ref{thm:convergence-bounded-graphon-L1}. Hence, we need Condition \ref{cond:bound-initial-cond} to be able to claim that the bound still holds.

Now notice that by Lipschitz-continuity of $g$ we have
\begin{equation}
\label{eq:bound-random-norm}
\begin{split}
    \|g-h_N \|_2 & = \bigg(\int_I\int_I |g(x,y)-h_N(x,y)|^2\lambda(dx)\lambda(dy) \bigg)^{1/2}\\
    & = \bigg(\sum_{i,j=1}^N \int_{x_i}^{x_{i+1}}\int_{x_j}^{x_{j+1}} |g(x,y)-g(x_i,x_j)|^2\lambda(dx)\lambda(dy) \bigg)^{1/2}\\
    & \leq \bigg( 2 \sum_{i,j=1}^N \int_{x_i}^{x_{i+1}}\int_{x_j}^{x_{j+1}} |x-x_i|^2 + |y-x_j|^2\lambda(dx)\lambda(dy)\bigg)^{1/2} \\
    & \leq C \sum_{i=1}^N (x_{i+1}-x_i)^{3/2}. \\
\end{split}
\end{equation}

The same argument yields
\begin{equation}
\label{eq:second-bound-random-points}
    \frac{1}{N}\sum_{i=1}^N \bigg(\int_I \left|h_N(x_i,y)-g(x_i,y)\right|^2\lambda(dy)\bigg)^{1/2} \leq C \sum_{i=1}^N (x_{i+1}-x_i)^{3/2}.
\end{equation}
Combining equations \eqref{eq:reduction.2-L1-coro}, \eqref{eq:bound-random-norm} and \eqref{eq:second-bound-random-points} we obtain
\begin{equation*}
    \frac{1}{N}\sum_{i=1}^N\E\Big[ \sup_{s\in [0,t]}\big|X^{i,N}_{t}-X^{i,g}_t\big| \Big] \leq   Cn\sqrt{\eps} +\frac{C(\eps)}{\sqrt{N\beta_N}} + \frac{C(\eps)}{\sqrt{N}}  +  C \sum_{i=1}^N (x_{i+1}-x_i)^{3/2}.\\
\end{equation*}
For the random variables $(U_i)_{1\le i\le N}$, we thus have for every $\eta>0$
\begin{equation*}
\begin{split}
    & \widetilde\P\bigg( \frac{1}{N}\sum_{i=1}^N\E\Big[\sup_{t\in [0,T]}|X_t^{U_i,N}-X_t^{U_i,g}|\Big| (U_i)_{1\leq i\leq N} \Big] > \eta \bigg) \\
    & \leq \widetilde \P\bigg( Cn\sqrt{\eps} +\frac{C(\eps)}{\sqrt{N\beta_N}} + \frac{C(\eps)}{\sqrt{N}}  +  C \sum_{i=1}^N (U_{i+1}-U_i)^{3/2}  > \eta \bigg)\\
    & \leq \frac{1}{\eta}\bigg(Cn\sqrt{\eps} +\frac{C(\eps)}{\sqrt{N\beta_N}} + \frac{C(\eps)}{\sqrt{N}}  +  C \sum_{i=1}^N \E\left[(U_{i+1}-U_i)^{3/2}\right]\bigg)\\
    & \leq \frac{1}{\eta}\bigg(Cn\sqrt{\eps} +\frac{C(\eps)}{\sqrt{N\beta_N}} + \frac{C(\eps)}{\sqrt{N}}  +  C \sum_{i=1}^N \E\left[(U_{i+1}-U_i)^{2}\right]^{3/4}\bigg),\\
\end{split}
\end{equation*}
where the last inequality follows by Hölder's inequality with parameter $4/3$. By \cite[Section $2.1$]{pyke-spacings} we know that $\E[(U_{i+1}-U_i)^2] = \frac{2}{(N+1)(N+2)}$ which implies that
\begin{equation}
\label{eq:random-second-est}
\begin{split}
     & \limsup_{N\rightarrow \infty} \widetilde \P\bigg( \frac{1}{N}\E\Big[\sup_{t\in [0,T]}|X_t^{U_i,N}-X_t^{U_i,g}| \Big| (U_i)_{1\leq i\leq N}\Big] > \eta \bigg) \\
     & \leq \limsup_{N\rightarrow \infty} \frac{1}{\eta}\bigg(Cn\sqrt{\eps} +\frac{C(\eps)}{\sqrt{N\beta_N}} + \frac{C(\eps)}{\sqrt{N}}  +  C\Big(\frac{N^{4/3}}{(N+1)(N+2)}\Big)^{3/4}\bigg)\\
     & = \frac{Cn\sqrt{\eps}}{\eta}.\\
\end{split}
\end{equation}
As equation \eqref{eq:random-second-est} holds for every $\eps>0$, the desired limit follows.
\end{proof}

Before stating the Weak Law of Large Numbers we will need the following Lemma:
\begin{lemma}
\label{lemma:independence-BM} Define for every $x\in I$
    \begin{equation*}
        A_x = \{y \in I : B^x \text{ and } B^y \text{ are independent}\}
    \end{equation*}
    and define $A$ to be the set of points such that $\lambda(A_x)=1$. Assume that $U_1\in A$. Then, for every $N\in \N$, we have that
    \begin{equation*}
        \widetilde \P\left( B^{U_1}, \dots, B^{U_N} \text{ are pairwise independent}\right) = 1.
    \end{equation*}
\end{lemma}

\begin{proof}
    This follows from the fact that the Brownian motions $(B^{x})_{x\in I}$ are essentially pairwise independent. Formally, As $\lambda$ is an extension of the Lebesgue measure, we can obtain the result by induction on $N$. For $N=2$, we get that:
    \begin{equation*}
        \widetilde\P\left(B^{U_1}, B^{U_2} \text{ are independent} \right) = \widetilde\P\left( U_2\in A_{0}|U_1=0\right) \widetilde\P\left(U_1=0\right)= \widetilde\P\left( U_2\in A_{0}\right)=1.
    \end{equation*}
    This holds because we know $U_1=0$ with probability $1$ and the last term is $1$ because as $U_2$ is sorted uniformly at random then $\widetilde \P( U_2\in A_{0})=\lambda(A_0)=1$.

    If we know that $B^{U_1},\dots, B^{U_k}$ are pairwise independent (p.i). Then, 
    \begin{equation*}
    \begin{split}
        \widetilde\P\left(B^{U_1},\dots, B^{U_{k+1}} \text{ are p.i}\right) &= \widetilde\P\left(B^{U_1},\dots, B^{U_{k+1}} \text{ are p.i }|B^{U_1},\dots, B^{U_{k}} \text{ are p.i }\right) \widetilde\P\left(B^{U_1},\dots, B^{U_{k}} \text{ are p.i}\right)\\
        &= \widetilde \P\left(B^{U_1},\dots, B^{U_{k+1}} \text{ are p.i }|B^{U_1},\dots, B^{U_{k}} \text{ are p.i }\right).
    \end{split}
    \end{equation*}
    We will show the complement is zero:
    \begin{equation*}
    \begin{split}
        \widetilde\P\left(B^{U_1},\dots, B^{U_{k+1}} \text{ are not p.i }|B^{U_1},\dots, B^{U_{k}} \text{ are p.i }\right) & \leq  \widetilde\P\left(U_{k+1}\not\in A_{U_1} \right) + \dots + \widetilde\P\left(U_{k+1}\not\in A_{U_k} \right).
    \end{split}
    \end{equation*}
    However, each of these summands is zero as for $i\neq k+1$
    \begin{equation*}
         \widetilde\P\left(U_{k+1}\not\in A_{U_i} \right) = \int_I \widetilde\P\left(U_{k+1}\not\in A_{x}\big| U_i = x \right)\widetilde\P\left(U_i \in \lambda(dx)\right) = \int_A \widetilde\P\left(U_{k+1}\not\in A_{x}\big| U_i = x \right)\widetilde\P\left(U_i \in \lambda(dx)\right) = 0,
    \end{equation*}
    where the last equality follows from the fact that for every $x\in A$, $\widetilde\P\left(U_{k+1}\not\in A_{x} \right) = \mathrm{Leb}(A^c_x) = \lambda(A^c_x)=0$. 
\end{proof}

\begin{corollary}
\label{coro:WLLN}
    Assume Conditions \ref{cond:initial-conditions}, \ref{cond:coefficients}, \ref{cond:random-points} and \ref{cond:bound-initial-cond} holds. Moreover, assume that the hypothesis of Lemma \ref{lemma:independence-BM} hold. If $g$ is Lipschitz, then for every $\eta>0$ we have
    \begin{equation*}
      \lim_{N\rightarrow \infty} \widetilde\P\bigg( \frac{1}{N}\sum_{i=1}^N\E\bigg[\Big|\frac{1}{N}\sum_{i=1}^N X_t^{i,N}-\int_I X_t^{y,g}\lambda(dy)\Big| \bigg| (U_i)_{1\leq i\leq N} \bigg] > \eta \bigg) = 0.
    \end{equation*}
\end{corollary}

\begin{proof}
    Notice that by Markov's inequality we have that: 
    \begin{equation*}
        \widetilde\P\Big( \E\Big[\Big|\frac{1}{N}\sum_{i=1}^N X_t^{U_i,N}-\int_I X_t^{y,g}\lambda(dy)\Big|\Big| (U_i)_{1\leq i\leq N} \Big] > \eta \Big) \leq \frac{\E\bigg[\E\Big[\Big|\frac{1}{N}\sum_{i=1}^N X_t^{U_i,N}-\int_I X_t^{y,g}\lambda(dy)\Big|\bigg| (U_i)_{1\leq i\leq N} \Big]\bigg]}{\eta}.
    \end{equation*}
    As in the Proof of Theorem \ref{thm:convergence-bounded-graphon-L1} we introduce an approximating step graphon $h_N$, which in this case we will specialize to be $h_N(x,y)=g(f_N(x),f_N(y))$. Notice that in this case $f_N$ is defined as
    \begin{equation*}
    f_N(x) := \sum_{i=1}^{N-1} U_i {\mathbf{1}_{\{x\in [U_i,U_{i+1})\}}} + U_N \mathbf{1}_{\{x\in [U_N,1]\}}.
    \end{equation*}
    We obtain the estimate
    \begin{equation*}
    \begin{split}
        & \E\bigg[ \Big|\frac{1}{N}\sum_{i=1}^N X^{i,N}_t  - \int_I X^{y,g}_t\lambda(dy)\Big|\bigg| (U_i)_{1\leq i\leq N}\bigg]\\
        & \leq 
        \E\bigg[ \frac{1}{N}\sum_{i=1}^N   \big|X^{i,N}_t  -  X_t^{U_i,g}\big|\Big| (U_i)_{1\leq i\leq N}\bigg] + \E\bigg[\frac{1}{N}\sum_{i=1}^N \big|  X^{U_i,g}_t  - X_t^{U_i,h_N}\big|\Big| (U_i)_{1\leq i\leq N}\bigg]\\
        & \quad +\E\bigg[\Big|\frac{1}{N}\sum_{i=1}^N X_t^{U_i,h_N}  - \int_I X^{y,h_N}_t \lambda(dy) \Big|\bigg| (U_i)_{1\leq i\leq N} \bigg]  + \E\bigg[\int_I \big| X^{y,h_N}_t - X^{y,g}_t\big|\lambda(dy)\Big| (U_i)_{1\leq i\leq N}\bigg]\\
        & = \cR_1 + \cR_2 + \cR_3 + \cR_4. \\
    \end{split}
    \end{equation*}

    We need to show that $\E[\cR_1+\cR_2+\cR_3+\cR_4] \rightarrow 0$ as $N\rightarrow \infty$. $\E[\cR_1]\rightarrow 0$ by Corollary \ref{coro:random-points}. $\cR_2$ can be dealt with using Proposition \ref{prop:stability-of-g-average} we obtain\footnote{Notice that in this case $f_N$ is a random function depending on the random points $\{U_i\}_{1\leq i \leq N}$.}: 
    \begin{equation*}
    \begin{split}
        \cR_2 \leq & \frac{C}{N} \sum_{i=1}^N \int_I |g(U_i,f_N(y))-g(U_i,y)| \lambda(dy) + \frac{C}{N}\sum_{i=1}^N\int_I|g(U_i,y)|^2\lambda(dy) \|g-h_N\|_2\\
        \leq & C\sum_{i=1}^N (U_{i+1}-U_i)^{3/2},\\
    \end{split}
    \end{equation*}
    where we used the fact that $g$ is Lipschitz and bounded and the same idea as in equations \eqref{eq:bound-random-norm} and  \eqref{eq:second-bound-random-points}. Taking expectation we obtain with the same argument as in equation \eqref{eq:random-second-est}: 
    \begin{equation*}
        \E[\cR_2] \leq C\E\left[\sum_{i=1}^N (U_{i+1}-U_i)^{3/2}\right] \leq \left(\frac{CN^{4/3}}{(N+1)(N+2)}\right)^{3/4} \rightarrow 0. 
    \end{equation*}
    $\cR_4$ is straightforward to estimate using the stability of the solutions of the graphon SDE (Proposition \ref{prop:stability-of-g-integral}):
    \begin{equation*}
    \begin{split}
        \E\left[\cR_4\right] \leq & C \E\left[ \left(\int_I\int_I |g(f_N(x),f_N(y))-g(x,y)|^2\right)^{1/2}\right] = C \E\left[\sum_{i=1}^N(U_{i+1}-U_i)^{3/2}\right]\rightarrow 0.\\
    \end{split}
    \end{equation*}
    We are left to estimate $\cR_3$.    
    By using the Exact Law of Large Numbers we can estimate: 
    \begin{equation*}
    \begin{split}
        \cR_3 & \leq  \E\left[\Big| \frac{1}{N}\sum_{i=1}^NX^{U_i,h_N}_t - \int_I \E[X^{y,h_N}_t\Big| (U_i)_{1\leq i\leq N}]\lambda(dy)\Big|\Big| (U_i)_{1\leq i\leq N}\right]\\
        \leq & \E\left[\Big| \frac{1}{N}\sum_{i=1}^NX^{U_i,h_N}_t - \frac{1}{N}\sum_{i=1}^N \E\left[X^{U_i,h_N}_t\Big| (U_i)_{1\leq i\leq N}\right]\Big|\Big| (U_i)_{1\leq i\leq N}\right] \\
        &\quad + \E\left[\Big|\frac{1}{N}\sum_{i=1}^N \E\left[X^{U_i,h_N}_t\Big| (U_i)_{1\leq i\leq N}\right] - \int_I \E[X^{y,h_N}_t\big| (U_i)_{1\leq i\leq N}]\lambda(dy)\Big|\Big| (U_i)_{1\leq i\leq N} \right] \\
        & \leq  \E\left[\Big( \frac{1}{N}\sum_{i=1}^N X^{U_i,h_N}_t - \E\left[X^{U_i,h_N}_t\Big| (U_i)_{1\leq i\leq N}\right]\Big)^2\Big| (U_i)_{1\leq i\leq N}\right]^{1/2}   \\
        & \quad + \E\left[\Big|\frac{1}{N}\sum_{i=1}^N\left(\frac{1}{N}-(U_{i+1}-U_{i})\right)\E[X^{U_i,h_N}_t\big| (U_i)_{1\leq i\leq N}]\Big|^2\Big| (U_i)_{1\leq i\leq N} \right]^{1/2} \\
        & \leq  \E\left[\frac{1}{N^2}\sum_{i=1}^N \Big( X^{U_i,h_N}_t - \E\left[X^{U_i,h_N}_t\Big| (U_i)_{1\leq i\leq N}\right]\Big)^2 \Big| (U_i)_{1\leq i\leq N}\right]^{1/2} \\
        & \quad + \left(\frac{1}{N}\sum_{i=1}^N\left(\frac{1}{N}-(U_{i+1}-U_{i})\right)^2\frac{1}{N}\sum_{i=1}^N\E\left[ \Big|X^{U_i,h_N}_t\Big|^2\Big| (U_i)_{1\leq i\leq N} \right]\right)^{1/2}.
    \end{split}
    \end{equation*}
    Where we used the fact that by Lemma \ref{lemma:independence-BM} the processes $X^{U_i}$ and $X^{U_j}$ are independent with probablity $1$. By Lemma \ref{lemma:L2-boundness-solutions} we can bound: 
    \begin{equation*}
    \begin{split}
        \cR_3 \leq &  \frac{1}{\sqrt{N}}\E\left[\frac{1}{N}\sum_{i=1}^N \Big| X^{U_i,h_N}_t |^2 \Big| (U_i)_{1\leq i\leq N}\right]^{1/2} + \left(\frac{1}{N}\sum_{i=1}^N\left(\frac{1}{N}-(U_{i+1}-U_{i})\right)^2\right)^{1/2} \\
        \leq &  \frac{1}{\sqrt{N}} + \left(\frac{1}{N}\sum_{i=1}^N\left(\frac{1}{N}-(U_{i+1}-U_{i})\right)^2\right)^{1/2}.
    \end{split}
    \end{equation*}
    Taking expectation we obtain: 
    \begin{equation*}
    \begin{split}
        \E\left[\cR_3\right] \leq &\frac{1}{\sqrt{N}} + \left(\frac{1}{N} \sum_{i=1}^N \frac{1}{N^2} -\frac{2}{N}\E[U_{i+1}-U_i] + \E\left[(U_{i+1}-U_i)^2\right]\right)^{1/2}\\
        \leq & \frac{1}{\sqrt{N}} + \left(\frac{1}{N} \sum_{i=1}^N \frac{1}{N^2} -\frac{2}{N}+ \E\left[(U_{i+1}-U_i)^2\right]\right)^{1/2}\\
        \leq & \frac{1}{\sqrt{N}} + \frac{1}{\sqrt{N}}\left(\frac{1}{N }- 2 + \frac{2N}{(N+1)(N+2)}\right)^{1/2}\rightarrow 0.\\
    \end{split}
    \end{equation*}

    Where we used \cite[Section $2.1$]{pyke-spacings} for the last bound. This finishes the proof.

\end{proof}

\subsection{Extensions to the unbounded case}

We start by a moment bound analogous to Lemma \ref{lemma:L2-boundness-solutions} but for the unbounded case. 

\begin{lemma}
\label{lemma:L2-boundness-solutions-g-unbounded}
    Assume Condition \ref{cond:coefficients} and Condition \ref{cond:graphon-cond-1} hold. Assume that there exists a constant $C'>0$ independent of $N$ such that $\sup_{i\leq N}\E[|\xi^i|^4]\leq C$. Then, there exists a constant $C>0$ independent of $N$ such that : 
    \begin{equation*}
        \sup_{0\leq i\leq N}\E\left[\sup_{t\in [0,T]}|X^{i,g}_t|^{4}\right] \leq C.
    \end{equation*}
    Moreover, if $\|g_N-g\|_4\rightarrow 0$ and $\|g\|_4 < \infty$ we have that
    Then there exists $C>0$ independent of $N$ such that
    \begin{equation*}
        \frac{1}{N}\sum_{i=1}^N \E\Big[\sup_{t\in [0,T]}|X^{i,N}_t|^{4}\Big] \leq C.
    \end{equation*}
\end{lemma}

\begin{proof}
    As in the proof of Lemma \ref{lemma:L2-boundness-solutions} we have that
    \begin{equation*}
    \begin{split}
          \frac{1}{N}\sum_{i=1}^N \E\Big[\sup_{t\in [0,T]}|X^{i,N}_t|^{4}\Big] \leq &  \frac{C}{N}\sum_{i=1}^N\E\bigg[1+|\xi^i|^{4} + \int_0^t \sup_{r\in [0,u]}|X^{i,N}_r|^{4} + \Big( \frac{1}{N}\sum_{j=1}^N \frac{\zeta^{N}_{ij}}{\beta_N}\Big)^{4} du \bigg]. \\
    \end{split}
    \end{equation*}
    Notice that the first two terms are clearly bounded and the third one can be absorbed using Gronwall's inequality. Hence, we only need to show that the last term is uniformly bounded in $N$. To do this we will expand the power as follows: 
    \begin{equation*}
    \begin{split}
        \frac{1}{N}\sum_{i=1}^N \E\left[\Big( \frac{1}{N}\sum_{j=1}^N \frac{\zeta^{N}_{ij}}{\beta_N}\Big)^{4}\right] \leq  & \frac{1}{N}\sum_{i=1}^N \frac{C}{(N\beta_N)^4} \left(\sum_{j_1=1}^N \E[(\zeta_{ij_1}^N)^4] + \sum_{j_1,j_2=1}^N \left(\E[(\zeta_{ij_1}^N)^3]\E[\zeta_{ij_2}^N] + \E[(\zeta_{ij_1}^N)^2]\E[(\zeta_{ij_2}^N)^2]\right)\right. \\
        & \left. +  \sum_{j_1,j_2,j_3=1}^N \E[(\zeta_{ij_1}^N)^2]\E[\zeta_{ij_2}^N]\E[\zeta_{ij_3}^N] + \sum_{j_1,j_2,j_3,j_4=1}^N \E[\zeta_{ij_1}^N]\E[\zeta_{ij_2}^N]\E[\zeta_{ij_3}^N]\E[\zeta_{ij_4}^N]\right)\\
        = &  \frac{1}{N}\sum_{i=1}^N\left(\frac{C}{N^4\beta_N^3}\sum_{j_1=1}^N g_N(x_i,x_{j_1}) +  \frac{C}{N^4\beta_N^2} \sum_{j_1,j_2=1}^N g_N(x_i,x_{j_1})g_N(x_i,x_{j_2})\right.\\
        &   + \frac{C}{N^4\beta_N} \sum_{j_1,j_2,j_3=1}^N g_N(x_i,x_{j_1})g_N(x_i,x_{j_2})g_N(x_i,x_{j_3}) \\
        & \left. + \frac{C}{N^4} \sum_{j_1,j_2,j_3,j_4=1}^N g_N(x_i,x_{j_1})g_N(x_i,x_{j_2})g_N(x_i,x_{j_3})g_N(x_i,x_{j_4}) \right) \\
        = & \frac{\|g_N\|_1}{(N\beta_N)^3} + \frac{\|g_N\|^2_2}{(N\beta_N)^2}  + \frac{\|g_N\|^3_3}{N\beta_N} + \|g_N\|_4.\\
    \end{split}
    \end{equation*}
    The previous calculation finishes the proof. 
\end{proof}

\begin{theorem}
    \label{thm:convergence-unbounded-graphon}
    Assume Conditions \ref{cond:initial-conditions}, 
    \ref{cond:coefficients}, \ref{cond:almost-continuity}, \ref{cond:bound-initial-cond} and \ref{cond:graphon-cond-1} hold. Furthermore, assume that:
    \begin{itemize}
        \item $N\beta_N^2 \rightarrow \infty$;
        \item $\|g\|_4 < +\infty$ and $\|g_N-g\|_4\rightarrow 0$;
        \item $g\circ (f_N,\mathrm{id})$ is dominated by a square integrable function;
    \end{itemize}
    
    Then for every $t\in [0,T]$ it holds
    \begin{equation}
    \label{eq:result-unbounded}
         \frac{1}{N}\sum_{i=1}^N\E\Big[ \sup_{t\in [0,T]}\left|X^{i,N}_{t}-X^{i,g}_t\right|^2 \Big] +  \frac1N\sum_{i=1}^N\E\Big[d^2_{\mathrm{BL}}(m^{i,N}_{\X_t},  \nu^{i,g}_{X_t}) \Big] \xrightarrow{N\rightarrow \infty} 0. 
    \end{equation}
\end{theorem}

As before it is straightforward to obtain a Corollary in terms of Condition \ref{cond:graphon-cond-2}. 

The proof follows from noticing that we can apply Theorem \ref{thm:convergence-unbounded-graphon} using $g_N(x,y)=g(f_N(x),f_N(y))$ and that by dominated convergence Theorem $\|g_N-g\|\rightarrow 0$. Hence, we can apply the Theorem. 

\begin{corollary}
    Assume Conditions \ref{cond:initial-conditions}, 
    \ref{cond:coefficients}, \ref{cond:almost-continuity}, \ref{cond:bound-initial-cond} and \ref{cond:graphon-cond-2} hold. Furthermore, assume that $\|g\|_4<+\infty$, $N\beta_N\rightarrow \infty$, $g\circ (f_N,\mathrm{id})$ and $g\circ (f_N,f_N)$ are dominated by a square integrable function. Then, for every $t\in [0,T]$ equation \eqref{eq:result-unbounded}.    
\end{corollary}

\begin{proof}
    Fix $1\geq \eta>0$ and consider the function $\widetilde g$ defined as a continuous function such that $\|g-\widetilde g\|_2 \leq \eta$. 
    Further define $\widetilde g_N(x,y) = \widetilde g(x_i,x_j)$ if $x\in [x_i,x_{i+1}), y\in [x_j,x_{j+1})$ and $\widetilde\zeta_{ij}^{N} \sim \mathrm{Bernoulli}(\beta_N\widetilde g_N(x_i,x_j))$  where $(\widetilde \zeta_{ij}^{N})_{i,j}$ are independent. 
    Denote $ m^{i,N}_{\widetilde \X} = \frac{1}{N}\sum_{j=1}^N \frac{\widetilde\zeta_{ij}^{N}}{\beta_N} \delta_{\widetilde X^{i,N}}$ where $\widetilde X^{i,N}$ solves the SDE \eqref{eq:part.system.B} with $(\zeta^N_{ij})_{1\le i,j\le N}$ replaced by $(\widetilde\zeta^N_{ij})_{1\le i,j\le N}$. 
    Then, we have
\begin{equation}
\label{eq:unbounded-est-1}
\begin{split}
    \frac{1}{N}\sum_{i=1}^N \E\Big[\sup_{s\in [0,T]}\left| X^{i,N}_s - X^{i,g}_s\right|^2\Big] & \leq \frac{C}{N}\sum_{i=1}^N \E\Big[\sup_{s\in [0,T]}\big| X^{i,N}_s - \widetilde X^{i,N}_s\big|^2\Big] + \frac{C}{N}\sum_{i=1}^N \E\Big[\sup_{s\in [0,T]}\big| \widetilde X^{i,N}_s - X^{i,\widetilde g}_s\big|^2\Big]\\
    & \quad +\frac{C}{N}\sum_{i=1}^N \E\Big[\sup_{s\in [0,T]}\big|X^{i,\widetilde g}_s - X^{i,g}_s\big|^2\Big].
\end{split}
\end{equation}   
Notice that
\begin{equation*}
    \|\widetilde g_N-\widetilde g\|_2 = \|\widetilde g\circ (f_N,f_N) - \widetilde g\|_2\xrightarrow{N\rightarrow \infty}0.
\end{equation*}
Thus we can use Theorem \ref{thm:convergence-bounded-graphon} to claim that the second term in \eqref{eq:unbounded-est-1} converges to zero. 
The third term in \eqref{eq:unbounded-est-1} can be bounded thanks to Proposition \ref{prop:stability-of-g-average} as
\begin{align*}
    \frac{1}{N}\sum_{i=1}^N \E\Big[\sup_{s\in [0,T]}\big|X^{i,\widetilde g}_s - X^{i,g}_s\big|^2\Big] &\le  \frac{C}{N}\sum_{i=1}^N \int_I\left(\widetilde g(x_i,y)-g(x_i,y)\right)^2\lambda(dy) +  C \|g\circ (f_N,id)\|^2_{2} \|g-\widetilde g\|_2^2  \\
    & \leq C \|\tilde g\circ (f_N,\mathrm{id}) - \tilde g\|_2^2 +C\| \tilde g-g\|_2^2  + C\|g-g\circ (f_N,\mathrm{id})\|_2^2 \\
    & \quad + C\|g\circ (f_N,id) - g\|_2^2\|g- \tilde g\|_2^2 + C\|g\|_2^2\|g- \tilde g\|_2^2\\
    & \leq C \|\tilde g\circ (f_N,\mathrm{id}) - \tilde g\|_2^2 +C\eta  + C\|g-g\circ (f_N,\mathrm{id})\|_2^2 \\
    & \quad + C\|g\circ (f_N,id) - g\|_2^2\eta + C\|g\|_2^2\eta. \\
\end{align*}
We first take the limit as $N$ goes to infinity using continuity of $\widetilde g$ and Lemma \ref{lemma:cond-implies-convergence}.
Then, we let $\eta$ go to zero
to obtain that the third term in  \eqref{eq:unbounded-est-1} goes to zero.

We are only left with showing that the first term in equation \eqref{eq:unbounded-est-1} converges to zero. 
Define $m^{i,N}_{\X_u^{\tilde g_N}}$ and $\widetilde m^{i,N}_{\X_u^{\tilde g_N}}$ as in equation \eqref{eq:def-coupled-measure} (replacing $\zeta_{ij}^N$ for $\widetilde \zeta_{ij}^N$ for defining $m^{i,N}_{\X_u^{\widetilde g_N}}$) and 
$m^{i,N,\eta}_{\X^{\tilde g_N}_u}$ and $\widetilde m^{i,N,\eta}_{\X^{\widetilde g_N}_u}$ are the regularized measures defined as the convolution with the Gaussian distribution with variance $\eta^2$ (see equation \eqref{eq:def-regularized-measure-density}). It is straightforward to obtain the estimate:
    \begin{equation}
    \label{eq:unbounded-est-2}
        \begin{split}
            \frac{1}{N}\sum_{i=1}^N &\E\bigg[\sup_{s\in [0,T]}\big| X^{i,N}_s - \widetilde X^{i,N}_s\big|^2\bigg]
             \leq \frac{C}{N}\sum_{i=1}^N \int_0^T \E\Big[d^2_{\mathrm{BL}}(m^{i,N}_{\X_u}, \widetilde m^{i,N}_{\X_u})\Big]du\\
            & \leq  \frac{C}{N}\sum_{i=1}^N \int_0^T \E\Big[d^2_{\mathrm{BL}}(m^{i,N}_{\X_u}, m^{i,N}_{ \X_u^{\widetilde g_N}})\Big] + \E\Big[d^2_{\mathrm{BL}}( m^{i,N}_{ \X_u^{\widetilde g_N}}, m^{i,N,\eta}_{ \X_u^{\widetilde g_N}})\Big]  +\E\Big[d^2_{\mathrm{BL}}( m^{i,N,\eta}_{ \X_u^{\widetilde g_N}}, \widetilde m^{i,N,\eta}_{ \X_u^{\widetilde g_N}})\Big]\\
            &\quad + \E\Big[d^2_{\mathrm{BL}}(\widetilde m^{i,N,\eta}_{\X^{\widetilde g_N}_u} , \widetilde m^{i,N}_{\X^{\widetilde g_N}_u})\Big] + \E\Big[d^2_{\mathrm{BL}}( \widetilde m^{i,N}_{\X^{\widetilde g_N}_u} , \tilde m^{i,N}_{\X_u} )\Big] du \\
            & \leq  \frac{C}{N}\sum_{i=1}^N \int_0^T \E\Big[d^2_{\mathrm{BL}}(m^{i,N}_{\X_u}, m^{i,N}_{ \X_u^{\widetilde g_N}})\Big] +\E\Big[d^2_{\mathrm{BL}}( m^{i,N,\eta}_{ \X_u^{\widetilde g_N}}, \widetilde m^{i,N,\eta}_{ \X_u^{\tilde g_N}})\Big]\\
            &\quad + \E\Big[d^2_{\mathrm{BL}}( \widetilde m^{i,N}_{\X^{\widetilde g_N}_u} , \widetilde m^{i,N}_{\X_u} )\Big] du + C\sqrt{n}\eta \\
            & = \int_0^t \cT_1(u) + \cT_2(u)+ \cT_3(u)du + C\sqrt{n}\eta,
\end{split}
\end{equation}    
where  we used Lemma \ref{lemma:distance-regularised-measures-L1} to obtain bounds for $\E\Big[d^2_{\mathrm{BL}}( m^{i,N}_{ \X_u^{\tilde g_N}}, m^{i,N,\eta}_{ \X_u^{\tilde g_N}})\Big]$ and $\E\Big[d^2_{\mathrm{BL}}(\tilde m^{i,N,\eta}_{\X^{\tilde g_N}_u} , \tilde m^{i,N}_{\X^{\tilde g_N}_u})\Big]$.  

To bound $\cT_1$ we use the definition of bounded Lipschitz distance:
\begin{equation}
\label{eq:unbounded-est-2.1}
\begin{split}
    \cT_1 
    & \leq \frac{1}{N}\sum_{i=1}^N \E\left[\sup_{\phi \in 1\text{-Lip}_b} \left(\frac{1}{N}\sum_{j=1}^N \frac{\zeta_{ij}^N}{\beta_N} \phi(X^{j,N}_{u}) - \frac{\zeta_{ij}^N}{\beta_N} \phi(X^{j,\widetilde g_N}_u)\right)^2 \right]\\
    & \leq \frac{C}{N}\sum_{i=1}^N \E\left[\left(\frac{1}{N}\sum_{j=1}^N \left(\frac{\zeta_{ij}^N}{\beta_N} -\widetilde g_N(x_i,x_j)\right)\left|X^{j,N}_{u}- X^{j,\widetilde g_N}_u\right|\right)^2\right] \\
    & \quad + \frac{1}{N}\sum_{i=1}^N \E\left[\left(\frac{1}{N}\sum_{j=1}^N \widetilde g_N(x_i,x_j) \left|X^{j,N}_{u}- X^{j,\widetilde g_N}_u\right|\right)^2\right].
\end{split}
\end{equation}
To deal with the first term we introduce the notation: 
\begin{equation*}
    \Gamma_{ij}=\zeta_{ij}^N-\beta_N \widetilde g_N(x_i,x_j) \quad \text{ and } \quad \big|\Delta X^j\big| = \big| X^{j,N}_u-X^{j,\widetilde g_N}_u\big|.
\end{equation*}
Then we can expand the square on equation \eqref{eq:unbounded-est-2.1} and use Hölder's inequality on the second term to obtain:
\begin{equation}
\label{eq:unbounded-est-2.2}
\begin{split}
    \cT_1 & \leq \frac{C}{N}\sum_{i=1}^N \E\left[\frac{1}{\beta^2_N N^2}\sum_{j=1}^N \Gamma^2_{ij}\left|\Delta X^j\right|^2 + \frac{1}{\beta_N^2 N^2} \sum_{j\neq k=1}^N \Gamma_{ij}\Gamma_{ik}\left|\Delta X^j\right| \big|\Delta X^k\big|\right] \\
    & \quad + \frac{1}{N}\sum_{i=1}^N \frac{1}{N}\sum_{j=1}^N \widetilde g^2_N(x_i,x_j) \E\left[ \frac{1}{N}\sum_{j=1}^N \left|\Delta X^j\right| \right]\\ 
    & \leq \frac{C}{\beta^2_N N^3}\sum_{j=1}^N   \E\left[\sum_{i=1}^N \Gamma^2_{ij}\left|\Delta X^j\right|^2\right] + \frac{C}{\beta_N^2 N^3} \sum_{j\neq k=1}^N\E\left[ \sum_{i=1}^N \Gamma_{ij}\Gamma_{ik}\left|\Delta X^j\right| \big|\Delta X^k\big|\right] \\
    & \quad + \|g_N\|_2^2 \E\left[ \frac{1}{N}\sum_{j=1}^N\left|\Delta X^j\right|\right].\\ 
\end{split}
\end{equation}

Using Cauchy-Schwarz inequality, Lemma \ref{lemma:L2-boundness-solutions} and the fact that the expectation of $\Gamma_{ij}$ is zero to obtain:  
\begin{equation}
\begin{split}
    \cT_1 & \leq \frac{C}{\beta^2_N N^3}\sum_{j=1}^N \E\left[\left(\sum_{i=1}^N \Gamma^2_{ij}\right)^2\right]^{1/2} + \frac{C}{\beta_N^2 N^3} \sum_{j\neq k}^N\E\left[ \left(\sum_{i=1}^N \Gamma_{ij}\Gamma_{ik}\right)^2\right]^{1/2} + \|g_N\|_2^2 \E\left[ \frac{1}{N}\sum_{j=1}^N\left|\Delta X^j\right|\right]\\ 
    & \leq \frac{C}{\beta^2_N N^3}\sum_{j=1}^N \E\left[\sum_{i=1}^N \Gamma^4_{ij}\right]^{1/2} + \frac{C}{\beta^2_N N^3}\sum_{j=1}^N \E\left[\sum_{i\neq l=1}^N \Gamma^2_{ij}\Gamma^2_{il}\right]^{1/2} + \frac{C}{\beta_N^2 N^3} \sum_{j\neq k=1}^N\E\left[ \sum_{i=1}^N \Gamma^2_{ij}\Gamma^2_{ik}\right]^{1/2} \\
    & \quad + \|g_N\|_2^2 \E\left[ \frac{1}{N}\sum_{j=1}^N\left|\Delta X^j\right|\right].
\end{split}
\end{equation}

Finally, using te bound in equations \eqref{eq:bounds-centered-ber} we obtain:
\begin{equation}
\begin{split}
    \cT_1 & \leq \frac{C}{N^3\beta_N^2} \sum_{j=1}^N \left(\sum_{i=1}^N g_N(x_i,x_j)\beta_N\right)^{1/2} + \frac{C}{N^3\beta_N^2} \sum_{j=1}^N\left(\sum_{i\neq l=1}^N \beta_N^2 g_N(x_i,x_j)g_N(x_l,x_j)\right)^{1/2}\\
    & \quad + \frac{C}{N^3\beta_N^2} \sum_{j\neq k=1}^N\left(\sum_{i=1}^N \beta_N^2g_N(x_i,x_j)g_N(x_i,x_k)\right)^{1/2} +  \frac{C\|g_N\|_2^2 }{N}\sum_{j=1}^N \E\left[ \left|X^{j,N}_{u}- X^{j,\widetilde g_N}_u\right|^2\right]\\
    & \leq \frac{C}{(N\beta_N)^{3/2}} \int_I\left(\int_I g_N(y,x)\lambda(dy)\right)^{1/2}\lambda(dx) + \frac{C}{N\beta_N} \int_I \int_I g_N(y,x)\lambda(dy)\lambda(dx)\\
    & \quad + \frac{C}{N\beta_N} \int_I \int_I \left(\sum_{i=1}^N g_N(x_i,y)g_N(x_i,z)\right)^{1/2}\lambda(dy)\lambda(dz) + \frac{C\|\widetilde g_N\|_2}{N}\sum_{j=1}^N \E\left[ \left|X^{j,N}_{u}- X^{j,\widetilde g_N}_u\right|^2\right]\\
    & \leq \frac{C\|g_N\|_1^{1/2}}{(N\beta_N)^{3/2}} + \frac{C\|g_N\|_1}{N\beta_N} + \frac{C\|g_N\|_2}{N^{1/2}\beta_N} + \frac{C\|\widetilde g_N\|_2}{N}\sum_{j=1}^N \E\left[ \left|X^{j,N}_{u}- \widetilde X^{j,N}_u \right|^2\right] \\
    & \quad+ \frac{C\|\widetilde g_N\|_2}{N}\sum_{j=1}^N \E\left[ \left|\widetilde X^{j,N}_u - X^{j,\widetilde g_N}_u\right|^2\right],
\end{split}
\end{equation}
where we used the estimations leading to \eqref{eq:bound-zeta-variance} (using Lemma \ref{lemma:L2-boundness-solutions-g-unbounded} instead of Lemma \ref{lemma:L2-boundness-solutions}). The first three terms converge to zero as $N\rightarrow \infty$. The fourth one can be absorbed using Gronwall's inequality. Moreover, the last term also converges to zero as
\begin{equation*}
\begin{split}
     \frac{1}{N}\sum_{j=1}^N \E\left[ \left|\widetilde X^{j,N}_u - X^{j,\widetilde g_N}_u\right|^2\right]\leq & \frac{1}{N}\sum_{j=1}^N \E\left[ \left|\widetilde X^{j,N}_u - X^{j,\widetilde g}_u \right|^2 \right] + \frac{1}{N}\sum_{j=1}^N \E\left[\left|X^{j,\widetilde g}_u - X^{j,\widetilde g_N}_u\right|^2\right] \rightarrow 0.\\
\end{split}
\end{equation*}
Where the first term goes to zero by Theorem \ref{thm:convergence-bounded-graphon} and the second one by Proposition \ref{prop:stability-of-g-average}. 

The term $\cT_3$ converges to zero by the same estimates used for $\cT_1$. Hence, we are left with $\cT_2$.  
\begin{equation}
\label{eq:unbounded-T_2}
\begin{split}
   \cT_2 (u) & =  
   \frac{1}{N}\sum_{i=1}^N \E\left[d^2_{\mathrm{BL}}( m^{i,N,\eta}_{ \X_u^{\widetilde g_N}}, \widetilde m^{i,N,\eta}_{ \X_u^{\widetilde g_N}} )\right]\\
   &\leq  \frac{1}{N}\sum_{i=1}^N \E\left[\sup_{\phi \in 1\text{-Lip}_b}\left( \int_{\R^n}\phi(z) \frac{1}{N}\sum_{j=1}^N \frac{\zeta_{ij}^N}{\beta_N} \varphi_\eta(z-X_u^{j,\widetilde g_N}) -  \frac{\widetilde \zeta_{ij}^N}{\beta_N}\varphi_\eta(z- X_u^{j,\widetilde g_N}) dz\right)^2\right]\\
   &\leq  \frac{1}{N}\sum_{i=1}^N \E\left[\int_{\R^n}(1+|z|^{n+1})\left( \frac{1}{\beta_N N}\sum_{j=1}^N (\zeta_{ij}^N -  \widetilde \zeta_{ij}^N)\varphi_\eta(z- X_u^{j,\widetilde g_N})\right)^2 dz\right]\\
   &\leq  \frac{C}{N}\sum_{i=1}^N \E\left[\int_{\R^n}(1+|z|^{n+1})\left( \frac{1}{\beta_N N}\sum_{j=1}^N (\zeta_{ij}^N -  \beta_Ng_N(x_i,x_j))\varphi_\eta(z- X_u^{j,\widetilde g_N})\right)^2 dz\right]\\
   & \quad + \frac{C}{N}\sum_{i=1}^N \E\left[\int_{\R^n}(1+|z|^{n+1})\left( \frac{1}{N}\sum_{j=1}^N ( g_N(x_i,x_j)- \widetilde g_N(x_i,x_j))\varphi_\eta(z- X_u^{j,\widetilde g_N})\right)^2 dz\right]\\
   & \quad + \frac{C}{N}\sum_{i=1}^N \E\left[\int_{\R^n}(1+|z|^{n+1})\left( \frac{1}{\beta_N N}\sum_{j=1}^N ( \beta_N\widetilde g_N(x_i,x_j)-\widetilde \zeta_{ij}^N)\varphi_\eta(z- X_u^{j,\widetilde g_N})\right)^2 dz\right]\\
   & =  \cL_1+ \cL_2+ \cL_3.\\
\end{split}
\end{equation}

We start by bounding $\cL_1$. Recall that we can define the family $\{y_i\}_{\{1\leq i\leq N\}}$ such that $\{X^{y_i,\tilde g_N}\}_{\{1\leq i\leq N\}}$ are independent (see proof of Theorem \ref{thm:convergence-bounded-graphon-L1}). Hence, we obtain by independence:
\begin{equation}
\label{eq:unbounded-L1}
\begin{split}
    \cL_1 & \leq \frac{C}{N}\sum_{i=1}^N \E\left[\int_{\R^n}(1+|z|^{n+1})\left( \frac{1}{\beta_N N}\sum_{j=1}^N (\zeta_{ij}^N -  \beta_Ng_N(x_i,x_j))\varphi_\eta(z- X_u^{j,\widetilde g_N})\right)^2 dz\right] \\
    & \leq \frac{C}{N^3\beta^2_N}\sum_{i=1}^N \int_{\R^n}(1+|z|^{n+1})\sum_{j=1}^N \mathrm{Var}[\zeta_{ij}^N]\E\left[\varphi^2_\eta(z- X_u^{j,\widetilde g_N})\right] dz \\
    & \leq  \frac{C(\eta)}{N^3\beta_N}\sum_{i=1}^N \sum_{j=1}^N g_N(x_i,x_j) \\
    & \leq  \frac{C(\eta)}{N\beta_N}\int_I \int_I g_N(x,y)\lambda(dy)\lambda(dx) \\
    & \leq \frac{C(\eta)}{N\beta_N}\|g_N\|_1.
\end{split}
\end{equation}

Similarly, we get that the third term can be estimated by:
\begin{equation}
\label{eq:unbounded-L3}
    \cL_3 \leq \frac{C(\eta)}{N\beta_N} \|\widetilde g_N\|_1.
\end{equation}
Finally, we bound $\cL_2$ using Cauchy-Schwarz inequality we obtain: 
\begin{equation}
\label{eq:unbounded-L2}
\begin{split}
    \cL_2 & \leq \frac{1}{N}\sum_{i=1}^N \E\left[\int_{\R^n}(1+|z|^{n+1})\left( \frac{1}{N}\sum_{j=1}^N ( g_N(x_i,x_j)- \tilde g_N(x_i,x_j))^2\right) \left(\frac{1}{N}\sum_{j=1}^N\varphi^2_\eta(z- X_u^{j,\tilde g_N})\right) dz\right]\\
    & \leq  C(\eta) \|g_N-\tilde g_N\|^2_2 \\
    & \leq C(\eta) (\|g_N-g\|^2_2 + \|g-\tilde g\|^2_2 + \|\tilde g -\tilde g_N\|^2_2). \\
\end{split}
\end{equation}

Combining equations \eqref{eq:unbounded-L1}, \eqref{eq:unbounded-L2} and \eqref{eq:unbounded-L3} we obtain:
\begin{equation*}
    \cT_3(u) \leq \frac{C(\eta)}{N\beta_N}\|g_N\|_1+ \frac{C(\eta)}{N\beta_N}\|\widetilde g_N\|_1 + C(\eta) (\|g_N-g\|^2_2 + \|g-\tilde g\|^2_2 + \|\tilde g -\tilde g_N\|^2_2)\rightarrow 0. 
\end{equation*}
Hence, we conclude that:
\begin{equation*}
    \limsup_{N\rightarrow \infty } \frac{1}{N}\sum_{i=1}^N \E\bigg[\sup_{s\in [0,T]}\big| X^{i,N}_s - \widetilde X^{i,N}_s\big|^2\bigg] \leq C\sqrt{n}\eta.
\end{equation*}

Finally, we can take $\eta\rightarrow 0$ to conclude. 

\end{proof}

In the case where the adjacency matrix of the graph $\zeta^N$ is deterministic we can state Theorem \ref{thm:convergence-unbounded-graphon} in terms of the sharp condition $N\beta_N\rightarrow \infty$.

\begin{proof}[Proof of Theorem \ref{thm:convergence-unbounded-graphon-deterministic}]
    Notice that the only point in the proof of Theorem \ref{thm:convergence-unbounded-graphon} where we used the hypotesis that $N\beta_N\rightarrow \infty$ is on equation \eqref{eq:unbounded-est-2.1}. Using the fact that now $\zeta_{ij}^N = \beta_Ng_N(x_i,x_j)$ we can modify equation \eqref{eq:unbounded-est-2.1} as follows: 
    \begin{equation*}
    \begin{split}
        \cT_1 & \leq \frac{1}{N}\sum_{i=1}^N \E\left[\sup_{\phi \in 1\text{-Lip}_b} \left(\frac{1}{N}\sum_{j=1}^N g_N(x_i,x_j) \phi(X^{j,N}_{u}) - g_N(x_i,x_j)  \phi(X^{j,\widetilde g_N}_u)\right)^2 \right]\\
        & \leq \frac{C}{N}\sum_{i=1}^N \E\left[\left(\frac{1}{N}\sum_{j=1}^N g_N(x_i,x_j) \left|X^{j,N}_{u}- X^{j,\widetilde g_N}_u\right|\right)^2\right]\\
        & \leq \frac{C}{N}\sum_{i=1}^N \E\left[\left(\frac{1}{N}\sum_{j=1}^N g_N(x_i,x_j) \left|X^{j,N}_{u}- X^{j,\widetilde g_N}_u\right|\right)^2\right]\\
        & \leq \frac{C}{N}\sum_{i=1}^N \E\left[\frac{1}{N}\sum_{j=1}^N g^2_N(x_i,x_j) \frac{1}{N}\sum_{j=1}^N \left|X^{j,N}_{u}- X^{j,\widetilde g_N}_u\right|^2\right]\\
        & \leq C \|g_N\|_2^2,\\
    \end{split}
    \end{equation*}

    where we used Lemma \ref{lemma:L2-boundness-solutions-g-unbounded}. This finishes the claim. 
    
\end{proof}

\subsection{Graph-theoretic formulation of the main result}
\label{sec:graph-theoretic-formulation}

We conclude the paper by stating the main convergence result of the paper in abstract graph theoretic terms.
This is for the convenience of readers that prefer the graph-theoretic language.

Let $\cG := (V, E)$ be a finite, simple (and thus undirected, with not loop) graph with vertex set $V = \{x_1,\dots,x_N\}$ and edge set $E$.
In particular, the  cardinality of $V$ is
$$
    |V| = N
$$ 
and we assume without loss of generality that $V\subset I$ and that for each $i$, $\lambda((x_i, x_{i+1}]) = 1/N$.
Each player $x \in V$ interacts with the rest of the population (or its neighbors) through the graph--weighted empirical measure
\begin{equation*}
    m^{x,\cG}_{\X} := \frac{1}{\mathrm{deg}_\cG(x)}\sum_{(x,y)\in E}\delta_{X^y},
\end{equation*}
with $\X:= (X^y)_{y\in V}$, where and $\mathrm{deg}_\cG(x)$ is the degree\footnote{That is, the number of neighbors of $x$, i.e. $\mathrm{deg}_\cG(x) = |\{y \in V: (x,y)\in E\}|$.} of the vertex $x$ and the notation $y\sim x$ means that $(x,y)\in E$.
  Consider the interacting particle system
\begin{equation}
\label{eq:part.syst.graph}
    dX^{x,V}_t = b(t, X^{x, V}_t, m^{x,\cG}_{\X^V}) dt + \sigma(t, X^{x, V}_t, m^{x,\cG}_{\X^V})dW_t^x,\quad X^x_0 = \xi^x\quad x\in V,
\end{equation}
where we use the superscript $V$ to emphasize that $\X^V$ is a vector of $|V|$ coordinates depending on the vertices of the graph $(V, E)$.
To each graph $\cG$, we associate a graphon $g_{\cG}:I\times I\to I$.
This is the function defined as follows:
We enumerate $V$ as $V=\{x_1,\dots, x_{|V|}\}$ and let $(I_1,\dots, I_{|V|})$ be a partition of $I$ by measurable sets such that for each $i$, we have $\lambda(I_i) = 1/|V|$ and $x_i\in I_i$.
We put
\begin{equation*}
    g_{\cG}(x,y) := 
    \begin{cases}
        1 &\text{if } (x,y) \in I_i \times I_j \text{ such that } (x_i, x_j)\in E\\
        0&\mathrm{otherwise}.
    \end{cases}
\end{equation*}
In particular, $ g_\cG$ is a step graphon.
The following corollary easily follows from Theorem \ref{thm:convergence-unbounded-graphon-deterministic}.
\begin{corollary}
\label{Cor:C-regular}
    Assume that $\cG_N = (V_N,E_N)$ is a sequence of graphs such that for every $x,y\in I$ it holds $\mathrm{deg_{\cG_N}}(x) = \mathrm{deg_{\cG_N}}(y)\rightarrow \infty$ as $N \rightarrow\infty$.
    Let $g_{\cG_N}$ be the graphon associated with $\cG_N$. 
    If the functions $g_N:= g_{\cG_N}/\|\cG_N\|_1$ converge in $\|\cdot\|_4$-norm to a graphon $g$ satisfying Condition \ref{cond:almost-continuity} and such that $\|g\|_4< \infty$ and the function $g\circ (f_N,\mathrm{id})$ is dominated by a square integrable function, then it holds
    \begin{equation*}
         \frac{1}{N}\sum_{i=1}^N\E\Big[ \sup_{t\in [0,T]}\big|X^{x_i,V_N}_{t}-X^{x_i,g}_t\big|^2\Big] \xrightarrow{N\rightarrow \infty} 0. 
    \end{equation*} 
    where $(X^{x,g})_{x\in I}$ solves the graphon SDE \eqref{eq:graphon-SDE}.
\end{corollary}

\begin{proof}
    Recall that the $L^p$-norm of the (simple) graph $\cG$ is defined as
    \begin{equation*}
        \|\cG\|_p^p = \sum_{i,j \in V} \frac{|g_\cG(x_i,x_j)|^p}{N^2}.
    \end{equation*}
    For any vector $\X^V$, it holds that
    \begin{align*}
        m^{x,\cG_N}_{\X^{V_N}}
        &= \frac{\|\cG_N\|_1}{\mathrm{deg}_{\cG_N}(x)}\sum_{y\in V_N}\frac{1}{\|\cG_N\|_1}g_{\cG_N}(x,y)\delta_{X^{x,V_N}}.
    \end{align*}
    Because of the constant degree assumption we have $\mathrm{deg}_{\cG_N}(x)
    = |V_N|\|\cG_N\|_1$, so that
    \begin{align*}
        m^{x,\cG_N}_{\X^{V_N}} &= \frac{1}{|V_N|}\sum_{y\in V_N}\frac{1}{\|\cG_N\|_1} g_{\cG_N}(x,y)\delta_{X^{x,V_N}}\\
                      &= \frac{1}{|V_N|\beta_N}\sum_{y\in V_N} g_{\cG_N}(x,y)\delta_{X^{x,V_N}},
    \end{align*}
    with $\beta_{N} := \|\cG_N\|_1$.
    Now, observe that putting $g_N(x,y):= \frac{g_{\cG_N}(x,y)}{\beta_N}$ we obtain a step graphon $g_N:I\times I\to [0,\infty)$ and we have
    \begin{align*}
        \zeta_{ij}^N := \mathrm{Bernoulli}(\beta_N g_N(x_i,x_j)) = \mathrm{Bernoulli}( g_{\cG_N}(x_i,x_j)) = g_{\cG_N}(x_i,x_j).
    \end{align*}
    In other words, recalling that $|V_N| = N$, we have in the notation of the previous sections
   
    $$
        m^{x, \cG_N}_{\X^V_N} = \frac{1}{N}\sum_{j=1}^N\frac{\zeta^N_{ij}}{\beta_N}\delta_{X^{x_j,N}}.
    $$
    That is, with this choice of $\zeta^{ij}_N$ the particle systems \eqref{eq:part.syst.graph} and \eqref{eq:intro.part.system} are the same.
    We already know (by construction) that $N\beta_N = \mathrm{deg}_{\cG_N}(x)\rightarrow \infty$ as $N\rightarrow \infty$ and $\beta_Ng_N(x_i,x_j)= g_{\cG_N}(x_i,x_j)\le 1$ for every $N\ge 1$.
    Because $\|g_{N}-g\|_4\to 0$ and $\|g\|_4< \infty$ with $g$ satisfying Condition \ref{cond:almost-continuity} and with $g\circ(f_N,\mathrm{id})$ dominated by a square integrable function, the result follows from Theorem \ref{thm:convergence-unbounded-graphon-deterministic}.
\end{proof}

\bibliographystyle{plainnat}
\bibliography{dispmonotone}

\end{document}